\documentclass{amsart}

\usepackage{amsmath}
\usepackage{amssymb}
\usepackage{graphicx}
\usepackage{color}
\usepackage[all]{xy}
\usepackage{amsthm}

\usepackage[colorlinks,linkcolor=blue,citecolor=blue,pdfstartview=FitH]{hyperref}

\theoremstyle{plain}

\newtheorem{thm}{Theorem}

\newtheorem{prop}{Proposition}[section]
\newtheorem{lemma}[prop]{Lemma}
\newtheorem{cor}[prop]{Corollary}

\newtheorem*{thm1}{Theorem}

\newtheorem*{fact}{Fact}

\theoremstyle{remark}
\newtheorem*{claim}{Claim}
\newtheorem*{remark}{Remark}

\theoremstyle{definition}

\newtheorem{dfn}[prop]{Definition}
\newtheorem{dfns}[prop]{Definitions}

\newcommand{\co}{\colon\thinspace}

\newcommand{\inv}[1]{#1^{-1}}

\newcommand{\bound}{\partial}

\newcommand{\AND}{\qquad \mathrm{and} \qquad}

\newcommand{\C}{\mathbb{C}}

\newcommand{\tr}{\mathrm{Tr}}

\renewcommand{\c}{{\sf c}}

\newcommand{\f}{{\sf f}}

\newcommand{\g}{{\sf g}}

\newcommand{\h}{{\sf h}}

\renewcommand{\k}{{\sf k}}

\newcommand{\p}{{\sf p}}

\renewcommand{\r}{{\sf r}}

\newcommand{\s}{{\sf s}}

\renewcommand{\t}{{\sf t}}

\newcommand{\meetcute}{meet cute }

\newcommand\calb{\mathcal{B}}
\newcommand\calf{\mathcal{F}}
\newcommand\calh{\mathcal{H}}
\newcommand\calk{\mathcal{K}}
\newcommand\calp{\mathcal{P}}

\newcommand\cals{\mathcal{S}}
\newcommand\calt{\mathcal{T}}
\newcommand\calv{\mathcal{V}}

\newcommand\sfa{{\sf a}}
\newcommand\sfb{{\sf b}}
\newcommand\sfc{{\sf c}}
\newcommand\sfm{{\sf m}}
\newcommand\sfq{{\sf q}}
\newcommand\sfh{{\sf h}}
\newcommand\sfn{{\sf n}}

\newcommand\sfu{{\sf u}}
\newcommand\sfw{{\sf w}}
\newcommand\sfx{{\sf x}}

\newcommand\bv{\mathbf{v}}
\newcommand\bw{\mathbf{w}}

\definecolor{gray1}{gray}{0.6}

\begin{document}

\title[Commensurability of link complements]{Algebraic invariants, mutation, and commensurability of link complements}

\author{Eric Chesebro}
\address{Department of Mathematical Sciences, University of Montana} 
\email{Eric.Chesebro@mso.umt.edu} 

\author{Jason DeBlois}
\address{Department of Mathematics, University of Pittsburgh} \email{jdeblois@pitt.edu}
\thanks{Second author partially supported by NSF grant DMS-1240329}

\begin{abstract}  We construct a family of hyperbolic link complements by gluing tangles along totally geodesic four-punctured spheres, then investigate the commensurability relation among its members.  Those with different volume are incommensurable, distinguished by their scissors congruence classes.  Mutation produces arbitrarily large finite subfamilies of nonisometric manifolds with the same volume and scissors congruence 
class.  Depending on the choice of mutation, these manifolds may be commensurable or incommensurable, distinguished in the latter case by cusp parameters.  All have trace field $\mathbb{Q}(i,\sqrt{2})$; some have integral traces while others do not.  \end{abstract}

\maketitle

\section{Introduction}

Manifolds are \emph{commensurable} if they have a common cover, of finite degree over each.  W.P. Thurston first studied the commensurability relation among hyperbolic knot and link complements in $S^3$, describing commensurable and incommensurable examples in Chapter 6 of his notes \cite{Th}.  The families of chain link complements \cite{NeR2}, two-bridge knot complements \cite{Reid_Walsh}, and certain pretzel knot complements \cite{MacMat} have since been further explored.  Here we construct another infinite family of hyperbolic link complements and explore the commensurability relation among its members.

We compute the following invariants on members of our family.  For $\Gamma < \mathrm{PSL}_2(\mathbb{C})$ the \textit{trace field} of $M = \mathbb{H}^3/\Gamma$ is the smallest field containing the traces of elements of $\Gamma$.  If each such trace is an algebraic integer we say $M$ has \textit{integral traces}.  The \textit{cusp parameters} of $M$, used in \cite{Th} and \cite{NeR2}, are algebraic invariants of the Euclidean structures on horospherical cross sections of the cusps of $M$.  The \textit{Bloch invariant} \cite{NeYang} is determined by a polyhedral decomposition.  

\begin{thm} \label{omnibus1}  For each $n \in \mathbb{N}$, there is a link $L_n \subset S^3$ such that $M_n = S^3 - L_n$ is hyperbolic with trace field $\mathbb{Q}(i,\sqrt{2})$ and integral traces.  If $m \neq n$ then $M_m$ and $M_n$ are incommensurable, distinguished by their  Bloch invariants and cusp parameters.  
\end{thm}

Having integral traces is commensurability-invariant \cite[\S 5.2]{MaR}, and the trace field is a commensurability invariant of link complements \cite[Cor.~4.2.2]{MaR}.  Commensurable manifolds have $\mathbb{Q}$-dependent Bloch invariants and $\mathrm{PGL}_2(\mathbb{Q})$-dependent cusp parameters, but the $M_n$ have neither (see Proposition \ref{M_n incomm} and Lemma \ref{PGL action}).  

Figure \ref{linkspic} depicts $L_2$.  The grey lines there indicate the presence of $2$-spheres that each meet $L_2$ in $4$ points, separating it left-to-right into a tangle $S$ in the three-ball $B^3$, two copies of a tangle $T \subset S^2 \times I$, and the mirror image $\overline{S}$ of $S$.
For arbitrary $n \in \mathbb{N}$, the link $L_n$ is constructed analogously, using $S$, $\overline{S}$, and $n$ copies of $T$.  We number the corresponding $2$-spheres for $L_n$ as $S^{(i)}$ for $0\leq i\leq n$, so that $S^{(0)}$ bounds $S$, $S^{(n)}$ bounds $\overline{S}$, and $S^{(i)}$ bounds a copy of $T$ with $S^{(i-1)}$ for $0<i\leq n$.

\begin{figure}[ht]

\setlength{\unitlength}{.1in}

\begin{picture}(40,10)
\put(0,0) {\includegraphics[width=4in]{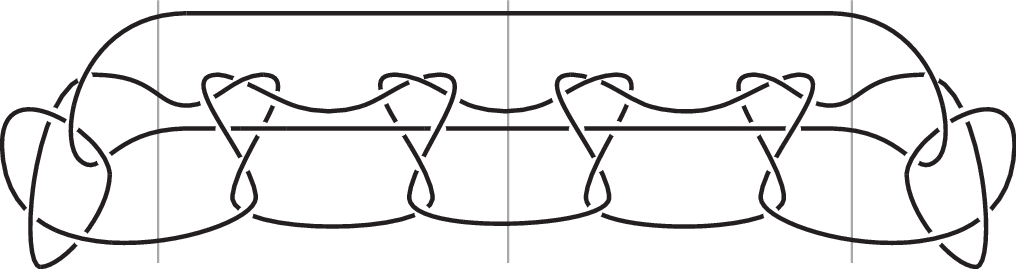}}
\put(12.5,0.25) {$T_1$}
\put(36,9.5) {$T_2$}
\end{picture}

\caption{The link $L_2$} \label{linkspic}

\end{figure}

We also describe the commensurability relation among the complements of links related to the $L_n$ by \textit{mutation} along the $S^{(i)}$: cutting along $S^{(i)}$ and re-gluing by an order-two mapping class that preserves $S^{(i)}\cap L_n$ and acts on it as an even permutation.  With $L_n$ projected as in Figure \ref{linkspic}, for each $i$ we mark the points of $S^{(i)} \cap L_n$ by $2$, $3$, $4$, and $1$, reading top to bottom, and refer to a mutation homeomorphism of $S^{(i)}$ by its permutation representation.

Below, for $n \in \mathbb{N}$ and $I \in \{0,1,2\}^{n+1}$ let $L_I$ be the link  obtained from $L_n$ by the mutation $(13)(24)$ (respectively, $(12)(34)$) along $S^{(i)}$, for each $i$ such that the $i$th entry of $I$ is $1$ (respectively, $2$).  Write $M_I = S^3-L_I$ for each such $I$.

\begin{thm}\label{omnibus1.5}  For $n \in \mathbb{N}$ and $I =(a_0,\hdots,a_n)\in\{0,1\}^{n+1}$, $M_I$ is commensurable to $M_n$.  For $J = (b_0,\hdots,b_n)\in\{0,1\}^{n+1}$, $M_J$ is isometric to $M_I$ if and only if either $b_i = a_i$ for each $i\in\{1,\hdots,n-1\}$ or $b_i = a_{n-i}$ for each such $i$.\end{thm}

We will show in a future paper that Theorem \ref{omnibus1.5} reflects the fact that $M_n$ has a \textit{hidden symmetry} (see eg.~\cite{NeR2}) arising from a \textit{hidden extension} of the mutation $(13)(24)$, an extension of a lift of $(13)(24)$ over a finite cover of $(S^2\times I) - T$.

Here we prove Theorem \ref{omnibus1.5} more directly, identifying an orbifold $O_n$ jointly covered by $M_n$ and the $M_I$, see Proposition \ref{commensurable mutants}.  The key advantage of this approach is that we can also prove the isometry classification above (see Proposition \ref{isometric mutants}) using the fact that $O_n$ is minimal in the commensurability class of $M_n$ (Corollary \ref{mutants' commensurator}).

Corollary \ref{mutants' commensurator} is proved following Goodman--Heard--Hodgson \cite{GHH}.  The key step, for each $n$, is to construct a tiling $\calt_n$ of $\mathbb{H}^3$ by convex polyhedra that is \textit{canonical} in the sense of \cite[\S 2]{GHH}. See Theorem \ref{tiling}.  This is of independent interest, as there are few infinite families for which canonical tilings have been identified.

The mutation $(12)(34)$ has a very different effect than $(13)(24)$.

\begin{thm} \label{omnibus2} For $n \in \mathbb{N}$, let $\mathcal{L}_n = \{L_I \,|\, I \in \{0,2\}^{n+1}\}$.  Then:
\begin{enumerate}
\item  For each $I \in \{0,2\}^{n+1} - \{(0,\hdots,0)\}$,  $M_I$ has the same volume, Bloch invariant, and trace field as $M_n$, but has a nonintegral trace.
\item  There is a subfamily of $\mathcal{L}_n$ with at least $n/2$ mutually incommensurable members, distinguished by their cusp parameters.
\item  There is a subfamily of $\mathcal{L}_n$ with $n$ members which all share cusp parameters.  \end{enumerate}
\end{thm}

\noindent {\it Remarks.}\ \ 
1.~Mutation along $4$-punctured spheres preserves hyperbolic volume \cite{Ru}, the trace field \cite{NeR1}, and the Bloch invariant \cite[Theorem 2.13]{Neumann}.  While unaware of the Bloch invariant reference we proved our case directly.  It is Proposition \ref{mutators Bloch}.  \smallskip \\
  2.~$L_n = L_{(0,\hdots,0)}$, which accounts for the gap in statement (1) of the theorem.\smallskip  \\
  3.~Corollaries \ref{different mods} and \ref{same moduli} describe the subfamilies from (2) and (3) above.  We do not know the commensurability relation among members of the latter subfamily.  \bigskip

Theorems \ref{omnibus1.5} and \ref{omnibus2} comprise the first study (to our knowledge) of commensurability among an infinite family of link complements related by mutation.  Mutants have a longstanding reputation for being difficult to distinguish, although the algorithm of \cite{GHH} can now be used to check particular examples.  (For instance, the complement of the ``Kinoshita--Terasaka knot,'' 11n42 in the knot tables, is incommensurable with that of its mutant, the ``Conway knot'' 11n34.)

Theorem \ref{omnibus1.5} further gives some evidence counter to the following conjecture of Reid--Walsh \cite{Reid_Walsh}: the commensurability class of a hyperbolic knot complement in $S^3$ contains at most two others.  This implies in particular that any hyperbolic knot complement is incommensurable with all but two of its (non-isometric) mutants.

We now outline the remainder of the paper.  We name the tangle complements $M_S \doteq B^3 - S$ and $M_T \doteq S^2 \times I - T$, and note that $M_T$ is the double of $M_{T_0} \doteq M_T \cap (S^2 \times [0,1/2])$ across a single boundary component.  Section \ref{sec:tangles} describes hyperbolic structures with totally geodesic boundary on $M_S$ and $M_{T_0}$ as identification spaces of the regular ideal octahedron and the right-angled ideal cuboctahedron, respectively.

The totally geodesic boundary $\partial M_S$ is isometric to the component of $\partial M_{T_0}$ contained in $\partial M_T$, and the reflective symmetry of $M_T$ ensures that its totally geodesic boundary components are orientation-reversing isometric.  In forming $M_n$ we glue $\partial B^3 - S$ to $S^2 \times \{0\} - T$ by a map isotopic to an isometry, so that the separating four-punctured spheres $F^{(i)} = S^{(i)} - T$ are totally geodesic in $M_n$ for $0 \leq i \leq n$.  Section \ref{sec:combo-nation!} describes this assembly.

Because the $F^{(i)}$ are totally geodesic, each copy of $M_S$ and $M_T$ in $M_n$ inherits its structure with totally geodesic boundary from the ambient hyperbolic structure.  This in turn makes it possible to compute the commensurability invariants of Theorem \ref{omnibus1} on the $M_n$.  We carry this out in Section \ref{sec:invariants}.  Few other link complements are known to contain a surface that is totally geodesic without some topological constraint forcing it so; see eg.~\cite{MaR2} and \cite{AR2}.  For related results see \cite{MeR}, \cite{Le}, \cite{Adams_etal_1}, \cite{Adams_etal_2}.

Our method of construction owes a debt to one that Adams \cite{Adams1} and Neumann--Reid used to produce families of hyperbolic $3$-manifolds, gluing together manifolds with $3$-punctured sphere boundary.  (However unlike the $4$-punctured sphere, a $3$-punctured sphere is totally geodesic in any hyperbolic $3$-manifold that contains it \cite[Theorem 3.1]{Adams1}.)  The work of Neumann--Reid can be used to show that for each imaginary biquadratic extension $k$ of $\mathbb{Q}$, there are infinitely many commensurability classes of hyperbolic $3$-manifolds with trace field $k$ (cf. \cite[\S 5.6]{MaR}).

In every hyperbolic $4$-punctured sphere, each mutation determines a homeomorphism properly isotopic to an isometry \cite{Ru}.  In Section \ref{sec:mutants} we describe the isometries determined by $(13)(24)$ and $(12)(34)$ and the hyperbolic structures on mutants of the $M_n$.  We prove Theorem \ref{omnibus1.5} in Section \ref{sec:commensurablo!} and Theorem \ref{omnibus2} in Section \ref{sec:in-commensurablo!}.

\subsection*{Acknowledgements}
\color{black}

The authors thank Ian Agol, Richard Kent, Chris Leininger, Peter Shalen, and Christian Zickert for helpful conversations, Joe Masters for suggesting the cusp parameter, and Dick Canary for helping us with Lemma \ref{convexcore}.  A referee on an earlier version of this paper pointed us to the Bloch invariant and motivated several major changes in this paper.  We also appreciate the thoughtful editorial comments from a second referee.  We want to especially thank Alan Reid for suggesting these questions to us and for many helpful conversations and suggestions.  The second author is grateful to the Clay Mathematics Institute for support during part of this project.  The authors also thank the University of Montana's Faculty Development Committee for their support.

\section{A pair of tangles}  \label{sec:tangles}

This section is devoted to describing hyperbolic structures with totally geodesic boundary on the complements of the tangles $S$, in $B^3$, and $T_0$, in $S^2 \times I$, depicted in Figure \ref{SandT_0}.  For a manifold $M$ with boundary, we refer by a \textit{tangle} in $M$ to a pair $(M,T)$, where $T$ is the  image of a disjoint union of circles and closed intervals, embedded in $M$ by a map taking each circle into the interior of $M$ and restricting on each interval to a proper embedding.

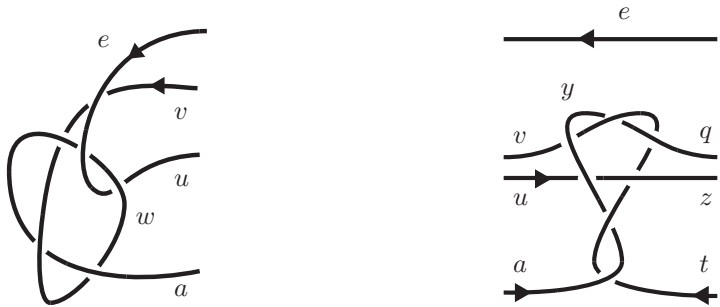
\begin{figure}

\setlength{\unitlength}{.1in}

\begin{picture}(40,15)
\put(0,0) {\includegraphics[height= 1.45in]{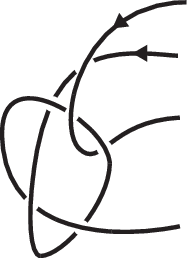}}
\put(4.75,13.5){$e$}
\put(6.75,4.25){$w$}
\put(8.75, .5){$a$}
\put(8.75, 9.75){$v$}
\put(8.75, 6.25){$u$}

\put(26,0) {\includegraphics[height= 1.45in]{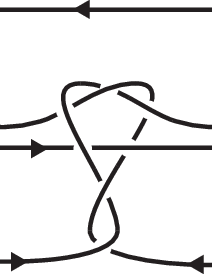}}
\put(32,15){$e$}
\put(29,11){$y$}
\put(36.25, 1.75){$t$}
\put(36.25, 8.75){$q$}
\put(36.25, 5.25){$z$}
\put(26.5, 1.75){$a$}
\put(26.5, 8.5){$v$}
\put(26.5, 5.25){$u$}
\end{picture}

\caption{Tangles $S$ and $T_0$, labeled with Wirtinger generators}  \label{SandT_0}

\end{figure}

We will prove there is a homeomorphism taking $B^3-S$ to a hyperbolic manifold with totally geodesic boundary which is an identification space of an ideal octahedron by pairing certain  faces.  This was previously known, and it follows from results in \cite{ZP} upon taking a geometric limit, but we do not know a reference for a direct proof.  We also prove there is a homeomorphism taking $S^2 \times I - T_0$ to a certain identification space of the right-angled ideal cuboctahedron.  As far as we are aware, this description was not previously known.

We prove existence of homeomorphisms using faithful representations, from the fundamental groups of tangle complements onto Kleinian groups generated by face pairings.  Our main tools drawing connections between the geometric, algebraic, and topological objects involved are Lemma \ref{combinatorics n cores}, which relates a hyperbolic $3$-manifold with totally geodesic boundary produced by pairing some faces of a right-angled polyhedron to the Kleinian group generated by the face pairing isometries, and Lemma \ref{convexcore}, which describes a homeomorphism between a pared manifold and the convex core of a Kleinian group to which its fundamental group represents.

In the remainder of the paper, we will let $\mathbb{H}^3 = \{(z,t)\,|\, z \in \mathbb{C}, t \in (0,\infty) \}$, the upper half space model of hyperbolic space, equipped with the complete Riemannian metric of constant sectional curvature $-1$.  In this model, the group of orientation--preserving isometries, $\mathrm{PSL}_2(\mathbb{C})$, acts by extending its action by M\"obius transformations on the \textit{ideal boundary} or \textit{sphere at infinity} $\mathbb{C} \cup \{\infty\}$.

The \textit{horosphere of height $t$ centered at $\infty$} is $\mathbb{C} \times \{t\} \subset \mathbb{H}^3$.  This inherits the Euclidean metric, scaled by $1/t$, from the ambient hyperbolic metric.  For $v \in \mathbb{C} \times \{0\}$, a \textit{horosphere centered at $v$} is a Euclidean sphere in $\mathbb{C} \times \mathbb{R}$ centered at a point in $\mathbb{H}^3$ and tangent to $\mathbb{C} \times \{0\}$ at $(v,0)$.  It is a standard fact that isometries of $\mathbb{H}^3$ take horospheres to horospheres.

A \textit{hyperplane} of $\mathbb{H}^3$ is a totally geodesic subspace of the form $\ell \times \mathbb{R}^+$ for a line $\ell \subset \mathbb{C}$, or the intersection with $\mathbb{H}^3$ of a Euclidean sphere centered at a point in $\mathbb{C} \times \{0\}$.  A \textit{half space} is the closure of a component of the complement in $\mathbb{H}^3$ of a hyperplane, and a \textit{polyhedron} is the nonempty intersection of a collection of half-spaces with the property that the corresponding collection of defining hyperplanes is locally finite.  A \textit{face} of a polyhedron is its intersection with one of its defining hyperplanes.  A polyhedron is \textit{right-angled} if its defining hyperplanes meet at right angles (if at all) and \textit{ideal} if any point at which more than two of its defining hyperplanes meet is on the sphere at infinity.  Such points are \textit{ideal vertices}.

We say a polyhedron $\calp \subset \mathbb{H}^3$ is \textit{checkered} if its set of faces is partitioned into sets $\cals_i$ and $\cals_e$ of  \textit{internal} and \textit{external} faces, respectively, so that each $f \in \cals_i$ intersects only faces in $\cals_e$ and vice--versa.  For a face $f$ of a checkered, right-angled ideal 
polyhedron $\calp$, let $\calh_f$ be the geodesic hyperplane in $\mathbb{H}^3$ containing $f$ and let  $U_f$  be the half-space bounded by $\calh_f$ that contains $\calp$.  Let the \textit{expansion} of $\calp$ be
$$  E(\calp) = \bigcap_{f \in \cals_i}  U_f.  $$ 
The expansion $E(\calp)$ is a polyhedron of infinite volume that contains $\calp$, and the components of the frontier of $\calp$ in $E(\calp)$ are the external faces of $\calp$.

An \textit{internal face pairing} for a checkered polyhedron $\calp\subset\calh^3$ is a collection $\{ \phi_f\, |\, f \in \cals_i\}$ of isometries, such that for each $f \in \cals_i$ there exists $f'\in\cals_i$ with $\phi_f(f) = f'$, $\phi_f(\calp) \cap \calp = f'$, and $\phi_{f'} = \phi_f^{-1}$.  It is \textit{proper} if $f'\neq f$ for all $f\in \cals_i$.    A proper internal face pairing determines a \textit{proper $\mathrm{Isom}(\mathbb{H}^3)$-side-pairing} of the expansion $E(\calp)$, in the sense of \cite[\S 10.1]{Ra}.  (In \cite{Ra}, \textit{faces} are called \textit{sides}.) 

Given a proper internal face pairing $\{\phi_f\}$ of a checkered polyhedron $\calp$, Theorem 10.1.2 of \cite{Ra} implies the identification space $E(\calp)/\{\phi_f\}$, determined by setting $x\sim\phi_f(x)$ for all $f\in\cals_i$ and $x\in f$, is a hyperbolic manifold.  The inclusion $\calp\hookrightarrow E(\calp)$ induces an inclusion from $M_{\calp}\doteq \calp/\{\phi_f\}$ to $E(\calp)/\{\phi_f\}$.  For each edge $e$ of each $g \in \cals_e$, there is a unique $f \in \cals_i$ such that $e \subset f \cap g$.  Since $f' = \phi_f(f)$ intersects a unique $g'\in\cals_i$ along $\phi_f(e)$, the internal face pairing for $\calp$ determines an edge pairing for the disjoint union of external faces of $\calp$.  Thus $M_{\calp}\doteq\calp/\{\phi_f\}$ is an isometrically embedded submanifold of $E(\calp)/\{\phi_f\}$, where $\partial M_{\calp}$ is the quotient of the disjoint union of the external faces by the edge pairing induced by $\{\phi_f\}$.  

Given an edge pair $\{ e, e'\}$ for $\bound M_\calp$, the total angle around this edge in $M_\calp$ is the sum of the dihedral angles for
$e$ and $e'$ in $\calp$.  Therefore,  if $\calp$ is right-angled, $\partial M_{\calp}$ is totally geodesic.  

If $\Gamma$ is a Kleinian group, we refer to the convex core of $\mathbb{H}^3/\Gamma$ as $C(\Gamma)$.  This is the convex submanifold of $\mathbb{H}^3/\Gamma$, minimal with respect to inclusion, with the property that the inclusion-induced homomorphism $\pi_1 C(\Gamma) \to \mathbb{H}^3/\Gamma$ is surjective.  (See \cite{Mo} for background on Kleinian groups.  The beginning of \S 6 there covers convex cores.)

\begin{lemma} \label{combinatorics n cores}  Let $\calp \subset \mathbb{H}^3$ be a finite-sided, checkered right-angled ideal polyhedron, with a proper internal face pairing $\{\phi_f\, |\, f \in \cals_i \}$.  Then $\Gamma \doteq \langle \phi_f \,|\, f \in \cals_i \rangle$ is a free Kleinian group, and the inclusion $\calp \hookrightarrow \mathbb{H}^3$ induces an isometry $p \co M_{\calp} \to C(\Gamma)$.  If $\calh$ is the hyperplane containing $g\in\cals_e$ then $\calh\to\mathbb{H}^3$ induces an isometric embedding of $\calh/\mathrm{Stab}_{\Gamma}(\calh)$ to the component of $\partial C(\Gamma)$ containing $p(g)$.\end{lemma}

\begin{proof}   We will continue to use some terminology and results from \cite{Ra}.  With these hypotheses the inclusion $\calp \to E(\calp)$ induces an isometric embedding $M_{\calp} \to E(\calp)/\{\phi_f\}$, and $\partial M_{\calp}$ is totally geodesic.  If $E(\calp)/\{\phi_f\}$ is complete as a hyperbolic $3$-manifold, then by Poincar\'e's polyhedron theorem (see eg. \cite[Theorem 11.2.2]{Ra}), $\Gamma = \langle \phi_f\,|\,f \in \cals_i \rangle$ is discrete and $E(\calp)$ is a fundamental domain for $\Gamma$.

By \cite[Theorem 11.1.6]{Ra}, to show completeness it suffices to check that the link of any cusp is a complete Euclidean surface.  Let $\lfloor v \rfloor = \{v_0,v_1,\hdots,v_{n-1}\}$ be an equivalence class of ideal vertices of $\calp$ under the relation generated by $x \sim \phi_f(x)$, $f \in \cals_i$, enumerated so that for each $j$ there exists $f_j \in \cals_i$ with $\phi_{f_j}(v_j) = v_{j+1}$ (taken modulo $n$).  In particular, $v_j$ is an ideal vertex of $f_j$ and also of $f'_j \doteq \phi_{j-1}(f_{j-1})$.  

For each $j$, let $\calb_j$ be a horosphere centered at $v_j$, chosen small enough that $\calb_j \cap \calb_{j'} = \emptyset$ for $j \neq j'$.  Since $\calp$ is right-angled, $\calb_j \cap \calp$ is a Euclidean rectangle for each $j$.  We may assume, by renumbering if necessary, that $\calb_0 \cap f_0$ has shortest length of all the arcs $\calb_j \cap f_j$.  Then since $\phi_0(\calb_0) \cap f'_1$ is parallel to $\phi_0(\calb_0) \cap f_1$ in $\phi_0(\calb_0) \cap \calp$, they have the same length: that of $\calb_0 \cap f_0$.  Since this is less than the length of $\calb_1 \cap f_1$, we have $\phi_0(\calb_0) \subset \calb_1$.

We may replace $\calb_1$ by $\phi_0(\calb_0)$, then replace $\calb_2$ with $\phi_1(\calb_1)$ and so on, yielding a new collection of horospheres which are pairwise disjoint and have the additional property that they are interchanged by the face pairings of $\calp$.  Equivalence classes of ideal vertices of $E(\calp)$ are the same as those of $\calp$; thus this collection satisfies the hypotheses of \cite[Theorem 11.1.4]{Ra}, and the link of $\lfloor v \rfloor$ is complete.  It follows that $E(\calp)/\{\phi_f\}$ is a complete hyperbolic $3$-manifold.

Now by the polyhedron theorem, $\Gamma$ is discrete and $E(\calp)$ is a fundamental domain for $\Gamma$.  It follows from a ping-pong argument that $\Gamma$ is free, since the fact that $\calp$ is right-angled implies that the hyperplanes containing its internal faces are mutually disjoint.  The inclusion $E(\calp) \to \mathbb{H}^3$ induces an isometry $E(\calp)/\{\phi_f\} \to \mathbb{H}^3/\Gamma$, so the inclusion $\calp \to \mathbb{H}^3$ induces an isometric embedding $p\co M_{\calp} \to \mathbb{H}^3/\Gamma$.

That $p(M_{\calp})\subseteq C(\Gamma)$ will follow from the fact that $\calp$ is contained in the convex hull of the limit set $\mathrm{Hull}(\Gamma)$ of $\Gamma$, since this is well known to be the universal cover of $C(\Gamma)$.  Fixed points of parabolic elements of $\Gamma$ lie in $\mathrm{Hull}(\Gamma)$, so since $\calp$ is the convex hull of its ideal vertices we show that it is in $\mathrm{Hull}(\Gamma)$ by observing that each such vertex is a parabolic fixed point of $\Gamma$.  Indeed, if $\{v_0,v_1,\hdots,v_{n-1}\}$ is an equivalence class of ideal vertices enumerated as we described above, then $v_0$ is fixed by $\phi_{f_{n-1}}\circ\hdots\circ\phi_{f_1}\circ\phi_{f_0}\in\Gamma$.

Since $p(M_{\calp})$  has totally geodesic boundary it is convex (cf.~\cite[Corollary I.1.3.7]{CEG}).  Thus if $p(M_{\calp})$ carries $\pi_1(\mathbb{H}^3/\Gamma)$ then $C(\Gamma)\subseteq p(M_{\calp})$.  To show this we use the nearest-point retraction $r\co E(\calp)\to\calp$ to produce a homeomorphism $M_{\calp}\cup_{\partial M_{\calp}}(\partial M_{\calp}\times[0,\infty))\to\mathbb{H}^3/\Gamma$ that restricts to $p$ on $M_{\calp}$.  The closure of each component of $E(\calp)-\calp$ intersects $\calp$ in a unique $g\in\cals_e$, and the map $x \mapsto (r(x),d(x,r(x)))$ determines a homeomorphism to $g\times [0,\infty)$.  The inverses of these homeomorphisms, taken over the disjoint union of all $g\in\cals_e$, combine to induce the map in question.  

The two paragraphs above combine to prove that $C(\Gamma)=p(M_{\calp})$.  In particular, $C(\Gamma)$ has totally geodesic boundary, so its preimage in $\mathbb{H}^3$ under the universal cover is a disjoint union of geodesic planes.  Since $p$ takes $g\in\cals_e$ to $\partial C(\Gamma)$, the hyperplane $\calh$ containing $g$ is a component of the preimage of $\partial C(\Gamma)$.  The final claim of the lemma follows.
\end{proof}

\begin{figure}

\setlength{\unitlength}{.1in}

\begin{picture}(40,25)

\put(3,0) {\includegraphics[height= 2.5in]{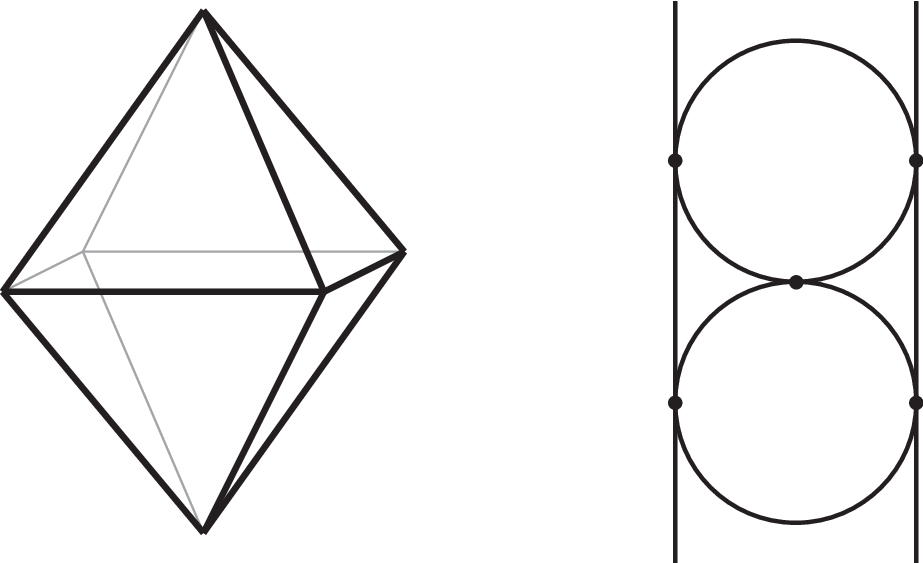}}

\put(17,15){$B$}

\put(10, 15){$X_4$}

\put(10, 9){$A$}

\put(19.8,9.2){\vector(-1,1){2}}

\put(20.3,8){$X_1$}

\put(1, 5){$X_2$}

\put(3,6){\vector(1,1){2}}

\put(20.5,21){$X_3$}

\put(20,21){\vector(-1,-1){2}}

\put(1,12){$\scriptstyle{\infty}$}

\put(15.5,12.6){$\scriptstyle{1}$}

\put(11.5,0){$\scriptstyle{0}$}

\put(30,12){$X_2$}

\put(45.5, 12){$X_4$}

\put(37.5, 6.5){$X_1$}

\put(37.5, 17){$X_3$}

\put(34,17.5){$\scriptstyle{i}$}

\put(34, 6.8){$\scriptstyle{0}$}

\put(41,17.5){$\scriptstyle{1+i}$}

\put(42, 6.8){$\scriptstyle{1}$}

\end{picture}

\caption{The regular ideal octahedron $\calp_1$, and its expansion $E(\calp_1)$}  \label{cap}

\label{fig:idealoct}

\end{figure}
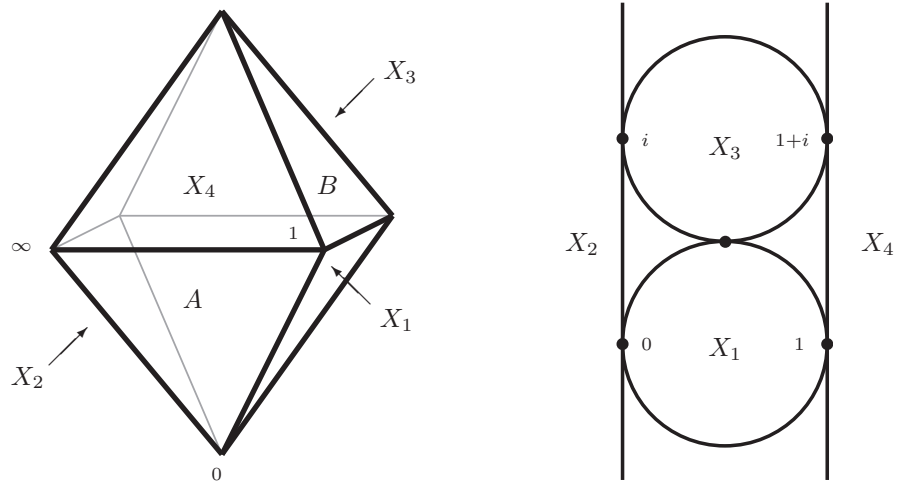

\begin{cor} \label{M_SandG} Let $\calp_1$ be the regular ideal octahedron in $\mathbb{H}^3$, embedded as indicated in Figure \ref{fig:idealoct}, and checkered by declaring the face $A$ to be external.  The collection $\{\s^{\pm 1},\t^{\pm 1}\}$ is an internal face pairing for $\calp_1$, where 
\[ \s= \left(\begin{array}{cc} 1 & 0 \\ -1 & 1 \end{array}\right) \AND
   \t= \left(\begin{array}{cc} 2i & 2-i \\ i & 1-i \end{array}\right). \]
Let $M_S = \calp_1/\{\s^{\pm1},\t^{\pm 1}\}$, and let $\Gamma_S = \langle \s,\t \rangle$.  Then the inclusion $\calp_1 \to \mathbb{H}^3$ induces an isometry $p_S \co M_S \to C(\Gamma_S)$.  \end{cor}

\begin{proof}  With the indicated embedding, $\mathcal{P}_1$ is a tile of the $\mathrm{PSL}_2(\mathcal{O}_1)$--invariant tesselation $\mathcal{T}_1$ constructed in \cite{Hatcher}.  Here $\mathcal{O}_1 = \mathbb{Z}[i]$ is the ring of integers of the field $\mathbb{Q}(i)$.  In particular, the face $A$ shown on the left in Figure \ref{fig:idealoct} has ideal vertices $0$, $1$, and $\infty$, and all other ideal vertices of $\mathcal{P}_1$ have positive imaginary part.  

Since $A$ is external, the faces $X_1$, $X_2$, $X_3$, and $X_4$ of $\mathcal{P}_1$ indicated on the left in Figure \ref{fig:idealoct} are internal.  Direct computation reveals that $\s$ takes $X_1$ to $X_2$, fixing the ideal vertex they share, and $\t$ takes $X_3$ to $X_4$ so that the vertex they share goes to the vertex shared by $X_4$ and $X_2$.  Hence $\{\s^{\pm 1},\t^{\pm 1}\}$ is an internal face pairing for $\calp_1$.  The corollary now follows from Lemma \ref{combinatorics n cores}.  \end{proof}

The external faces of $\calp_1$ triangulate $\partial M_S$, and their images under $p_S$ determine a triangulation of $\partial C(\Gamma_S)$, which we will denote by $\Delta_S$. 

\begin{cor} \label{M_T0andH0} Let $\mathcal{P}_2$ be the right-angled ideal cuboctahedron in $\mathbb{H}^3$, embedded as indicated in Figure \ref{middle}, and checkered by declaring triangular faces external.  The collection $\{\f^{\pm 1},\g^{\pm 1},\h^{\pm 1}\}$ is an internal face pairing for $\calp_2$, where \begin{align*} 
  & \f\ =\ \left( \begin{smallmatrix} 1 & 0 \\ -1 & 1 \end{smallmatrix} \right) &
  & \g\ =\ \left( \begin{smallmatrix} -1+i\sqrt{2} &  1-2i\sqrt{2} \\ -2 & 3-i\sqrt{2} \end{smallmatrix} \right) &
  & \h\ =\ \left( \begin{smallmatrix} 2i\sqrt{2} & -3-i\sqrt{2} \\ -3+i\sqrt{2} & -3i\sqrt{2} \end{smallmatrix} \right). 
\end{align*}
Let $M_{T_0} = \mathcal{P}_2/\{\f^{\pm 1},\g^{\pm 1},\h^{\pm 1}\}$, and let $\Gamma_{T_0} = \langle \f,\g,\h \rangle$.  The inclusion $\calp_2 \to \mathbb{H}^3$ induces an isometry $p_{T_0} \co M_{T_0} \to C(\Gamma_{T_0})$.  \end{cor}

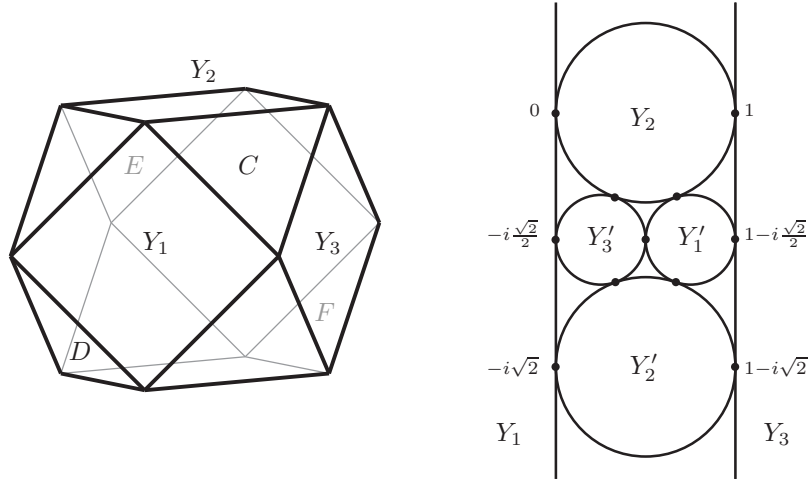
\begin{figure}

\setlength{\unitlength}{.1in}

\begin{picture}(40,25)

\put(3.5,0) {\includegraphics[height= 2.5in]{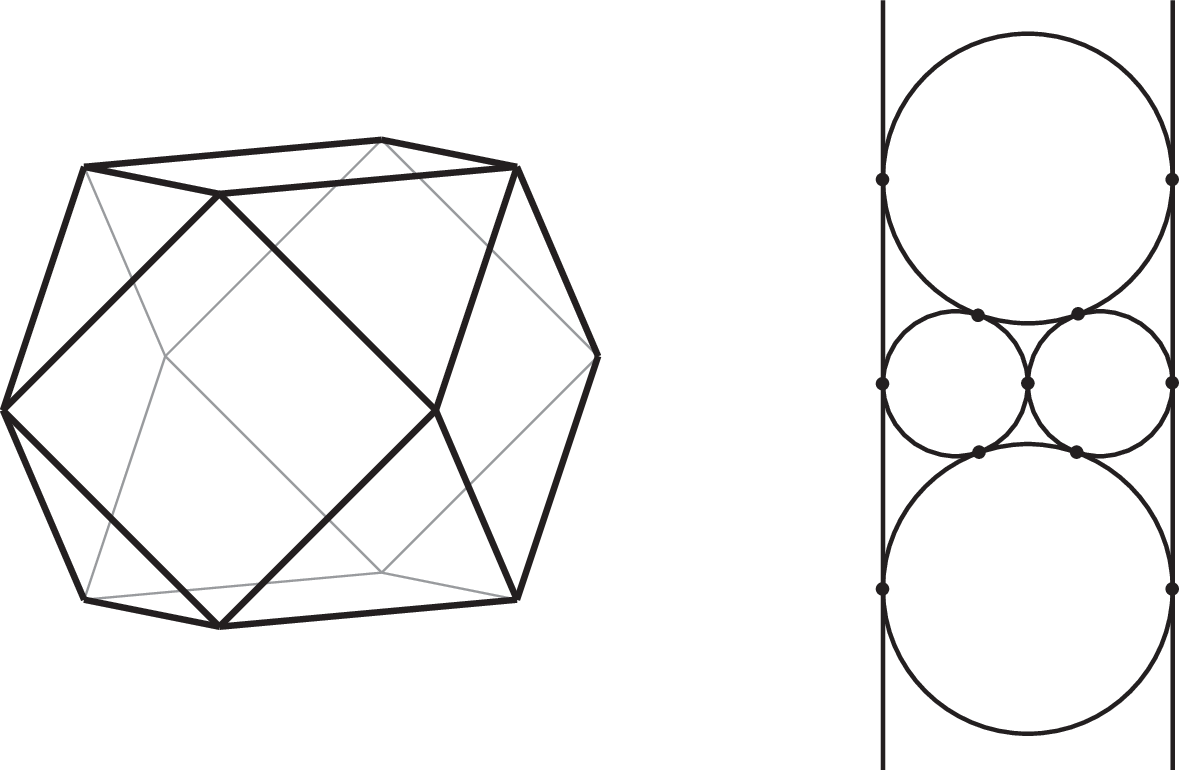}}

\put(15.5,16){$C$}

\put(10.5, 12){$Y_1$}

\put(6.65, 6.25){$D$}

\put(9.5,16){$\textcolor{gray1}{E}$}

\put(19.5,8.25){$\textcolor{gray1}{F}$}

\put(13,21){$Y_2$}

\put(19.5,12){$Y_3$}

\put(43, 2){$Y_3$}

\put(29, 2){$Y_1$}

\put(36, 18.5){$Y_2$}

\put(36, 5.5){$Y'_2$}

\put(33.75,12.25){$Y'_3$}

\put(38.5,12.25){$Y'_1$}

\put(28.5,12.5){$\scriptstyle{-i\frac{\sqrt{2}}{2}}$}

\put(42,19){$\scriptstyle{1}$}

\put(28.5, 5.5){$\scriptstyle{-i\sqrt{2}}$}

\put(30.75,19){$\scriptstyle{0}$}

\put(42, 5.5){$\scriptstyle{1-i\sqrt{2}}$}

\put(42,12.5){$\scriptstyle{1-i\frac{\sqrt{2}}{2}}$}

\end{picture}

\caption{The right-angled ideal cuboctahedron $\calp_2$, and $E(\calp_2)$}  \label{middle}

\label{fig:cuboct}

\end{figure}

\begin{proof}  With the indicated embedding, $\mathcal{P}_2$ is a tile of the $\mathrm{PSL}_2(\mathcal{O}_2)$--invariant tesselation $\mathcal{T}_2$ of $\mathbb{H}^3$ defined in \cite{Hatcher}, where $\mathcal{O}_2 = \mathbb{Z}[i\sqrt{2}]$ is the ring of integers of $\mathbb{Q}(i\sqrt{2})$.  In particular, the face $C$ labeled in the figure has ideal vertices $0$, $1$, and $\infty$.

Label the internal faces $Y_i$ as indicated on the left in Figure \ref{fig:cuboct}, and label the square face opposite $Y_i$ as $Y_i'$.  Direct computation reveals that $\f$ takes $Y_2$ to $Y_1$, fixing the ideal vertex they share, $\g$ takes $Y_3$ to $Y'_1$, fixing the ideal vertex they share, and $\h$ takes $Y'_2$ to $Y'_3$, taking the vertex they share to the opposite vertex on $Y_3'$.  Hence $\{\f^{\pm 1},\g^{\pm 1}, \h^{\pm 1}\}$ is an internal face pairing for $\calp_2$.  The corollary now follows from Lemma \ref{combinatorics n cores}.  \end{proof}  

The external faces of $\calp_2$ triangulate $\partial M_{T_0}$.  This has two components that we will call $\partial_+ M_{T_0}$ and $\partial_- M_{T_0}$, with the latter triangulated by the letter-labeled faces of Figure \ref{fig:cuboct}.  Let $\bound_{\pm} C(\Gamma_{T_0}) = p_{T_0}(\partial_{\pm} M_{T_0})$ and let $\Delta^{\pm}_{T_0}$ refer to the triangulation for $\bound_\pm C(\Gamma_{T_0})$ determined by the images under $p_{T_0}$ of the external faces of $\calp_2$.

In the remainder of the paper, if $\g$ and $\h$ are elements of a group and $G$ is a subgroup, we let $\g^\h$ denote the conjugate of $\g$ by $\h$, ${\sf hgh}^{-1}$, and $G^\h=\h G\h^{-1}$.  Below we describe parabolic isometries $\p_1$, $\p_2$ and $\p_3$ which lie in $\Gamma_S \cap \Gamma_{T_0}$.  
\[ \begin{array}{rcccccc}

\p_1 & = & \s^{-1} & = & \f^{-1} & = &  \begin{pmatrix} 1 & 0 \\ 1 & 1 \end{pmatrix}  \\ 

\p_2 & = &  {\sf stst}^{-2} & = & {\sf fg}^{-1}\f^{-1}\h^{-1}\g & = & \begin{pmatrix} -1 & 5 \\ 0 & -1 \end{pmatrix}  \\ 

\p_3 & = & (\s^{-1})^{{\sf tst}} & = & (\g^{-1})^{\g^{-1}\f^{-1}\h} & = & \begin{pmatrix} -14 & 25 \\ -9 & 16 \end{pmatrix}. \end{array} \]
Since these are in $\mathrm{PSL}_2(\mathbb{R})$, they stabilize the hyperplane $\calh$ with boundary $\mathbb{R}\cup\{\infty\}$.
 
\begin{lemma} \label{M_Sboundary}  The polygon $\mathcal{F}$ of Figure \ref{Sface} is a fundamental domain for the action of $\Lambda \doteq \langle \p_1,\p_2,\p_3 \rangle < \mathrm{PSL}_2(\mathbb{R})$ on $\calh$, and $F^{(0)} = \calh/\Lambda$ is a $4$-punctured sphere.  Also:  \begin{enumerate}
\item $\Lambda = \mathrm{Stab}_{\Gamma_S}(\mathcal{H}) = \mathrm{Stab}_{\Gamma_{T_0}}(\calh)$, 
\item the inclusion $\mathcal{H} \hookrightarrow \mathbb{H}^3$ induces an isometry $\iota_-^{(0)} \co F^{(0)} \rightarrow \bound C(\Gamma_S)$ and an isometry $\iota^{(0)}_+ \co F^{(0)} \to \partial_- C(\Gamma_{T_0})$, and
\item  the triangulation of $\mathcal{F}$ pictured in Figure \ref{Sface} projects to a triangulation $\Delta_F$ of $F^{(0)}$ taken by $\iota^{(0)}_-$ and $\iota^{(0)}_+$, respectively, to $\Delta_S$ and $\Delta^{-}_{T_0}$. 
\end{enumerate}
\end{lemma}

\begin{figure}[ht]
\begin{center}
\input{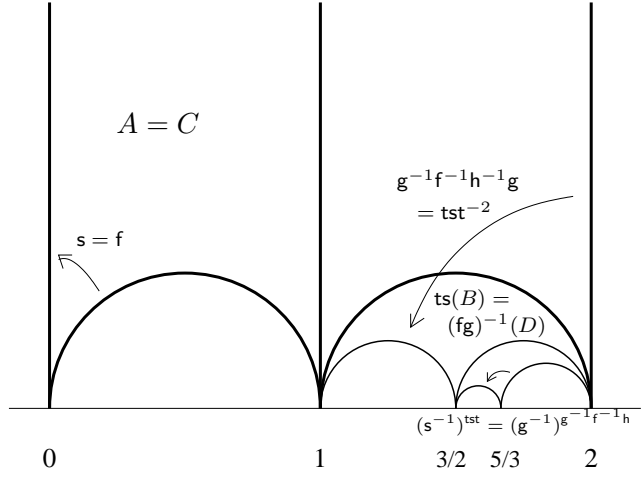}
\end{center}
\caption{A triangulated fundamental domain $\mathcal{F}$ for the action of $\Lambda$ on $\calh$, and side pairings.}
\label{Sface}
\end{figure}

\begin{proof}  With $\calp_1$ and $\calp_2$ embedded as prescribed in Figures \ref{fig:idealoct} and \ref{fig:cuboct}, respectively, their faces $A$ and $C$ coincide and lie in $\calh$ as described in Figure \ref{Sface}.  As noted in the proofs of Corollaries \ref{M_SandG} and \ref{M_T0andH0}, $\Gamma_S$-translates of $\calp_1$ lie in the tesselation $\mathcal{T}_1$ described in \cite{Hatcher}, and $\Gamma_{T_0}$-translates of $\calp_2$ lie in $\mathcal{T}_2$.  The Farey tesselation is $\mathcal{T}_1 \cap\calh = \mathcal{T}_2\cap\calh$, so this contains any $\Gamma_S$-translate of any face of $\calp_1$ and any $\Gamma_{T_0}$-translate of any face of $\calp_2$.  

Let $A'$ be the external face of $\calp_1$ which shares the ideal vertex $0$ with $A$, and let $B'$ be the external face which shares the vertex $\infty$ with $A$ and $1+i$ with $B$.  Since $\t$ takes $X_3$ to $X_4$, with the edge $X_3 \cap B'$ taken to $X_4 \cap A$, it follows that $\t(B')$ lies in $\calh$, abutting $A$ along the geodesic between $1$ and $\infty$.  Since $\t(B')$ is a Farey triangle it has its other ideal vertex at $2$.  It follows similarly that $\g^{-1}(E) = \t(B')$ (where $E$ is as labeled in Figure \ref{fig:cuboct}), that $\t\s(B) = (\f\g)^{-1}(D)$, as indicated in Figure \ref{Sface}, and that $\t\s\t(A') = \g^{-1}\f^{-1}\h(F)$ has vertices at $3/2$, $2$, and $5/3$.

Combinatorial considerations or direct calculation establish that $\s = \f$, $\t\s\t^{-2} = \g^{-1}\f^{-1}\h^{-1}\g$, and $(\s^{-1})^{\t\s\t} = (\g^{-1})^{\g^{-1}\f^{-1}\h}$, and that each stabilizes $\calh$ and pairs edges of $\calf$ as indicated in Figure \ref{Sface}.  By inspection the quotient is a $4$-punctured sphere $F^{(0)}$.  By the polyhedron theorem $\calf$ is a fundamental domain for the group that they generate, which acts on $\calh$ with quotient $F^{(0)}$.  Since $\p_1$, $\p_2$, and $\p_3$ are easily obtained from the edge pairings above and vice-versa, it follows that 
$$ \Lambda = \left\langle \s, \t\s\t^{-2}, (\s^{-1})^{\t\s\t} \right\rangle = \left\langle \f,\g^{-1}\f^{-1}\h^{-1}\g,(\g^{-1})^{\g^{-1}\f^{-1}\h} \right\rangle. $$
Therefore $\calf$ is a fundamental domain for $\Lambda$, and $\calh/\Lambda=F^{(0)}$.  

It is easy to see from its combinatorics that $\partial M_S$ is a four-punctured sphere, as is the component of $\partial M_{T_0}$ containing $C$.  Thus by Corollaries \ref{M_SandG} and \ref{M_T0andH0} the same holds true for $\partial C(\Gamma_S)$ and the component of $\partial C(\Gamma_{T_0})$ containing the image of $C$.   Lemma \ref{combinatorics n cores} implies that $\partial C(\Gamma_S)$ is the image of $\calh/\mathrm{Stab}_{\Gamma_S}(\calh)$ under the inclusion-induced map.  Since it is clear from the above that $\Lambda<\mathrm{Stab}_{\Gamma_S}(\calh)$, and since $\calh/\Lambda$ is itself a four-punctured sphere, the conclusions of assertions (1) and (2) above follow for $\Gamma_S$.  A similar argument implies the same for $\Gamma_{T_0}$.  The conclusions of (3) follow from the description above of the triangulation of $\calf$.\end{proof}

\noindent {\it Remarks.} \smallskip \\
1.  The parabolic elements of $\Lambda$ fixing the ideal points $0, \, \infty,$ and $5/3$ of $\mathcal{H}$ are $\p_1, \, \p_2$, and $\p_3$.  The final conjugacy class of parabolic elements in $\Lambda$ is represented by: 
$$ \p_4 \ = \  \p_1 \p_2 \p_3^{-1} \ = \ ({\sf sts t}^{-2})^{{\sf tst}^{-1}} = \begin{pmatrix} 29 & -45 \\ 20 & -31 \end{pmatrix}. $$
Evidently $\p_1$ and $\p_3$ are conjugate in $\Gamma_S$, as are $\p_2$ and $\p_4$.  The combinatorial considerations of Lemma \ref{lem:M_S_moduli} will show that $C(\Gamma_S)$ has exactly two  cusps, each of rank one, so every parabolic element of $\Gamma_S$ is conjugate to one of $\p_1$ or $\p_2$.  \smallskip \\
2.  There exists $\k \in \mathrm{PSL}_2(\mathbb{C})$, with order $2$, which normalizes $\Gamma_{T_0}$:  \begin{align} \label{sigma}
  \k \ = \ \begin{pmatrix} i & i-\sqrt{2} \\ 0 & -i \end{pmatrix}.  \end{align}
The action of $\k$ on the generators $\f,$ $\g,$ and $\h$ is given by 
\[ \f^\k = \g^{\f\g^{-1}}, \qquad \g^\k = \f^{{\sf fg}^{-1}}, \qquad \text{and} \qquad \h^\k = (\h^{-1})^{{\sf fg}^{-1}}. \]
   \smallskip

If $\Gamma$ is a Kleinian group and ${\sf u} \in \text{Isom}(\mathbb{H}^3)$, we write $\phi_{\sf u} \co C(\Gamma) \rightarrow C(\Gamma^{\sf u})$ for the restriction to $C(\Gamma)$ of the isometry $\mathbb{H}^3/\Gamma\to\mathbb{H}^3/\Gamma^{\sf u}$ induced by ${\sf u}$.  
Since ${\sf k}$ normalizes $\Gamma_{T_0}$, $\phi_\k \co C(\Gamma_{T_0}) \to C(\Gamma_{T_0})$ is an orientation-preserving involution.  The elements $\p_i^\k,$ $i \in \{1,2,3\},$ all preserve the geodesic hyperplane $\k (\mathcal{H})$, which lies over the line $\mathbb{R} - i\sqrt{2}$ and contains an external face of $\mathcal{P}_2$ projecting to $\partial_+ C(\Gamma_{T_0})$.  The lemma below follows and, together with Lemma \ref{M_Sboundary}, completely describes $\partial C(\Gamma_{T_0})$.

\begin{lemma}\label{M_T0boundary}  $\Lambda^{\k} = \mathrm{Stab}_{\Gamma_T}(\k(\calh))$, and the inclusion $\k(\calh) \to \mathbb{H}^3$ induces an isometry from $F'\doteq \k(\calh)/\Lambda^{\k}$ to $\partial_+ C(\Gamma_{T_0})$.  \end{lemma}

It is easy to see that $\p_1^{\k}$ is conjugate in $\Gamma_{T_0}$ to $\p_3$ and that $\p_2^{\k} = \p_2^{-1}$.  The combinatorial considerations of Lemma \ref{lem:M_T_moduli} will imply that $M_{T_0}$ has four cusps.  Hence by Lemma \ref{M_T0boundary}, each of the cusps of $C(\Gamma_{T_0})$ joins $\partial_- C(\Gamma_{T_0})$ to $\partial_+ C(\Gamma_{T_0})$, and each parabolic in $\Gamma_{T_0}$ is conjugate to exactly one $\p_i$, $i \in \{1,2,3,4\}$.


Our second main tool in this section is Lemma \ref{convexcore} below.  We refer to \cite[Definition 4.8]{Mo} for the definition of a pared manifold.

\begin{lemma}  Let $(M,P)$ be a pared manifold, and suppose that $\rho: \pi_1 M \rightarrow \Gamma < \mathrm{PSL}_2 (\mathbb{C})$ is a faithful representation onto a non-Fuchsian geometrically finite Kleinian group $\Gamma$, where $C(\Gamma)$ has totally geodesic boundary.  If $\rho$ determines a one--to--one correspondence between conjugacy classes of subgroups of $\pi_1(M)$ corresponding to components of $P$ and conjugacy classes of maximal parabolic subgroups of $\Gamma$, then $\rho$ is induced by a homeomorphism of $M - P$ to $C(\Gamma)$.  \label{convexcore} \end{lemma}

This is well known to experts in Kleinian groups, but we do not know of a reference for a written proof.  It seems worth writing down as it may fail if $C(\Gamma)$ does not have totally geodesic boundary (see \cite{CM} for a thorough exploration of this phenomenon).  The proof follows easily from results in \cite{CM} for example, but requires introduction of the characteristic submanifold machinery.  Since this falls outside the scope of the rest of the paper, we defer the proof to Appendix \ref{appendix:convexcore}.

Let $(B^3,S)$ be the tangle pictured on the left in Figure \ref{SandT_0}.  Take a base point for $\pi_1(B^3-S)$ on $\partial (B^3-S)$ high above the projection plane, and let its Wirtinger generators correspond in the usual way to labeled arcs of the diagram. 

\begin{prop} \label{Smap}  There is a homeomorphism $f_S \co B^3 - S \to C(\Gamma_S)$, such that  \[f_{S\ast} \co \pi_1(B^3-S) \rightarrow \Gamma_S\] is given by $f_{S\ast}(a) = \p_1^{-1}$, $f_{S\ast}(e) = \p_2$, and $f_{S\ast}(v) = \p_3^{-1}$.  \label{Srep} \end{prop}

\begin{proof}  

Reducing a standard Wirtinger presentation for $\pi_1(B^3-S)$, we obtain 

$$\left \langle a, w, e \,\Big| \, ew\inv{e}a = awa\inv{w} \right \rangle 
  = \left\langle a, w, e\, \Big|\, w(e^{-1}aw) = (e^{-1}aw) a \right\rangle$$
Thus taking $b = e^{-1}aw$, one finds that $\pi_1(B^3 - S)$ is freely generated by $a$ and $b$.

By Lemma \ref{combinatorics n cores} and Corollary \ref{M_SandG}, $\Gamma_S$ is free on the generators $\s$ and $\t$.  Hence, the map $f_{S\ast} \co \pi_1(B^3-S) \longrightarrow \Gamma_S$ given by $a \mapsto \s$ and $b \mapsto \t$ is an isomorphism.  Notice that the subgroup of $\pi_1(B^3-S)$ corresponding to the 4-punctured sphere $\bound B^3-\bound S$ is freely generated by $a,$ $v,$ and $e.$  It is easily checked that

$$\begin{array}{ccccc}

f_{S\ast}(v) & = & \begin{pmatrix} 16 & -25 \\ 9 & -14 \end{pmatrix}  & = & \p_3^{-1} \\ \\

&& \mathrm{and} \\ \\

f_{S\ast}(e) & = & \begin{pmatrix} -1 & 5 \\ 0 & -1 \end{pmatrix} & = & \p_2 .

\end{array}$$

$f_{S\ast}$ takes $\pi_1(\partial B^3 - S)$ isomorphically to $\Lambda$, since $a$, $v$, and $e$ generate $\pi_1(\partial B^3 - S)$ and their images in $\Gamma_S$ generate $\Lambda$.  Since any meridian of $S$ is conjugate in $\pi_1(B^3-S)$ to either $a$ or $e$, and these are taken to $\p_1$ and $\p_2$ respectively, meridians are taken to parabolic elements of $\Gamma_S$. 

Now let $N(S)$ be a small open tubular neighborhood of $S$ in $B^3$.  Then $B^3-N(S)$ is a compact manifold with genus two boundary, and the pair $(B^3 - N(S),\partial N(S))$ is a pared manifold.  The proposition follows from Lemma \ref{convexcore}, after noting that $(B^3-N(S))-\partial N(S)$ is homeomorphic to $B^3-S$.  \end{proof}

Let $(S^2\times I,T_0)$ be the tangle pictured on the right side of Figure \ref{SandT_0}, where $I$ is oriented so that $\partial_- T_0 \doteq T_0 \cap S^2 \times \{0\}$ contains the endpoints labeled $a$, $u$, and $v$.  Take a base point for $\pi_1 (S^2 \times I - T_0)$ on $S^2 \times \{0\}$ high above the projection plane and let Wirtinger generators correspond to the labeled arcs of Figure \ref{SandT_0}.

The proposition below is the analog for $T_0$ of Proposition \ref{Smap}.

\begin{prop} \label{T0map}    There is a homeomorphism $f_{T_0} \co S^2 \times I - T_0 \longrightarrow C(\Gamma_{T_0})$ such that
$$  f_{T_0 \ast} \co \pi_1 (S^2 \times I - T_0) \longrightarrow \Gamma_{T_0}  $$
is given by $f_{T_0 \ast}(a)=\p_1^{-1}$, $f_{T_0 \ast}(e)=\p_2$, and $f_{T_0 \ast}(v)=\p_3^{-1}$. \end{prop}

\begin{proof}  $(S^2 \times I - N(T_0))$ may be isotoped in $S^3$ to a standard embedding of a genus-3 handlebody.  Thus $\pi_1(S^2 \times I -T_0)$ is free on three generators.  We claim that the group is generated by $a$, $e$, and $t$.  This follows after noticing that $v=\inv{y}xy$ where $y=\inv{(ta)}a(ta)$ and $x=\inv{(azq)}t(azq)=\inv{(ate)}t(ate)$.   (The relation $zq=te$ used in the last equality comes from the relationship between four peripheral elements in a 4-punctured sphere group.)  So far, we have established that $v, y \in \langle a, e, t \rangle.$  Now using the other punctured sphere relation, we have $u = a^{-1}ev \in \langle a, e, t \rangle.$  Finally, $z=yuy^{-1}$ and $q=z^{-1}te.$  Therefore $a,$ $e,$ and $t$ generate the group as claimed.

By Lemma \ref{combinatorics n cores} and Corollary \ref{M_T0andH0}, $\Gamma_{T_0}$ is freely generated by $\f$, $\g$, and $\h$.  For our purposes, a more convenient free generating set for $\Gamma_{T_0}$ is $\{\f,\, {\sf fgf}^{-1},\, \p_2  \}.$  Note that all of these generators are parabolic and peripheral, and conjugation by $\k$ interchanges the first two and takes the third to its inverse.  The representation of $\pi_1(S^2 \times [0,1/2] - T_0)$ given by \begin{align*}
& a \mapsto \f & & t \mapsto {\sf fgf}^{-1} & & e \mapsto \p_2
\end{align*}
is clearly faithful, and it is easily checked that $v$ maps to $\p_3^{-1}$.  Because $u = a^{-1}ev$ is mapped to $\p_1\p_2\p_3^{-1} = \p_4$, we conclude that meridians are mapped to parabolic elements and that $\pi_1 (S^2 \times \{0\}-\bound_- T_0)$ is taken to $\Lambda$.  The result now follows from Lemma \ref{convexcore} as previously.  \end{proof}

There is a visible involution of $S^2 \times I-T_0$ which is a rotation by $\pi$ around a circle in $S^2 \times \{ 1/2\}$.  This involution exchanges the two boundary components.  With a proper choice of path between our basepoint and its image under this involution, the corresponding action on $\pi_1(S^2 \times I-T_0)$ is given by \begin{align*}
  & a \leftrightarrow t & & e \leftrightarrow e^{-1} \end{align*}
This commutes with the action of the element $\k$ defined in (\ref{sigma}) on $\Gamma_{T_0}$, under the representation $f_{T_0*}$.  Hence this involution is isotopic to the pullback of $\phi_{\k}$ by $f_{T_0}$.

Recall from Lemma \ref{M_Sboundary} that $\Lambda = \mathrm{Stab}_{\Gamma_{T_0}}(\calh)$, and from Lemma \ref{M_T0boundary} that $\Lambda^{\k} = \mathrm{Stab}_{\Gamma_{T_0}}(\k(\calh))$.  By its definition in Proposition \ref{T0map}, it is clear that $f_{T_0*}$ maps $\pi_1 \left(S^2 \times \{0\} - \partial_- T_0\right)$ isomorphically to $\Lambda$.  Since $\calh$ projects to $\partial_- C(\Gamma_{T_0})$, using the involution equivariance of $f_{T_0}$ we obtain the corollary below.

\begin{cor}\label{T0boundary}  Let $\partial_+ T_0 = T_0 \cap S^2 \times \{1\}$.  Then $f_{T_0}(S^2 \times \{0\} - \partial_- T_0) = \partial_- C(\Gamma_{T_0})$, and $f_{T_0}(S^2 \times \{1\} - \partial_+ T_0) = \partial_+ C(\Gamma_{T_0})$.  \end{cor}

\section{Combination}  \label{sec:combo-nation!}

In this section, we will describe how to join copies of the tangles $S$ and $T_0$ to construct links in $S^3$ whose complements are uniformized by combinations of $\Gamma_S$ and $\Gamma_{T_0}$.  The main tool in this section is a corollary of Maskit's combination theorem for free products with amalgamation \cite{Mask}.  Denote the convex hull of the limit set for a Kleinian group $\Gamma$ by $\text{Hull}(\Gamma)$. 

\begin{dfn}  Kleinian groups $\Gamma_0$ and $\Gamma_1$ \textit{\meetcute}along a hyperplane $\mathcal{K} \subset \mathbb{H}^3$ if $\mathcal{K} = \text{Hull}(\Gamma_0) \cap \text{Hull}(\Gamma_1)$ 
and $\mathrm{Stab}_{\Gamma_0}(\calk) = \mathrm{Stab}_{\Gamma_1}(\calk)$.  \end{dfn}

The fact below follows easily from this definition, and accounts for its utility.  

\begin{fact} If $\Gamma_0$ and $\Gamma_1$ \meetcute along $\calk$ then $\mathrm{Stab}_{\Gamma_0}(\calk) = \Gamma_0\cap\Gamma_1 = \mathrm{Stab}_{\Gamma_1}(\calk)$.  Furthermore, $\calk$ divides $\mathbb{H}^3$ into $\mathcal{B}_0$ and $\mathcal{B}_1$ such that for $i \in \{0,1\}$, if $\g_i \in \Gamma_i$ satisfies $\g_i(\mathcal{B}_{1-i}) \cap \mathcal{B}_{1-i} \neq \emptyset$, then $\g_i \in \Gamma_0 \cap \Gamma_1$. \end{fact}

In general, if $\Theta$ is a subgroup of $\Gamma$, the limit set of $\Theta$ is contained in that of $\Gamma$, and so the covering map $\mathbb{H}^3/\Theta \rightarrow \mathbb{H}^3/\Gamma$ maps $C(\Theta)$ into $C(\Gamma)$ --- we will call this restriction the \textit{natural map} $C(\Theta) \rightarrow C(\Gamma)$.  When $\Gamma_0$ and $\Gamma_1$ meet cute along $\mathcal{K}$ then the natural map $C(\Gamma_0 \cap \Gamma_1) \rightarrow C(\Gamma_i)$ restricts to an embedding of the 2-orbifold $\mathcal{K}/(\Gamma_0 \cap \Gamma_1)$.

The lemma below is a geometric combination theorem for Kleinian groups which \meetcute along a hyperplane.  It follows from Maskit's combination theorem and observations on the geometry of Kleinian groups that go back at least to J.~Morgan's account of geometrization for Haken manifolds \cite{Mo}.

\begin{lemma}\label{Maskit}  Suppose $\Gamma_0$ and $\Gamma_1$ \meetcute along a plane $\mathcal{K}$.  Let $E = \mathcal{K}/\Theta$, where $\Theta = \Gamma_0 \cap \Gamma_1$, and for $i=0$, $1$ let $\iota_i \co E \to C(\Gamma_i)$ be the natural embedding.  
Then $\langle \Gamma_0,\Gamma_1\rangle$ is a Kleinian group, and the inclusions $\Gamma_i \rightarrow \langle \Gamma_0,\Gamma_1\rangle$ determine an isomorphism $\Gamma_0 \ast_{\Theta} \Gamma_1 \rightarrow \langle \Gamma_0,\Gamma_1\rangle$ as abstract groups.  The natural maps $C(\Gamma_i) \to C(\langle \Gamma_0,\Gamma_1\rangle)$ determine an isometry $C(\Gamma_0) \cup_{\iota_1 \iota_0^{-1}} C(\Gamma_1) \rightarrow C(\langle \Gamma_0,\Gamma_1\rangle)$.  \end{lemma}

In using Lemma \ref{Maskit}, we often write $C(\Gamma_0) \cup_E C(\Gamma_1)$ when the maps $\iota_i$ are clear.

\begin{proof}  We will use the version of Maskit's combination theorem recorded in \cite[Theorem 8.2]{Mo}).  The fact above implies that $\Gamma_0$ and $\Gamma_1$ which \meetcute along a hyperplane $\calk$ satisfy the hypotheses of \cite[Theorem 8.2]{Mo}, from which the group-theoretic conclusions above thus follow.  That the desired isometry exists follows from the remarks in \cite{Mo} below Theorem 8.2, which have since been considerably fleshed out in Anderson--Canary's ``The visual core of a hyperbolic manifold'' \cite{AnCa}.  

The function $\tilde{f}\co\mathbb{H}^3\to [0,1]$ described in \cite{Mo} is the harmonic extension of the characteristic function of $\mathcal{B}_0$: $\tilde{f}(y)$ is the visual measure of the set of vectors pointing from $y$ toward $\mathcal{B}_0$.  See \cite[\S 2]{AnCa} for a precise analytic definition.  It is not hard to see that here $\calk=\tilde{f}^{-1}(\frac{1}{2})$ (cf.~\cite[Proposition 2.2]{AnCa}), whence our $E$ is Morgan's $X = f^{-1}(\frac{1}{2})$.

Our $\iota_i$ is Morgan's $p_i$, mapping to $N_i = \mathbb{H}^3/\Gamma_i$ for $i \in\{0,1\}$.  For each $i$, $\iota_i(E)$ is a convex core boundary component of $N_i$, so $p_0(N_{-}) \cap C(N_0) = p_0(E)$, $p_1(N_{+})\cap C(N_1) = p_1(E)$, and the result follows from the equation at the bottom of \cite[p.~76]{Mo}.  See \cite{AnCa}, Proposition 5.2 and the remarks after 5.3 for related results.\end{proof}

We first apply Lemma \ref{Maskit} to join $C(\Gamma_{T_0})$ to a copy of itself across $\partial_+ C(\Gamma_{T_0})$.  Recall from above Lemma \ref{M_Sboundary} that we have defined $\calh$ to be the geodesic hyperplane of $\mathbb{H}^3$ with ideal boundary $\mathbb{R} \cup \{\infty\}$.  Let $\r \in \text{Isom}(\mathbb{H}^3)$ be the reflection through $\mathcal{H}$.   This acts on $\mathbb{C} \cup \{ \infty\}$ by complex conjugation; thus if ${\sf q} \in \Gamma < \mathrm{PSL}_2(\mathbb{C})$, then ${\sf q}^\r = \bar{{\sf q}}$, where $\bar{\sf q} \in \mathrm{PSL}_2(\mathbb{C})$ is the element whose entries are the conjugates of the entries of ${\sf q}$.  Hence, we let $\overline{\Gamma}$ denote $\Gamma^\r$.

\begin{lemma} \label{M_TandH}  Define $\c = \left( \begin{smallmatrix} 1 & i\sqrt{2} \\ 0 & 1 \end{smallmatrix} \right)$.  Then $\Gamma_T \doteq \langle \Gamma_{T_0}, \overline{\Gamma}_{T_0}^{\c^{-2}} \rangle$ is a Kleinian group, there is an inclusion-induced isomorphism $\Gamma_{T_0} *_{\Lambda^{\k}} \overline{\Gamma}_{T_0}^{\c^{-2}} \to \Gamma_T$ and an isometry $C(\Gamma_{T_0}) \cup_{F'} C(\overline{\Gamma}_{T_0}^{\c^{-2}}) \to C(\Gamma_T)$ determined by the natural maps.  Furthermore, $\c^{-2}\r$ normalizes $\Gamma_T$, and $\phi_{\c^{-2}\r}\co C(\Gamma_T) \to C(\Gamma_T)$ is an orientation-reversing involution fixing $F'$ and exchanging its complementary components.\end{lemma}

\begin{proof}  Recall from Lemmas \ref{M_Sboundary} and \ref{M_T0boundary} that $\Lambda$ and $\Lambda^{\k}$ are the stabilizers in $\Gamma_{T_0}$ of the geodesic planes $\calh$ and $\k(\calh)$, respectively, and that these planes project to the components of $\partial C(\Gamma_{T_0})$.  It follows that $\calh$ and $\k(\calh)$ are components of the boundary of $\mathrm{Hull}(\Gamma_{T_0})$, so $\mathrm{Hull}(\Gamma_{T_0})$ is contained in the region between them.

With $\c$ as defined in the statement of the lemma, note that $\c(\calh)= (\mathbb{R}+i\sqrt{2})\times \mathbb{R}^+$ and that $\c\k(\calh) = \calh$.   Since $\mathrm{Hull}(\Gamma^{\c}_{T_0})$ has boundary components $\c(\calh)$ and $\c\k(\calh) = \calh$, and $\Lambda^{\c\k} = \mathrm{Stab}_{\Gamma^{\c}_{T_0}}(\calh)$ is invariant under conjugation by $\r$, $\Gamma^{\c}_{T_0}$ and $\overline{\Gamma^{\c}_{T_0}}$ \meetcute along $\calh$.  Applying Lemma \ref{Maskit}, we obtain an isomorphism $\Gamma^{\c}_{T_0} *_{\Lambda^{\c\k}} \overline{\Gamma^{\c}_{T_0}} \to \langle \Gamma^{\c}_{T_0},\overline{\Gamma^{\c}_{T_0}} \rangle$ and  an isometry
\[ C(\Gamma_{T_0}^\c) \cup_{\phi_{\c}(F')} C(\overline{\Gamma_{T_0}^\c}) \rightarrow C(\langle \Gamma_{T_0}^\c, \overline{\Gamma_{T_0}^\c} \rangle)\]
induced by the natural maps.  It is clear that $\r$ normalizes $\langle \Gamma^{\c}_{T_0},\overline{\Gamma^{\c}_{T_0}} \rangle$, exchanging amalgamands, hence $\phi_{\r}$ acts as an orientation-reversing involution of $C(\langle \Gamma^{\c}_{T_0},\overline{\Gamma^{\c}_{T_0}} \rangle)$, fixing $F'$ and exchanging $C(\Gamma^{\c}_{T_0})$ with $C(\overline{\Gamma^{\c}_{T_0}})$.

Observe that $\bar{\c}=\c^{-1}$.  It follows that $\overline{\Gamma^{\c}_{T_0}} = \overline{\Gamma}_{T_0}^{c^{-1}}$, and hence that $\Gamma_T = \langle \Gamma_{T_0}^\c, \overline{\Gamma_{T_0}^\c} \rangle^{\c^{-1}}$.  Conjugating the groups of the paragraph above by $c^{-1}$, we obtain an inclusion-induced isomorphism $\Gamma_{T_0} *_{\Lambda^{k}} \overline{\Gamma}_{T_0}^{\c^{-2}} \to \Gamma_T$, and an isometry 
$$ C(\Gamma_{T_0}) \cup_{F'} C(\overline{\Gamma}_{T_0}^{\c^{-2}}) \to C(\Gamma_T)  $$
induced by the natural maps.  Furthermore, $\c^{-1} \r \c = \c^{-2} \r$ normalizes $\Gamma_T$ and induces an orientation-reversing involution $\phi_{\c^{-2}\r}$ fixing $F'$ and exchanging its sides.  \end{proof}

If $M$ is an oriented manifold with a boundary component $F$, the \textit{double} of $M$ across $F$ is $M \cup_F \overline{M}$, where $\overline{M}$ is a copy of $M$ with orientation reversed, and the gluing map $F \to \overline{F} \subset \overline{M}$ is the identity map.  

\begin{cor}\label{the real M_T}  There is an isometry $p_T\co M_T\to C(\Gamma_T)$, where $M_T$ is the double of $M_{T_0}$ across $F'$, that is the natural map following $p_{T_0}$ from Corollary \ref{M_T0andH0} on $M_{T_0}$.\end{cor}

The advantage that $\Gamma_T$ has over $\Gamma_{T_0}$ for our purposes is that the components of $\partial C(\Gamma_T)$ are naturally orientation-reversing isometric, since they are exchanged by $\phi_{\c^{-2}\r}$.  Recall from Lemma \ref{M_Sboundary} that $\Lambda = \mathrm{Stab}_{\Gamma_{T_0}}(\calh)$; thus by Lemma \ref{M_TandH}, $\Lambda = \mathrm{Stab}_{\Gamma_T}(\calh)$.  We will again refer by $i^{(0)}_+$ to the natural map $F^{(0)} \to C(\Gamma_T)$.  Then the lemma below follows from Lemma \ref{M_TandH}.

\begin{lemma}\label{M_Tboundary}  Let $F^{(1)} = \c^{-2}\calh/\Lambda^{\c^{-2}}$, and let $\phi_{\c^{-2}} \co F^{(0)} \to F^{(1)}$ and $\iota^{(1)}_- \co F^{(1)} \to C(\Gamma_T)$ be the natural maps.  Then $\partial C(\Gamma_T) = \partial_- C(\Gamma_T) \sqcup \partial_+ C(\Gamma_T)$, where $\partial_- C(\Gamma_T) \doteq \iota^{(0)}_+(F^{(0)})$ and $\partial_+ C(\Gamma_T) \doteq \iota^{(1)}_-(F^{(1)})$, and $\iota^{(1)}_- \phi_{\c^{-2}} = \phi_{\c^{-2}\r} \iota^{(0)}_+$.  \end{lemma}

$C(\Gamma_T)$ is a geometric model for the double of $(S^2 \times I,T_0)$ across $(S^2 \times \{1\},\partial_+ T_0)$.  Note that the double of $S^2 \times I$ across $S^2 \times \{1\}$ is again homeomorphic to $S^2 \times I$, by a map taking $(p,t) \in S^2 \times I$ to $(p,t/2)$ and $(p,t) \in \overline{S^2 \times I}$ to $(p,1-t/2)$.

\begin{dfn}  Let $(S^2 \times I,T)$ be the double of $(S^2 \times I,T_0)$ across $(S^2 \times \{1\},\partial_+ T_0)$.  We will identify $(S^2 \times I,T_0) \subset (S^2 \times I,T)$ with its image under the map discussed above, so that $T_0 = T \cap S^2 \times [0,1/2]$.  In particular, we have $\bound_- T = \bound_- T_0 = T \cap S^2 \times \{0\}$ and $\bound _+ T_0 = T \cap S^2 \times \{ 1/2\}$, and we will take  $\bound _+ T = T \cap S^2 \times \{ 1\}$.  \end{dfn}

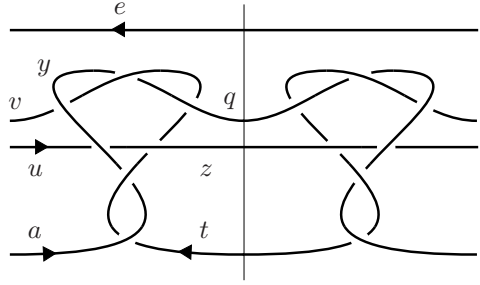
\begin{figure}[ht]

\setlength{\unitlength}{.1in}

\begin{picture}(40,15)

\put(11.5,0) {\includegraphics[height= 1.45in]{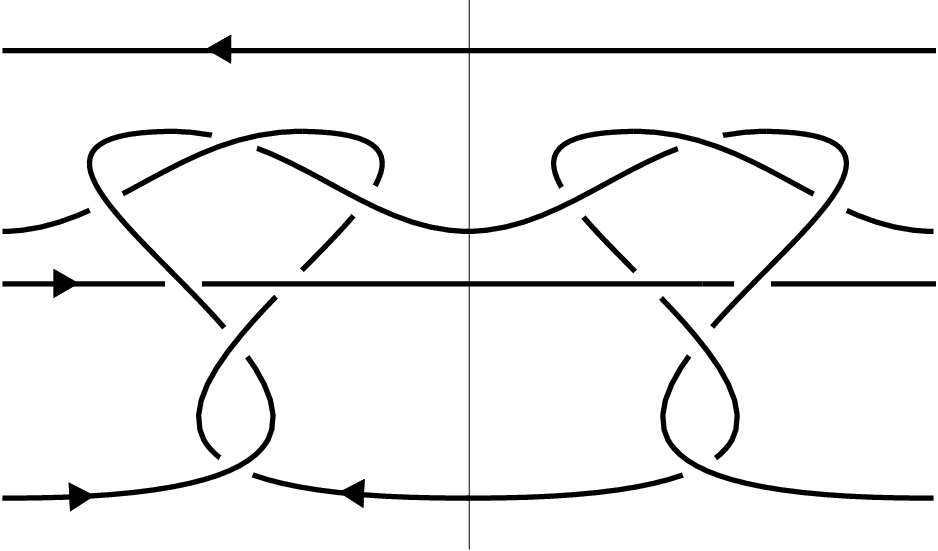}}

\put(17,14){$e$}

\put(13,11){$y$}

\put(21.5, 2.25){$t$}

\put(22.75, 9.25){$q$}

\put(21.5, 5.5){$z$}

\put(12.5, 2.25){$a$}

\put(11.5, 9){$v$}

\put(12.5, 5.5){$u$}

\end{picture}

\caption{The tangle $T \subset S^2\times I$ with labeled Wirtinger generators for $T_0$}  \label{T}
\label{fig:T}

\end{figure}

The tangle $(S^2 \times I,T)$ is pictured in Figure \ref{fig:T}, with $T_0 \subset T$ visible to the left of the grey vertical line representing $S^2 \times \{1/2\}$.  There is a mirror symmetry of $(S^2\times I,T)$, visible in the figure as reflection through the grey vertical line:
$$ r_T \co (S^2 \times I, T) \rightarrow (S^2 \times I, T)  $$
given by $r_T(p,x)=(p,1-x)$, hence fixing $(S^2 \times \{1/2\},\partial_+ T_0)$.

\begin{prop}\label{Tmap} There is a homeomorphism $f_T \co S^2 \times I -T \rightarrow C(\Gamma_T)$, 
which restricts on $S^2 \times [0,1/2]-T_0$ to $f_{T_0}$ followed by the natural map, such that the diagram below commutes.
\[ \xymatrix{  S^2 \times I -T \ar[d]_{r_T} \ar[r]^{f_T} & C(\Gamma_T) \ar[d]^{\phi_{\c^{-2}\r}} \\ S^2 \times I-T \ar[r]^{f_T} & C(\Gamma_T) } \]  
Furthermore, $f_T$ takes $S^2 \times \{0\} - \partial_- T$ to $\partial_- C(\Gamma_T)$ and $S^2 \times \{1\} - \partial_+ T$ to $\partial_+ C(\Gamma_T)$.  \end{prop}

\begin{proof}  We define $f_T$ using the properties described in the statement of the proposition.  Namely, we first require $f_T$ to restrict on $S^2 \times [0,1/2] - T_0$ to the homeomorphism $f_{T_0}$ defined in Proposition \ref{T0map}, followed by the natural map $C(\Gamma_{T_0}) \to C(\Gamma_T)$.  For $x \in S^2 \times [1/2,1] - T$, we define $f_T(x) = \phi_{\c^{-2}\r} f_T\, r_T(x)$.  The resulting map is well-defined, since $r_T$ fixes $S^2 \times \{1/2\} - \partial_+ T_0$ and $\phi_{\c^{-2}\r}$ fixes $F'$.  It is a homeomorphism, since $r_T$, $f_{T_0}$, and $\phi_{\c^{-2}\r}$ are.  By Corollary \ref{T0boundary}, $f_{T_0}$ takes $S^2 \times \{0\} - \partial_- T_0$ to $\partial_- C(\Gamma_{T0})$; it therefore follows from the definitions and Lemma \ref{M_Tboundary} that $f_T(S^2 \times \{0\} - \partial_- T) = \partial_- C(\Gamma_T)$.  The conclusion thus follows from the reflection equivariance of $f_T$.  \end{proof}

\begin{dfns}\label{L_n dfns}  
\[\] \vspace{-.4in}
\begin{enumerate}
\item Let $j \co (\partial B^3, \partial S) \rightarrow (S^2 \times \{0\}, \partial_- T)$ be the homeomorphism such that $(B^3,S) \cup_{j} (S^2 \times I, T)$ is the tangle pictured in Figure \ref{tanglepic}. 
\bigskip
\item Define $h \co S^2 \times \mathbb{R} \rightarrow S^2 \times \mathbb{R}$ by $h(p,x)=(p, x+1)$, and with $T \subset S^2 \times I \subset S^2 \times \mathbb{R}$, let $T^{(i)} = h^{i-1}(T)$ (so $T^{(1)} = T$ in particular).  For $n \in \mathbb{N}$, define
\[  (S^3,L_n)\ = \ (B^3, S) \, \cup_j \, \left(S^2 \times [0,n],  \cup_{i=1}^{n} T^{(i)} \right) \cup_{j_n} (\overline{B}^3,\overline{S}).\]
For $i \in \{0,1,\hdots,n\}$, let $S^{(i)}$ be the image in $(S^3,L_n)$ of $S^2 \times \{i\}\subset S^2\times[0,n]$.    Above, $(\overline{B}^3,\overline{S}) = r_S(B^3,S)$, where $r_S$ is the reflective involution of $S^3$ fixing the boundary of an embedding of $B^3$ and exchanging its sides, and $j_n  \ = \ r_S j^{-1} r_T h^{-n+1} \co \left(S^{(n)}, \bound_+ T^{(n)}\right) \longrightarrow  \left( \bound \overline{B}^3, \bound \overline{S}\right)$. 
\item   Using Figure \ref{fig:T} and taking $T \subset S^2 \times I \subset S^2 \times \mathbb{R}$, label the points of $S^{(0)} \cap L_n = S^2 \times \{0\} \cap T$ by $2$, $3$, $4$, and $1$ top-to-bottom, so that for example $2$ is the terminal point of the tangle string labeled $e$ and $1$ is the initial point of the string labeled $a$.  Label each point of $S^{(i)} \cap L_n$ by its image under $h^{-i}$.
\end{enumerate}  \end{dfns}

\begin{figure}
\begin{center}
\includegraphics[height=1.15in]{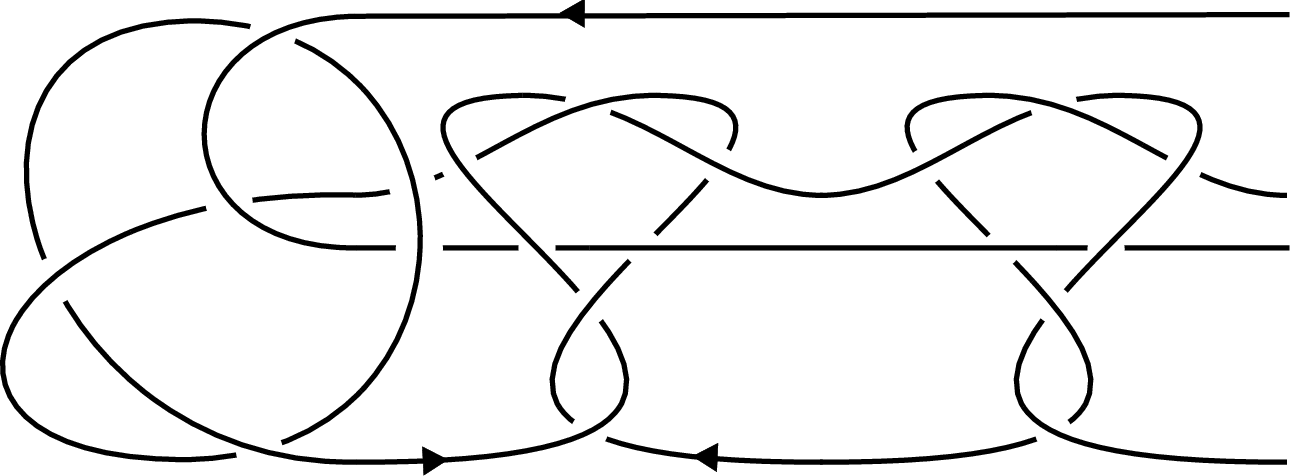}
\end{center}
\caption{$S \cup T$} \label{tanglepic}
\end{figure}

\begin{remark}  With Wirtinger generators for $\pi_1(B^3 - S)$ and $\pi_1(S^2 \times I - T)$ as labeled in Figures \ref{SandT_0} and \ref{fig:T}, we have $j_*(a) = a$, $j_*(u) = u$, and $j_*(v) = v$.\end{remark}

We now construct a geometric model of $S^3-L_n$.

\begin{dfns}\label{geometric pieces} 
\[\] \vspace{-.4in}
 \begin{enumerate}
\item For $i \geq 0$, let $\Lambda^{(i)}= \Lambda^{\c^{-2i}}$ and $F^{(i)}= \c^{-2i}(\mathcal{H})/\Lambda^{(i)}$
\item For $i \geq 1$, let $\Gamma_T^{(i)}=\Gamma_T^{\c^{-2(i-1)}}$, and define $\phi_i = \phi_{\c^{-2(i-1)}} \co C(\Gamma_T) \rightarrow C(\Gamma_T^{(i)})$.
\end{enumerate} \end{dfns}

The definitions above of $F^{(0)}$ and $F^{(1)}$ above agree with our previous definitions.  Also, $\Gamma^{(1)}_T = \Gamma_T$, and since $\Gamma^{(i)}_T = \c^{-2(i-1)}\Gamma_T \c^{2(i-1)}$, Lemma \ref{M_Tboundary} implies that
$$\Lambda^{(i-1)} = \mathrm{Stab}_{\Gamma^{(i)}_T}(\c^{-2(i-1)}(\calh))\ \ \mbox{and}\ \ \Lambda^{(i)} = \mathrm{Stab}_{\Gamma^{(i)}_T}(\c^{-2i}(\calh)),$$ 
and the resulting natural maps $\iota^{(i-1)}_+ \co F^{(i-1)} \to C(\Gamma^{(i)}_T)$ and $\iota^{(i)}_- \co F^{(i)} \to C(\Gamma^{(i)}_T)$ map to the components of its totally geodesic boundary.

\begin{prop}\label{geometric model}  For $n \in \mathbb{N}$, define 
$M_n = C(\Gamma_S) \cup C(\Gamma_T^{(1)}) \cup \hdots \cup C(\Gamma_T^{(n)}) \cup C(\overline{\Gamma}_S)$, using gluing maps defined as follows: \begin{align*} 
  &\iota^{(0)}_+(\iota^{(0)}_-)^{-1}\co \partial C(\Gamma_S) \to \partial_- C(\Gamma_T^{(1)})  \\
  &\iota^{(i)}_+(\iota^{(i)}_-)^{-1} \co \partial_+ C(\Gamma_T^{(i)}) \to \partial_- C(\Gamma_T^{(i+1)})\  \mathrm{for}\ 1\leq i<n \\
  &\phi_{\r}\iota^{(0)}_-\phi_{n+1}^{-1}(\iota^{(n)}_-)^{-1} \co \partial_+ C(\Gamma_T^{(n)}) \to \partial C(\overline{\Gamma}_S) \end{align*}
There is a homeomorphism $f_n\co S^3 - L_n \to M_n$ which restricts on $B^3 - S$ to $f_S$, on $S^2 \times [i-1,i] - T^{(i)}$ to $\phi_i f_T h^{-i+1}$ for $1 \leq i \leq n$, and on $\overline{B}^3-\overline{S}$ to $\phi_{\r} f_S r_S$. \end{prop}

\begin{proof}  We will use the description of $f_n$ above as its definition.  Then by Proposition \ref{Smap} and the definitions, $f_n$ restricts on $B^3 - S$ and $\overline{B}^3-\overline{S}$ to homeomorphisms to $C(\Gamma_S)$ and $C(\overline{\Gamma}_S)$, respectively.  By Proposition \ref{Tmap} and the definitions, for each $i$ between $1$ and $n$ it restricts on $S^2 \times [i-1,i] - T^{(i)}$ to a homeomorphism to $C(\Gamma_T^{(i)})$.  Thus in order to show that $f_n$ is a homeomorphism, we must only show that it is well defined on the spheres $S^{(i)}-\{1,2,3,4\}$ that separate these tangle complements.

We first check the case $i=1$, showing that $f_n$ is well defined on $S^{(0)}-\{1,2,3,4\}$.  Since $T^{(1)}=T$ and $\Gamma_T^{(1)}=\Gamma_T$, and $h^0$ and $\phi_1$ are each the identity map, in this case  we must only show that on $\partial B^3 - \partial S$, $f_T \circ j = \iota^{(0)}_+ (\iota^{(0)}_-)^{-1} \circ f_S$.

By their definitions above, $f_S$ and $f_T \circ j$ induce the same isomorphism from $\pi_1 (\partial B^3 - \partial S)$ to $\Lambda = \Gamma_S \cap \Gamma_T$.  Recall from Lemma \ref{M_Sboundary} and the remarks above Lemma \ref{M_Tboundary} that  the $\iota^{(0)}_{\pm}$ are induced by the inclusions of $\Lambda$ into $\Gamma_S$ and $\Gamma_T$.  Therefore at the level of fundamental group, $\left(\iota^{(0)}_+(\iota^{(0)}_-)^{-1} \circ f_S\right)_* = \left(f_T \circ j\right)_*$.  Since any two homeomorphisms between 4-punctured spheres that induce the same map on fundamental groups are properly isotopic, we may isotope $j$ so that $f_S$ and $f_T j$ agree on $S^{(0)}$.  The conclusion thus follows in this case.

For $1\leq i<n$, we may use the fact that $\Gamma^{(i)}_T$ and $\Gamma^{(i+1)}_T$ are conjugates of $\Gamma_T$ to obtain the following model descriptions for $\iota^{(i)}_+$ and $\iota^{(i)}_-$:  \begin{align}\label{model i's}
 & \iota^{(i)}_+ = \phi_{i+1} \iota^{(0)}_+ \phi_{i+1}^{-1} &
 & \iota^{(i)}_- = \phi_{i} \iota^{(1)}_- \phi_{i}^{-1}  \end{align}
Here $\iota^{(0)}_+ \co F^{(0)} \to \partial_- C(\Gamma_T)$ and $\iota^{(1)}_- \co F^{(1)} \to \partial_+ C(\Gamma_T)$ are the natural maps of Lemma \ref{M_Tboundary}.  Using the reflection-invariance property described there, we thus obtain \begin{align}\label{reflection i's}
  \iota^{(i)}_+(\iota^{(i)}_-)^{-1} = \phi_{i+1} \iota^{(0)}_+\left( \iota^{(1)}_- \phi_{2} \right)^{-1} \phi_{i}^{-1} = \phi_{i+1} \phi_{\c^{-2}\r}^{-1} \phi_{i}^{-1}.  \end{align}
Then by the reflection-equivariance property of $f_T$ from Proposition \ref{Tmap}, we have
$$ \iota^{(i)}_+(\iota^{(i)}_-)^{-1} \circ \phi_{i} f_T h^{-i+1} = \phi_{i+1} \phi_{\c^{-2}\r}^{-1} f_T h^{-i+1} = \phi_{i+1} f_T r_T h^{-i+1}.  $$
It follows directly from the definitions that $r_T h^{-i+1} = h^{-i}$ on $S^{(i)}$, whence $f_n$ is well defined on $S^{(i)}-\{1,2,3,4\}$ for $1 \leq i <n$.

To show $f_n$ is well defined on $S^{(n)}$ requires another definition-chase, this time to check
$$  \phi_{\r}f_Sr_S \circ j_n = \phi_{\r}\iota_-^{(0)} \phi_{n+1}^{-1} (\iota^{(n)}_-)^{-1} \circ \phi_n f_T h^{-n+1}. $$
By Definition \ref{L_n dfns}(2), $j_n =  r_Sj^{-1}r_Th^{-n+1}$; therefore simplifying the left-hand side above  yields $\phi_{\r} f_S j^{-1} r_T h^{-n+1}$.  On the other hand, using the model description of $\iota^{(n)}_-$ from (\ref{model i's}), the right-hand side above simplifies to $\phi_{\r}\iota_-^{(0)}\phi_{2}^{-1}(\iota^{(1)}_-)^{-1}f_Th^{-n+1}$.  The reflection-invariance property of Lemma \ref{M_Tboundary} and an appeal to the case $i=0$ establish the desired equation.  \end{proof}

\begin{cor}\label{geodesic spheres}  For $0 \leq i<n$, refer again by $F^{(i)}$ to the image of $\iota^{(i)}_+(F^{(i)}) \subset C(\Gamma_T^{(i)})$ under its inclusion into $M_n$, and refer by $F^{(n)}$ to the image of $\iota^{(n)}_-(F^{(n)})$.  For each $i$, $F^{(i)}$ is totally geodesic in $M_n$ and $f_n(S^{(i)}-\{1,2,3,4\}) = F^{(i)}$.  \end{cor}

This follows immediately from Proposition \ref{geometric model}, since the maps $\iota^{(i)}_{\pm}$ are isometric embeddings.  The proposition below describes an algebraic model for $M_n$.

\begin{prop}\label{algebraic model}  For $n \in \mathbb{N}$, define $\Gamma_n = \langle \Gamma_S,\Gamma_T^{(1)},\hdots,\Gamma_T^{(n)},\overline{\Gamma}_S^{\c^{-2n}} \rangle$.  There is an isometry $M_n \to \mathbb{H}^3/\Gamma_n$ which restricts on $C(\Gamma_S)$ and each $C(\Gamma_T^{(i)})$ to the natural map, and on $C(\overline{\Gamma}_S)$ to $\phi_{n+1}$ followed by the natural map.  \end{prop}

\begin{proof} We first recall from Lemma \ref{M_Sboundary} that the plane $\calh$ with ideal boundary $\mathbb{R}\cup\{\infty\}$ projects to $\partial C(\Gamma_S)$ under the quotient map $\mathbb{H}^3 \to \mathbb{H}^3/\Gamma_S$, so it is a component of $\partial \mathrm{Hull}(\Gamma_S)$.   Because the octahedron $\calp_1$ is contained in $\mathrm{Hull}(\Gamma_S)$ and all its ideal vertices have non-negative imaginary part, it follows that 
$$\mathrm{Hull}(\Gamma_S) \subset \{ z \in \mathbb{C}\,|\,\Im z \geq 0 \}\cup \{\infty\}.$$  
Similarly, from Lemma \ref{M_Tboundary} and the positioning of $\calp_2$ we find that 
$$\mathrm{Hull}(\Gamma_T) \subset \{ z \in \mathbb{C}\,|\,0 \geq \Im z \geq -2\sqrt{2} \}\cup \{\infty\}.$$
Then inspecting the action of $\c$ on $\mathbb{C}\cup\{\infty\}$, we find that for each $i\in\mathbb{N}$, any point of $\mathrm{Hull}(\Gamma_T^{(i)})$ has imaginary part between $-2(i-1)\sqrt{2}$ and $-2i\sqrt{2}$ for $i \in \mathbb{N}$.

The claim below builds an inductive picture of a family of isometrically embedded, codimension-0 submanifolds of $M_n$ with totally geodesic boundary. 

\begin{claim}  For $1 \leq i \leq n$, define $\Gamma^{(i)}_- = \langle \Gamma_S,\Gamma_T^{(1)},\hdots,\Gamma_T^{(i)} \rangle$.  There is an isometry $C(\Gamma_S) \cup C(\Gamma_T^{(1)}) \cup \hdots \cup C(\Gamma_T^{(i)}) \to C(\Gamma^{(i)}_-)$, where the gluing maps for the domain are as in Proposition \ref{geometric model}, which restricts on $C(\Gamma_S)$ and each $C(\Gamma^{(j)})$, $j<i$, to the natural map.  Furthermore: \begin{enumerate}
  \item $\Lambda^{(i)} = \mathrm{Stab}_{\Gamma^{(i)}_-}(\c^{-2n}(\calh))$, and the resulting natural map  $F^{(i)} \to \partial C(\Gamma^{(i)}_-)$ factors as $\iota^{(i)}_-$ followed by the natural map $C(\Gamma_T^{(i)}) \to C(\Gamma^{(i)}_-)$.
  \item $\mathrm{Hull}(\Gamma^{(i)}_-) \subset \{ z \in \mathbb{C}\,|\,\Im\ z \geq -i\sqrt{2} \} \cup\{\infty\}.$
\end{enumerate} \end{claim}

\begin{proof}[Proof of claim]  We will prove the claim by induction.  If it holds for some $i<n$, then (1) and (2) above, together with the observations above the claim imply that $\Gamma^{(i)}_-$ and $\Gamma_T^{(i+1)}$ \meetcute along $\c^{-2i}(\calh)$.  Then by Lemma \ref{Maskit}, the natural maps determine an isometry $C(\Gamma^{(i)}_-) \cup C(\Gamma_T^{(i+1)}) \to C(\Gamma^{(i+1)}_-)$, where by the inductive hypothesis and the observation above Proposition \ref{geometric model}, the gluing map for the domain is $\iota^{(i)}_+(\iota^{(i)}_-)^{-1}$ following the inverse of the natural map.

Furthermore, since $C(\Gamma^{(i)}_-)$ has a unique totally geodesic boundary component, which is isometrically identified with $\partial_- C(\Gamma_T^{(i+1)})$ in the isometry to $C(\Gamma^{(i+1)}_-)$ described above, the unique totally geodesic boundary component of $C(\Gamma^{(i+1)}_-)$ is the isometric image of $\partial_+ C(\Gamma_T^{(i+1)})$.  Therefore the observations above Proposition \ref{geometric model} imply that this boundary component is the image of $\iota^{(i+1)}_-(F^{(i+1)})$ under the natural map.  Assertion (1) of the claim thus follows for $\Gamma^{(i+1)}_-$.  It follows that $\mathrm{Hull}(\Gamma^{(i+1)}_-)$ is entirely on one side or the other of the boundary at infinity of $\c^{-2(i+1)}(\calh)$.  Since $\Gamma_T^{(i+1)} < \Gamma^{(i+1)}_-$, assertion (2) now follows.

By our definition of ``natural map'' above Lemma \ref{Maskit}, the composition of the natural map $C(\Gamma_T^{(j)}) \to C(\Gamma^{(i)}_-)$, with the natural map $C(\Gamma^{(i)}_-) \to C(\Gamma^{(i+1)}_-)$ is itself natural, for $j \leq i$.  Hence if the claim holds for $\Gamma^{(i)}_-$, $i<n$, it holds for $\Gamma^{(i+1)}_-$.  The claim will therefore hold by induction if it is true in the base case $i=1$.  But this follows from the fact that $\Gamma_S$ and $\Gamma_T^{(1)}$ \meetcute along $\calh$.  This follows in turn from Lemmas \ref{M_Sboundary} and \ref{M_Tboundary}, which establish that $\Lambda^{(0)} = \mathrm{Stab}_{\Gamma_S}(\calh)=\mathrm{Stab}_{\Gamma_T}(\calh)$, and the first paragraph of the proof.  \end{proof}

Using the claim, it now follows that $\Gamma^{(n)}_-$ and $\overline{\Gamma}_S^{\c^{-2n}}$ \meetcute along $\c^{-2n}(\calh)$; hence a final application of Lemma \ref{Maskit} implies that the natural maps determine an isometry $C(\Gamma^{(n)}_-) \cup C(\overline{\Gamma}_S^{\c^{-2n}}) \to C(\Gamma_n)$.  Since each of $C(\Gamma^{(n)}_-)$ and $C(\overline{\Gamma}_S^{\c^{-2n}})$ has a unique boundary component, $C(\Gamma_n)$ is boundaryless and hence equal to $\mathbb{H}^3/\Gamma_n$.  The conclusion of the proposition follows.
\end{proof}

The result below follows from Proposition \ref{algebraic model}, or really, its proof.

\begin{cor}\label{splitting}  For fixed $n$ and $0 \leq i \leq n$, define \begin{align*}
 & \Gamma^{(i)}_- = \Big\langle \Gamma_S,\Gamma_T^{(1)}, \hdots, \Gamma_T^{(i)} \Big\rangle &
 & \Gamma^{(i)}_+ = \Big\langle \Gamma_T^{(i+1)},\hdots,\Gamma_T^{(n)},\overline{\Gamma}_S^{\c^{-2n}} \Big\rangle,  \end{align*}
Then $\Gamma^{(i)}_+$ and $\Gamma^{(i)}_-$ \meetcute along $\c^{-2i}(\calh)$ and the natural maps determine an isometry $C(\Gamma^{(i)}_-) \cup C(\Gamma^{(i)}_+) \to \mathbb{H}^3/\Gamma_n$.  The isometry of Proposition \ref{algebraic model} factors through this map, so that the component of $M_n - F^{(i)}$ containing $C(\Gamma_S)$ is taken isometrically to its image in $C(\Gamma^{(i)}_-)$.\end{cor}

In the remainder of the paper, we will frequently take the isometry above for granted and refer to the components obtained by splitting $M_n$ along $F^{(i)}$ by $C(\Gamma^{(i)}_{\pm})$.

\section{Invariants}  \label{sec:invariants}

\subsection{Traces} \label{subsec:traces}

If $\Gamma \subset \mathrm{PSL}_2(\mathbb{C})$ is a discrete group, its trace field $\mathbb{Q}(\mathrm{tr}\,\Gamma)$ is obtained by adjoining to $\mathbb{Q}$ the traces of elements of $\Gamma$.  If the hyperbolic 3-manifold $M = \mathbb{H}^3/\Gamma$ has finite volume, Mostow rigidity implies that this is a topological invariant of $M$.  It follows from the local rigidity theorems of Garland and Prasad that in this case the trace field is a number field; ie, a finite extension of $\mathbb{Q}$.  The trace field is not generally an invariant of the commensurability class of $M$, however, and to obtain one we pass to the \textit{invariant trace field} $k\Gamma$.  This is obtained by adjoining to $\mathbb{Q}$ the traces of squares of elements of $\Gamma$.  When $M$ is the complement of a link in a $\mathbb{Z}_2$--homology sphere, its trace field and invariant trace field coincide (cf. \cite{MaR}).

\begin{prop}\label{trace fields}  $k(\Gamma_S) = \mathbb{Q}(i)$, $k(\Gamma_T) = \mathrm{Q}(i\sqrt{2})$, and $k(M_n) = \mathbb{Q}(i,i\sqrt{2})$ for all $n \in \mathbb{N}$.  In particular, $M_n$ is not arithmetic for any $n\in\mathbb{N}$.\end{prop}

\begin{proof}  Its definition in Corollary \ref{M_SandG} immediately implies $\Gamma_S < \mathrm{PSL}_2(\mathbb{Q}(i))$.  The description in Corollary \ref{M_T0andH0}, of $\Gamma_{T_0}$, and Lemma \ref{M_TandH} imply that $\Gamma_T< \mathrm{PSL}_2(\mathbb{Q}(i\sqrt{2}))$.  Thus $k\Gamma_S \subseteq \mathbb{Q}(i)$, and $k\Gamma_T \subseteq \mathbb{Q}(i\sqrt{2})$.  That equality holds is clear upon noting that $\tr({\sf h})=\pm i\sqrt{2}$ and $\tr({\sf t})=\pm (1+i)$.   Since $\Gamma_S$ and $\Gamma_T$ are in $\Gamma_n$ we have $\mathbb{Q}(i,i\sqrt{2})\subseteq k(\Gamma_n)$.  For the other containment we note that $\c$ from Lemma \ref{M_TandH} lies in $\mathrm{PSL}_2(\mathbb{Q}(i\sqrt{2}))$, and $\Gamma_n$ is contained in the group generated by $\Gamma_S$, $\Gamma_T$, and $\c$.  

It is well known that any non-compact arithmetic manifold $M$ has $k(M)\subset \mathbb{Q}(i\sqrt{d})$ for some $d\in\mathbb{N}$ (see eg. \cite[Theorem 8.2.3]{MaR}), so $M_n$ is not arithmetic.\end{proof}

We say $M = \mathbb{H}^3/\Gamma$ has \textit{integral traces} if for each $\gamma \in \Gamma$, $\mathrm{tr}\ \gamma$ is an algebraic integer.  Otherwise we say $M_n$ has a nonintegral trace.  $M$ has integral traces if and only if all manifolds commensurable to $M$ do as well (cf. \cite{MaR}).

\begin{prop} \label{prop:integraltraces}  For each $n$, $M_n$ has integral traces.  \end{prop}

\begin{proof}  As in the proposition above, this follows from the fact that each $\Gamma_n$ is contained in the group generated by $\Gamma_S$, $\Gamma_T$, and $\c$.  It is easy to see that the \textit{entries} of the generators for $\Gamma_S$ and $\Gamma_{T_0}$ are algebraic integers.  Since $\c$ has integral entries as well, all elements of $\Gamma_n$ have integral entries, hence integral traces.  \end{proof}

\begin{remark} Bass showed that if $M = \mathbb{H}^3/\Gamma$ where $\Gamma$ has an element with a nonintegral trace, there are closed essential surfaces in $M$ associated to this trace \cite{Ba}.  We say that such surfaces are \textit{detected by the trace ring}.  For fixed $n$ and $1\leq i \leq n$, closed essential surfaces in $M_n$ can be obtained by ``tubing" $S^{(i)}$ through $B^3 - L^{(i)}_-$.  More precisely, let $\mathcal{N}_i$ be a regular neighborhood of $L^{(i)}_-$ in $(B^3,L^{(i)}_-) \subset (S^3,L_n)$, let $A_i = \mathcal{N}_i \cap \overline{B^3 - \mathcal{N}_i}$, and let
$$ \hat{S}_i = \overline{S^{(i)} - (S^{(i)} \cap \mathcal{N}_i)} \cup A_i.  $$
Then $\hat{S}_i$ is a closed surface of genus two which is incompressible in $M_n$.  We will show below that certain mutants have nonintegral traces, and one easily finds surfaces analogous to $\hat{S}_i$ in the mutants.  It is interesting to note that although these surfaces are present in all of these  link complements, the trace ring does not detect any closed surfaces in the $M_n$.  \end{remark}

\subsection{Scissors congruence and the Bloch invariant} \label{subsec:bloch}

We will prove in Proposition \ref{M_n incomm} below that the \textit{Bloch invariant} distinguishes the commensurability class of $M_m$ from that of $M_n$ for $m \neq n$.  This is an invariant of a polyhedral decomposition which by construction is invariant under \textit{scissors congruence}: cutting the constituent polyhedra apart and reassembling them in new ways.  Its deep connection to  algebraic $k$-theory is what makes the Bloch invariant useful, though.  For background and an account of the connection to scissors congruence we refer the reader to \cite{Dup} and \cite{Ne}, our main source for the expository material here.

\begin{dfn}  For a field $k \subset \mathbb{C}$, define the \textit{pre-Bloch group} $\mathcal{P}(k)$ to be the quotient of the free $\mathbb{Z}$--module on $k - \{0,1\}$ by all instances of the following relations.  \begin{align}  \label{fiveterm}
&  [x] - [y] + \left[ \frac{y}{x} \right] - \left[ \frac{1-x^{-1}}{1-y^{-1}} \right] + \left[ \frac{1-x}{1-y} \right] = 0, \ \ x \neq y \in k -\{0,1\} \\  \label{permute}
&  [z] = \left[ 1- \frac{1}{z} \right] = \left[\frac{1}{1-z}\right] = -\left[ \frac{1}{z} \right] = - \left[ \frac{z}{z-1} \right] = -[1-z], \ \ z \in k - \{0,1\}  \end{align}
There is a map $\delta \co \mathcal{P}(k) \rightarrow k^* \wedge k^*$ given by $[z] \mapsto 2(z \wedge (1-z))$.  (Here $k^*$ is considered a $\mathbb{Z}$--module with multiplication as the group operation and $\mathbb{Z}$--action given by $a.x = x^a$, $a \in \mathbb{Z}$.)  The \textit{Bloch group} is $\mathcal{B}(k) = \ker \delta$.  \end{dfn}

\begin{remark}  If $k$ is algebraically closed, the relation (\ref{fiveterm}) above, called the \textit{five term relation}, implies (\ref{permute}).  For instance, taking $\sqrt{z}$ and $\sqrt{z^{-1}}$ as $x$ and $y$, respectively, in (\ref{fiveterm}), then interchanging their roles and summing the results yields $[z] + [1/z] = 0$.  \end{remark}

For any ideal tetrahedron $T$ in $\mathbb{H}^3$, there is an orientation-preserving isometry of $\mathbb{H}^3$ taking its ideal vertices to $0$, $1$, $\infty$ and a complex number $z$ with non-negative imaginary part.  Let the \textit{cross ratio parameter} of $T$ be $[z]\in\mathcal{P}(\mathbb{C})$.  This is well-defined because any other isometry answering the description above fixes $z$ or replaces it by one of $1-\frac{1}{z}$ or $\frac{1}{1-z}$.

For $k'\subset k$, inclusion induces a map $\mathcal{P}(k')\to\mathcal{P}(k)$.  Although this is not injective in general, a theorem of Borel that we record below implies that if $k'$ is a number field then $\mathcal{B}(k')$ does inject, modulo torsion.  We offer this observation to excuse occasional imprecision about the precise location of our invariants.

\begin{dfn}  Let $M = T_1 \cup \hdots \cup T_n$ be a triangulated complete, orientable hyperbolic $3$-manifold of finite volume (with or without boundary); ie, with each $T_i$ isometric to an ideal hyperbolic tetrahedron and $T_i \cap T_j$ either empty, an edge of each, or a face of each for $i \neq j$.  Define the \textit{Bloch invariant} of $M$ as
$$ \beta(M) = [z_1] + [z_2] + \hdots + [z_n] \in \mathcal{P}(\mathbb{C}),  $$
where $[z_i]$ is the cross ratio parameter of $T_i$ for each $i$ in $\{1,\hdots,n\}$.\end{dfn}

\begin{remark}  If $\partial M = \emptyset$ then $\beta(M)\in\mathcal{B}(\mathbb{C})$ by a geometric interpretation of the Bloch invariant, and by work of Neumann--Yang \cite{NeYang} it does not vary with triangulation.  \end{remark}

We will obtain a triangulation of $M_n$ by subdividing the decomposition below.

\begin{lemma}\label{poly decomp} The members of $\cals = \{\calp_1,\calp_2,\c^{-1}\calp_2,\hdots,\c^{-2n+1}\calp_2,\c^{-2n}\r\calp_1\}$ project under $\mathbb{H}^3\to\mathbb{H}^3/\Gamma_n$ to the cells of an ideal polyhedral decomposition of $M_n$.\end{lemma}

\begin{proof}  By Corollary \ref{M_SandG}, $\calp_1$ projects under $\mathbb{H}^3\to\mathbb{H}^3/\Gamma_S$ to an ideal polyhedral decomposition of $C(\Gamma_S)$: it maps onto $C(\Gamma_S)$ with internal faces identified in pairs.  Corollary \ref{M_T0andH0} implies the same for $\calp_2\to C(\Gamma_{T_0})$ under $\mathbb{H}^3\to\mathbb{H}^3/\Gamma_{T_0}$, and hence also for $\c^{-2}\r\calp_2\to C(\overline{\Gamma}_{T_0}^{c^{-2}})$ (cf.~the paragraph above Lemma \ref{M_TandH}).

It is easy to see that $\r\calp_2 = \c\calp_2$, for instance by comparing sets of ideal vertices, so $\c^{-2}\r\calp_2 = \c^{-1}\calp_2$.  Therefore Lemma \ref{M_TandH} implies that $\calp_2\cup\c^{-1}\calp_2$ projects to an ideal polyhedral decomposition of $C(\Gamma_T)$ under $\mathbb{H}^3\to\mathbb{H}^3/\Gamma_T$.  In particular, this projection identifies the external faces of $\calp_2$ that map to $F'$ with external faces of $\c^{-1}\calp_2$ pairwise, since their images are fixed by the doubling involution $\phi_{\c^{-2}\r}$.

It follows from the above that $\c^{-2(i-1)}\calp_2\cup\c^{-2i+1}\calp_2$ projects to a decomposition of  $C(\Gamma_T^{(i)})$ for any $i\in\mathbb{N}$ (recall Definitions \ref{geometric pieces}), and from the first paragraph that the same holds for $\c^{-2n}\r\calp_1\to C(\overline{\Gamma}_S^{\c^{-2n}})$.  By Proposition \ref{algebraic model}, it remains only to show that the gluings producing $M_n$ preserve induced triangulations of boundaries.  These are defined in Proposition \ref{geometric model}.  Lemma \ref{M_Sboundary}(3) implies that $\iota_+^{(0)}(\iota_-^{(0)})^{-1}$ preserves triangulations, and (\ref{reflection i's}) from the proof of Proposition \ref{geometric model} does the same for $\iota_+^{(i)}(\iota_-^{(i)})^{-1}$ for $1\leq i\leq n$.  They combine to imply that the final map does as well.\end{proof}

\begin{lemma}\label{triangulations}  $M_n$ has Bloch invariant $\beta_1-\bar{\beta}_1+n\beta_2\in\mathcal{B}(\mathbb{Q}(i,i\sqrt{2}))$ for any $n$, where $\beta_1 = 4\left[\frac{1+i}{2}\right]\in\mathcal{P}(\mathbb{Q}(i))$, $\bar{\beta}_1 = 4\left[\frac{1-i}{2}\right]$, and $\beta_2\in\mathcal{P}(\mathbb{Q}(i\sqrt{2}))$.\end{lemma}

\begin{proof}  We will produce a triangulation of $M_n$ by subdividing the polyhedral decomposition from Lemma \ref{poly decomp}.  $\mathcal{P}_1$ divides into a collection of $4$ tetrahedra by the addition of a single edge $\gamma$ joining the ideal vertices $(1+i)/2$ and $\infty$, and four ideal triangular faces that share $\gamma$.  One has ideal vertices $0$, $1$, $\infty$, and $(1+i)/2$ and thus a parameter of $\left[\frac{1+i}{2}\right]$.  Since the others are its image under rotation about $\gamma$ they have identical cross ratio parameters.  Their union projects to a triangulation of $C(\Gamma_S)$ with Bloch invariant $\beta_1 = 4\left[\frac{1+i}{2}\right]$.

Any ideal tetrahedron with its vertex set contained in that of $\calp_2$ has cross ratio parameter in $\mathcal{P}(\mathbb{Q}(i\sqrt{2}))$, since $\calp_2$ has ideal vertices in $\mathbb{Q}(i\sqrt{2})\cup\{\infty\}$.  We leave it to the reader to divide $\calp_2$ into ideal tetrahedra in such a way that the resulting division of square faces, each into two ideal triangles, is preserved by the face-pairings that produce $M_{T_0}$.  Such a triangulation projects to one of $C(\Gamma_{T_0})$, and its image under $\c^{-2}\r$ projects to one of $C(\Gamma_{T_0})$.

Above it is important to use $\c^{-2}\r$ and not $\c^{-1}$, since the face pairings of $\c^{-1}\calp_2$ project it to $C(\overline{\Gamma}_{T_0}^{\c^{-2}})$.  Recall that $\r$ is a reflection, extending to $\mathbb{C}$ as complex conjugation.  One checks using (\ref{permute}) that if a tetrahedron has cross ratio parameter $[z]$ then its mirror image has parameter $-[\bar{z}]$.  Since $\mathbb{Q}(i\sqrt{2})$ is preserved by complex conjugation, using the triangulations from the paragraph above gives $C(\Gamma_T)$ a Bloch invariant $\beta_2\in\mathcal{P}(\mathbb{Q}(i\sqrt{2}))$.

For each $i$ with $1\leq i\leq n$, $C(\Gamma_T^{(i)})$ inherits a triangulation with Bloch invariant $\beta_2$ from $\c^{-2(i-1)}\calp_2\cup\c^{-2i+1}\calp_2 = \c^{-2(i-1)}(\calp_2\cup\c^{-1}\calp_2)$, and $C(\overline{\Gamma}_S)$ inherits one with invariant $\bar{\beta_1}$ from $\r(\calp_1)$.  Lemma \ref{poly decomp} implies that these combine to triangulate $M_n$, so its Bloch invariant is as described above.\end{proof}

Below we record a standard formula for the \textit{Bloch-Wigner dilogarithm function} $D_2 \co \mathbb{C} - \{0,1\} \rightarrow \mathbb{R}$ in terms of the \textit{dilogarithm}, $\psi(z) = \sum_{i=1}^{\infty} \frac{z^n}{n^2}\ (\mbox{for}\ |z|<1)$:
$$  D_2(z) = \Im \psi(z) + \log |z| \arg (1-z)  $$
For $z$ in the upper half plane, the ideal tetrahedron with ideal vertices $0$, $1$, $\infty$, and $z$ has volume $D_2(z)$; note also that $D_2(\bar{z}) = -D_2(z)$.  $D_2$ determines a well-defined functional on $\mathcal{P}(\mathbb{C})$, and this in turn produces the \textit{Borel regulator}, the map below:

\begin{thm1}[Borel, \cite{Borel}]  For a number field $k$ fix embeddings $\sigma_1, \hdots, \sigma_{r_2}$ to $\mathbb{C}$, one representing each complex-conjugate pair.  The map $\mathrm{B}_k\co\mathcal{P}(k) \rightarrow \mathbb{R}^{r_2}$ extending $[z] \mapsto (D_2(\sigma_1(z)), \hdots, D_2(\sigma_{r_2}(z)) )$ takes $\mathcal{B}(k)$ onto a lattice in $\mathbb{R}^{r_2}$, with kernel consisting entirely of torsion elements.  \end{thm1}

We use the Borel regulator $\mathrm{B}_k$ to show that Bloch invariants distinguish the commensurability class of $M_m$ from that of $M_n$ for $m\neq n$.

\begin{prop}\label{M_n incomm} For $m \neq n$, $M_m$ is not commensurable with $M_n$.  \end{prop}

\begin{remark}  We thank the referee on an earlier version of this paper for describing the argument below.  (Our original proof used cusp parameters; cf.~Lemma \ref{PGL action}.)\end{remark}

\begin{proof}  It is clear that $k=\mathbb{Q}(i,i\sqrt{2})$ has two pairs of complex conjugate embeddings, each determined by its action on $i$ and $i\sqrt{2}$.  We will take $\sigma_1 = \mathit{id}_k$, and $\sigma_2(i) = i$, $\sigma_2 (i\sqrt{2}) = -i\sqrt{2}$, in defining the Borel regulator $\mathrm{B}_k$ on $k$.  Since each $\sigma_i$ restricts on $\mathbb{Q}(i)$ to the identity, $\mathrm{B}_k$ takes each of $\beta_1$ and $-\bar{\beta}_1$ to $(v_1,v_1)\in\mathbb{R}^2$, where $v_1$ is the volume of $\calp_1$.  On the other hand, $\mathrm{B}_k(\beta_2)=2(v_2,-v_2)$, where $v_2 = \mathrm{vol}(\calp_2)$.

For any $n$, a covering space $\widetilde{M} \rightarrow M_n$ of degree $k$ has $\beta(\widetilde{M}) = k \beta(M_n)$.  This is because the preimage in $\widetilde{M}$ of each tetrahedron $T$ from the triangulation of $M_n$ described in Lemma \ref{triangulations} is a non-overlapping union of $k$ isometric copies of $T$.  Thus if $\widetilde{M}\to M_m$ with degree $p$ and $\widetilde{M}\to M_n$ with degree $q$ it would follow that 
$$ p \left[ \beta_1-\bar{\beta}_1 + m \beta_2 \right] = q\left[ \beta_1-\bar{\beta}_1 + n \beta_2\right].  $$
Applying $\mathrm{B}_k$ to each side of the equation above, we find that since $(v_1,v_1)$ and $(v_2,-v_2)$ are linearly independent in $\mathbb{R}^2$ we must have $p =q$ and $m = n$.
\end{proof}


\subsection{Cusp parameters}  \label{subsec:param}

Following Neumann--Reid \cite[\S 2.3]{NeR2}, for a cusp of a complete hyperbolic $3$-manifold $M$ we will call the \textit{cusp parameter} the complex modulus (aka conformal parameter) of a horospherical cusp cross-section, a Euclidean torus.  Thurston also used this invariant to distinguish hyperbolic manifolds \cite[Ch.~6]{Th}.

\begin{dfn}  Let $T = \mathbb{C}/\Lambda$ be a Euclidean torus, where $\Lambda \subset \mathbb{C}$ is a lattice.  Define the \textit{complex modulus} of $T$ as $m(T) = \alpha/\beta$, where $\Lambda = \langle \alpha, \beta \rangle$.  \end{dfn}

\begin{remark}  The complex modulus is not really an invariant of a Euclidean torus; rather, it is an invariant of a particular basis for $\pi_1$.  However, we have:\end{remark}

\begin{lemma}\label{commensurable moduli}  The $\mathrm{PGL}_2(\mathbb{Z})$-orbit of the complex modulus is a similarity invariant of Euclidean tori.  The $\mathrm{PGL}_2(\mathbb{Q})$-orbit is a commensurability invariant.  \end{lemma}

Here we say $T$ and $T'$ are \textit{commensurable} if $T$ has a finite cover which is similar to a cover of $T'$.

\begin{proof}  The complex modulus is clearly scale-invariant.

Let $T=\mathbb{C}/\Lambda$ be a Euclidean torus, where $\Lambda=\langle\alpha,\beta\rangle$.  For a different generating pair $\gamma = p\alpha + q\beta,\ \delta = r\alpha + s\beta$ the change of basis matrix $\sfm=\left( \begin{smallmatrix} p & r \\ q & s \end{smallmatrix} \right)\in\mathrm{PSL}_2(\mathbb{Z})$ has an inverse there as well, since $\alpha$ and $\beta$ are linear combinations of $\gamma$ and $\delta$.  Computing the modulus with $\gamma$ and $\delta$ yields 
$$ \frac{p\alpha + q\beta}{r\alpha + s\beta} = \frac{p(\alpha/\beta) + q}{r(\alpha/\beta) + s} = \sfm^{\mathrm{T}}(m(T)).  $$
If $\gamma$ and $\lambda$ generate a finite-index sub-lattice then since they are linearly independent, $\sfm$ has non-zero determinant.  This implies the commensurability-invariance assertion.\end{proof}

It will prove useful here to understand the complex modulus of a torus by decomposing it into annuli using a family of parallel geodesics.

\begin{dfn}  For a Euclidean annulus $A$ with core of length $\ell$ and distance $d$ between geodesic boundary components, let the \textit{real modulus} of $A$ be $m(A) = d/\ell$.  \end{dfn}

If $T = \mathbb{C}/\Lambda$, and $\Lambda = \langle \alpha, \beta \rangle$, then $\alpha$ and $\beta$ determine isotopy classes of  simple closed geodesics on $T$ with representatives which intersect once.    These are the projections to $T$ of the line segments in $\mathbb{C}$ joining $0$ to $\alpha$ and $\beta$, respectively.  Below let $A_{\beta}$ denote the Euclidean annulus with geodesic boundary obtained as the path completion of the metric on $T-\beta$ inherited from $T$.

\begin{lemma}  \label{mod_decomp} Let $T = C/\Lambda$ be a Euclidean torus, and suppose $\alpha,\ \beta$ is a generating pair for $\Lambda$.  Decompose $m(T)$ into real and imaginary parts:
$$m(T) =  \tau_{\beta} + i \cdot \mu_{\beta},$$ 
where $\tau_{\beta} = \Re (\alpha/\beta)$ and $\mu_{\beta} = \Im (\alpha/\beta) \in \mathbb{R}$.  Then $\tau_{\beta} = \frac{\|\alpha\|}{\|\beta\|} \cos \theta$, where $\theta$ is the angle between the geodesics $\alpha$ and $\beta$ on $T$, and $|\mu_{\beta}| = m(A_{\beta})$.  \end{lemma}

\begin{proof}  Write $\alpha = \|\alpha\| e^{i \theta_1}$ and $\beta = \|\beta\| e^{i\theta_2}$.  Then $\theta = \theta_1 - \theta_2$ is the angle between the geodesics corresponding to $\alpha$ and $\beta$, and $\frac{\alpha}{\beta} = \frac{\|\alpha\|}{\|\beta\|} e^{i\theta}$.  Writing $e^{i\theta} = \cos \theta + i\sin \theta$ yields the first assertion immediately.

To establish the second, consider the strip $\widetilde{A}_{\beta}$ in $\mathbb{C}$ bounded by the line containing $0$ and $\beta$ and its translate by $\alpha$, containing $\alpha$ and $\alpha + \beta$.  The quotient of $\widetilde{A}_{\beta}$ induced by the action of $\beta$ is the universal covering $\widetilde{A}_{\beta} \rightarrow A_{\beta}$.  The distance between boundary components of $\widetilde{A}_{\beta}$ is $\|\alpha\| |\sin \theta|$, and the length of the core of $A_{\beta}$ is the translation length of $\beta$, which is $\|\beta\|$.  \end{proof}

Lemma \ref{mod_decomp} provides a convenient means for understanding the modulus of a Euclidean torus in terms of ``Fenchel-Nielsen" coordinates $(\mu_{\beta}, \tau_{\beta})$ associated to a simple closed geodesic $\beta$.  We regard $\mu_{\beta}$ as a length parameter for the annulus $A_{\beta}$, and $\tau_{\beta}$ as a twist parameter.

\begin{lemma}  \label{lem:annulus_sum}  Suppose $T$ is a Euclidean torus decomposed into annuli $A_1, \hdots, A_n$ by simple closed geodesics parallel to $\beta$.  Then 
$$|\mu_{\beta}| = m(A_1) + m(A_2) + \hdots +m(A_n).$$  \end{lemma}

\begin{proof}  By isotoping $\beta$ if necessary, we may assume that it is one of the geodesics determining the $A_i$; hence $A_{\beta} = A_1 \cup A_2 \hdots \cup A_n$.  Then if $\alpha_0$ is an arc perpendicular to $\partial A_{\beta}$, joining one component to the other, for each $i$, $\alpha_0 \cap A_i$ is an arc perpendicular to $\partial A_i$ joining one component to the other.  This is because $\partial A_i$ is parallel to $\beta$.  Since $\ell(\alpha_0) = \sum_i \ell(\alpha_0 \cap A_i)$ and the core of each $A_i$ has length $\ell(\beta)$, the result follows.  \end{proof}

The annuli we are concerned with arise as horospherical cross sections of the cusps of $M_S$ and $M_T$.  Recall from Lemma \ref{M_Sboundary} that $\mathrm{Stab}_{\Gamma_S}(\calh)$ is a group $\Lambda$ generated by parabolic isometries $\p_1$, $\p_2$ and $\p_3$.  Furthermore, as pointed out in Remark 1 below the lemma, $\p_1$ and $\p_3$ are conjugate in $\Gamma_S$, as are $\p_2$ and $\p_4 = \p_1 \p_2 \p_3^{-1}$.  We asserted there that $C(\Gamma_S)$ has two cusps, one corresponding to $\p_1$ and one to $\p_2$.  This follows from the lemma below.  

In what follows, we let $\calv_1= \{ \infty, 0, 1,i, 1+i, (1+i)/2 \}$, the set of ideal vertices of the ideal octahedron $\calp_1$.  Let $\{h_v\,|\, v \in \mathcal{V}_1\}$ be a collection of horospheres invariant under the action of the symmetry group of $\mathcal{P}_1$, such that $h_v$ is centered at $v$ for each $v \in \mathcal{V}_1$ and $h_{\infty}$ is at height $2$.  

\begin{lemma}  \label{lem:M_S_moduli} The projection to $M_{S}$ of $\bigcup \left(h_v \cap \calp_1\right)$ is a disjoint union of Euclidean annuli $A_1$ and $A_2$ with geodesic boundary, such that $p_S(A_1)$ is a horospherical cross section of the cusp of $C(\Gamma_S)$ corresponding to $\p_1$, $p_S(A_2)$ is a cross section of the cusp corresponding to $\p_2$, and $m(A_1) = 1$, $m(A_2) = 1/5$.  \end{lemma}

\begin{proof}  Since $h_{\infty}$ is at height $2$ and our embedding of $\calp_1$ is as in Figure \ref{fig:idealoct}, $h_{\infty} \cap \mathcal{P}_1$ is a square with sides of length $1/2$.  Since the symmetry group of $\mathcal{P}_1$ acts transitively on vertices, this holds for all $h_v \cap \mathcal{P}_1$, $v \in \mathcal{V}_1$.  We will call a side of $h_v \cap \mathcal{P}_1$ \textit{internal} if it is contained in an internal face of $\mathcal{P}_1$ and \textit{external} otherwise.  The face--pairing $\s$ has the property that if $v$ and $v'$ are ideal vertices of $\mathcal{P}_1$ and $\s(v) = v'$, then $\s(h_v)  = h_{v'}$, and $\s(h_v \cap \mathcal{P}_1)$ abuts $h_{v'} \cap \mathcal{P}_1$ along an internal side.  The analogous property holds for $\t$.  

\begin{figure}
\begin{center}
\input{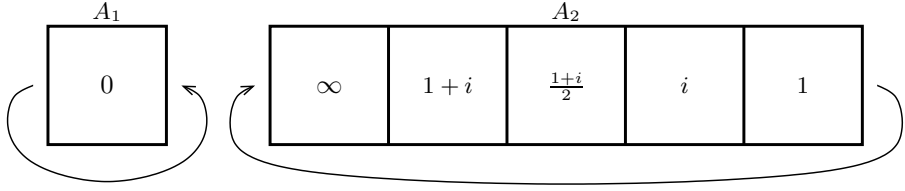}
\end{center}
\caption{Cross sections of the cusps of $M_S$}
\label{fig:M_Scusp}
\end{figure}

Each of $\s$ and $\t$ identifies a pair of internal faces of $\mathcal{P}_1$, yielding $M_S$.  The isometry $p_S$ of Corollary \ref{M_SandG} is induced by the inclusion $\mathcal{P}_1 \rightarrow \mathbb{H}^3$.  Since $\p_1 = \s^{-1}$ fixes the ideal vertex of $\mathcal{P}_1$ at $0$, it identifies the opposite internal sides of $h_0 \cap \mathcal{P}_1$.  This square thus projects to a cusp cross section $A_1$ of $M_S$, mapped by $p_S$ to one of the cusp of $C(\Gamma_S)$ corresponding to $\p_1$.  This is depicted on the left side of Figure \ref{fig:M_Scusp}. 

The other cusp cross section of $M_S$, the annulus $A_2$, is the identification space of the collection 
$$\{h_v \cap \mathcal{P}_1\,|\, v \in \mathcal{V}_1 - \{0\}\}$$ 
shown on the right side of Figure \ref{fig:M_Scusp}.  In the figure, each square is the projection to $M_S$ of $h_v \cap \mathcal{P}_1$ for the ideal vertex $v$ by which it is labeled.  The combinatorics can be verified by considering the action of $\s$ and $\t$ on $\mathcal{V}_1$.

By assumption each square in Figure \ref{fig:M_Scusp} has side length $1/2$, and so the cores of $A_1$ and $A_2$ have lengths $1/2$ and $5/2$, respectively.  For any square in Figure \ref{fig:M_Scusp}, a vertical side projects to an arc joining the distinct boundary components of the corresponding $A_i$, hence the distance between them is $1/2$.  Thus it follows directly from the definition that $m(A_1) = 1$ and $m(A_2) = 1/5$. \end{proof}

The lemma below describes the moduli of the cusps of $C(\Gamma_{T_0})$.  We asserted below Lemma \ref{M_T0boundary} that $C(\Gamma_{T_0})$ has four cusps, one corresponding to each $\p_i$, $i \in \{1,2,3,4\}$, and each joining $\partial_- C(\Gamma_{T_0})$ to $\partial_+ C(\Gamma_{T_0})$.  This follows from Lemma \ref{lem:M_T_moduli} below.  Let $\mathcal{V}_2$ be the set of ideal vertices of $\mathcal{P}_2$, and consider a collection of horospheres $\{ h_v \,|\,v \in \mathcal{V}_2 \}$, invariant under the symmetry group of $\mathcal{P}_2$, such that $h_v$ is centered at $v$ for each $v \in \mathcal{V}$ and $h_{\infty}$ is at height $2$.  

\begin{lemma}  \label{lem:M_T_moduli}  The projection of $\bigcup \left( h_v \cap \calp_2 \right)$ to $M_{T_0}$ is a collection of disjoint Euclidean annuli $B_j$ with geodesic boundary, $j \in \{1,2,3,4\}$, such that $p_{T_0}(B_j)$ is a cross section of the cusp of $C(\Gamma_{T_0})$ corresponding to $\p_j \in \Lambda$, and $m(B_1) = m(B_3) = \sqrt{2}$, $m(B_2) = m(B_4) = \sqrt{2}/5$.  \end{lemma}

\begin{figure}
\begin{center}
\input{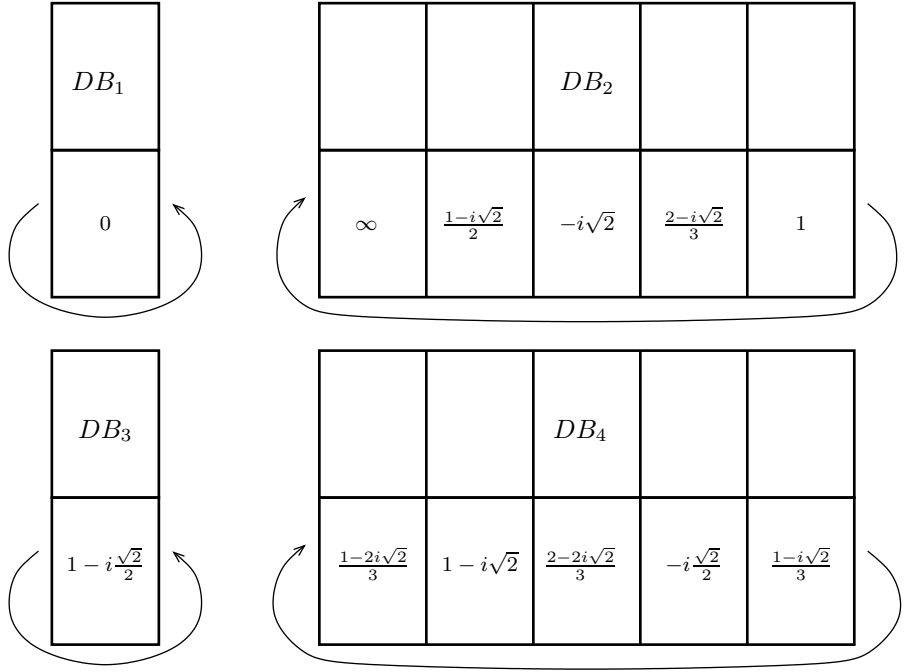}
\end{center}
\caption{Cross sections of the cusps of $M_T$}
\label{fig:M_Tcusp}
\end{figure}

\begin{proof}  For $v \in \calv_2$, we again call a side of $h_v \cap \mathcal{P}_2$ external if it is contained in an external face of $\calp_2$ and internal otherwise.  Each cusp cross section of $M_{T_0}$ is the projection of a subcollection of the $h_v \cap \mathcal{P}_2$, identified along their internal faces.  From Figure \ref{fig:cuboct}, we find that $h_{\infty} \cap \mathcal{P}_2$ is a Euclidean rectangle with two opposite internal sides and two external.  Since the symmetry group of $\mathcal{P}_2$ is transitive on its set of ideal vertices, this holds for the other $h_v$ as well.  It follows that each cusp cross section of $M_{T_0}$ is a Euclidean annulus with geodesic boundary.

In Figure \ref{fig:M_Tcusp}, the lower rectangles of each annulus $DB_j$ are labeled by vertices $v$ such that $h_v \cap \mathcal{P}_2$ projects to a subrectangle of cross section of the cusp of $M_{T_0}$ whose image under $p_{T_0}$ corresponds $\p_j$.  Then $B_j$ is the lower half of $DB_j$.  The reasons for this picture will become clear after the current proof.

The isometries $\f$, $\g$, and $\h$ defined in Corollary \ref{M_T0andH0} identify the internal sides of $\mathcal{P}_2$ in pairs, yielding the manifold $M_{T_0}$ with totally geodesic boundary.  The parabolic $\p_1 = \f^{-1}$ fixes $0$, identifying the internal sides of $\mathcal{P}_2$ sharing this ideal vertex.  Thus in $M_{T_0}$, $B_1$ consists of $h_0 \cap \mathcal{P}_2$ with its internal sides identified.  The description of the $\p_i$ in terms of $\f$, $\g$, and $\h$ above Lemma \ref{M_Sboundary} shows that $\p_3$ is a conjugate of $\g^{-1}$.  Since $\g^{-1}$ fixes $1-i\frac{\sqrt{2}}{2}$, identifying the internal edges of $\mathcal{P}_2$ which abut it, $h_{1-i\sqrt{2}/2}$ projects to $B_3$ in $M_{T_0}$.  This justifies the depictions of $B_1$ and $B_3$ in Figure \ref{fig:M_Tcusp}.  

Since $\p_2$ fixes $\infty$, $h_{\infty} \cap \mathcal{P}_2$ projects to a subrectangle of $B_2$.  Since $\g$ takes the internal side $Y_3$ to $Y_1'$  and $\infty$ to $(1-i\sqrt{2})/2$, in $B_2$ the projection of $h_{\infty} \cap \mathcal{P}_2$ meets the projection of $h_{(1-i\sqrt{2})/2} \cap \mathcal{P}_2$ along a side contained in the projection of $Y_3$ to $M_{T_0}$.  Since the internal face of $\calp_2$ meeting $Y_1'$ at $(1-i\sqrt{2})/2$ is $Y_3'$, and this is taken to $Y_2'$ by $\h^{-1}$, the rectangle meeting the projection of $h_{(1-i\sqrt{2})/2}$ in $B_2$ on the internal side opposite its intersection with $h_{\infty}$ is $h_{-i\sqrt{2}}$.  Carrying this line of argument to completion yields the depictions of $B_2$ and $B_4$ in the figure.  

From Figure \ref{fig:cuboct}, we find that the internal sides of $h_{\infty} \cap \mathcal{P}_2$ have length $\sqrt{2}/2$ and the external sides length $1/2$.  Since the symmetry group of $\mathcal{P}_2$ is transitive on its ideal vertices, the same holds for each rectangle $h_v \cap \mathcal{P}_2$.  Thus the cores of $B_1$ and $B_3$ have length $1/2$, and the cores of $B_2$ and $B_4$ have length $5/2$.  For any square $h_v \cap \calp_2$, an internal side projects to a perpendicular arc joining opposite sides of the cusp cross section in $M_{T_0}$ containing $h_v \cap \calp_2$.  The moduli are thus as described.
\end{proof}

By Corollary \ref{the real M_T}, $p_{T_0}\co M_{T_0}\to C(\Gamma_{T_0})$ determines a reflection-invariant map from the double $M_T$ of $M_{T_0}$ across $\partial_+ M_{T_0}$ to $C(\Gamma_T)$.  Furthermore, as we remarked below Lemma \ref{M_T0boundary}, each cusp of $C(\Gamma_{T_0})$ joins one component of $\partial C(\Gamma_{T_0})$ to the other.  Therefore taking $DB_j \subset M_T$, $j \in \{1,2,3,4\}$, to be the double of $B_j$ across its component of intersection with $\partial_+ M_{T_0}$, we have:

\begin{lemma}\label{real M_T moduli}  For each $j\in \{1,2,3,4\}$, the image in $C(\Gamma_T)$ of $DB_j$ is a cross section of the cusp corresponding to $\p_j$, and $m(DB_1) = m(DB_3) = 2\sqrt{2}$, $m(DB_2) = m(DB_4) = 2\sqrt{2}/5$.  \end{lemma}

It is a well known consequence of Margulis' lemma that each cusp $C$ of a hyperbolic manifold $M = \mathbb{H}^3/\Gamma$ of finite volume is foliated by similar Euclidean tori, the projections to $M$ of horospheres in $\mathbb{H}^3$ centered at the fixed point of a parabolic subgroup of $\Gamma$ corresponding to $C$.

\begin{dfn}  The \textit{parameter} of a cusp $C$ of a finite--volume complete hyperbolic manifold is the complex modulus of a horospherical cross--section of $C$.  \end{dfn}

By Lemma \ref{commensurable moduli} the $\mathrm{PSL}_2(\mathbb{Z})$ orbit of the cusp parameter is an invariant of the \textit{cusp shape}, the Euclidean similarity class of a cross-section.

\begin{prop} \label{M_nmoduli}  For $j=1,2$, let $T_j$ be a cusp cross section of $M_n$ containing the annular cusp cross-section $A_j$ of $M_S$ (cf.~Lemma \ref{lem:M_S_moduli}).  Then $m(T_1) = i(2+4n\sqrt{2}),$ and $m(T_2)$ is $\mathrm{PGL}_2(\mathbb{Q})$-equivalent to $m(T_1)$.  \end{prop}

\noindent {\it Remarks.}\ \ 
1.~It is not hard to show that $m(T_2)=i(2+4n\sqrt{2})/5$, but this is not necessary for our purposes and requires more work. \smallskip \\
  2.~The cusps $T_1$ and $T_2$ are labeled in Figure \ref{linkspic}.  
  

\begin{proof}  By Proposition \ref{geometric model} $M_n = C(\Gamma_S) \cup C(\Gamma^{(1)}_T) \cup \hdots \cup C(\Gamma^{(n)}_T) \cup C(\overline{\Gamma}_S)$,
with gluing maps that factor through the inclusion induced isometries $\iota_{\pm}^{(i)}$ defined on $F^{(i)}$ for $0 \leq i \leq n-1$, and final gluing $\phi_{\r}\iota_-^{(0)}\phi_{n+1}^{-1}(\iota_-^{(n)})^{-1}\co\partial_+ C(\Gamma_T^{(n)})\to\partial C(\overline{\Gamma}_S)$.

For $j\in\{1,2,3,4\}$ and $i \in \mathbb{N}$, define $DB_j^{(i)} =\phi_i \circ p_T(DB_j) \subset C(\Gamma_T^{(i)})$, with $\phi_i$ as in Definitions \ref{geometric pieces}.  We also refer by  $DB_j^{(i)}$ to its image in $C(\Gamma_n)$ under the natural map, or in $M_n$ under inclusion.  Let $\partial_{\pm} DB_j^{(i)} = DB^{(i)}_j \cap \partial_{\pm} C(\Gamma_T^{(i)})$.   

By Lemma \ref{lem:M_S_moduli}, $p_S(A_1)$ is a cross section of the cusp of $C(\Gamma_S)$ corresponding to $\p_1$, and by Lemma \ref{real M_T moduli}, $DB_1^{(1)}$ is a cross section of the cusp of $C(\Gamma_T)$ corresponding to $\p_1$.  Lemma \ref{M_Sboundary} thus implies that $\iota^{(0)}_+(\iota^{(0)}_-)^{-1}$ takes one component of $p_S(\partial A_1)$ to $\partial_- DB_1^{(0)}$.  In Remark 1 below that lemma, we note that $\p_1$ and $\p_3$ are conjugate in $\Gamma_S$.  It follows that the other component of $p_S(\partial A_1)$ is a cross section of the cusp of $\partial C(\Gamma_S)$ corresponding to $\p_3$, so $\iota^{(0)}_+(\iota^{(0)}_-)^{-1}$ takes this component to $\partial_- DB_3^{(1)}$.

The doubling involution of $M_T$ preserves $DB_j$ by construction, exchanging its boundary components.  Therefore by Corollary \ref{the real M_T}, $\phi_{\c^{-2}\r}$ preserves $p_T(DB_j)$ and exchanges boundary components.  It follows that $\iota^{(i)}_+(\iota^{(i)}_-)^{-1}$ takes $\partial_+ DB^{(i)}_j$ to $\partial_- DB^{(i+1)}_j$ for each $i$ between $1$ and $n-1$, upon recalling the identity (\ref{reflection i's}) from the proof of Proposition \ref{geometric model}:  
\begin{align*}
  \iota^{(i)}_+(\iota^{(i)}_-)^{-1} = \phi_{i+1} \iota^{(0)}_+ \phi_{2}^{-1} \left(\iota_-^{(1)}\right)^{-1} \phi_{i}^{-1} = \phi_{i+1} \phi_{\c^{-2}\r} \phi_i^{-1}.  \end{align*} 

One finds that $\phi_{\r}\iota_-^{(0)}\phi_{n+1}^{-1}(\iota_-^{(n)})^{-1}$ takes $\partial_+ DB_1^{(n)}\sqcup\partial_+ DB_3^{(n)}$ to the components of $\phi_{\r}\circ p_S(\partial A_1)$, arguing as above and applying (\ref{model i's}) from Proposition \ref{geometric model}. Therefore $T_1$ is decomposed by its intersection in $M_n$ with the separating spheres $F^{(i)}$ into the following collection of Euclidean annuli with geodesic boundary.
$$  p_S(A_1) \cup DB_1^{(1)} \cup \hdots \cup DB_1^{(n)} \cup \phi_{\r}\circ p_S(A_1) \cup DB_3^{(n)} \cup \hdots DB_3^{(1)} $$
Similarly, we find that $T_2$ decomposes into the union of $p_S(A_2)$, $\phi_{\r}\circ p_S(A_2)$, and $DB^{(i)}_j$ for $1 \leq i \leq n$ and $j = 2,4$.  We may take $\beta_1$ to be the geodesic $\partial_- DB_1^{(1)}$ on $T_1$ and $\beta_2 = \partial_- DB_2^{(1)} \subset T_2$.  Then we obtain the following from Lemma \ref{lem:annulus_sum}, applying lemmas \ref{lem:M_S_moduli} and \ref{real M_T moduli}.
\begin{align*}  \Im(m(T_1)) = \pm (2+4n\sqrt{2}) & & \Im(m(T_2)) = \pm \frac{2+4n\sqrt{2}}{5} \end{align*}

We will show $m(T_1)$ and $m(T_2)$ have real part equal to $0$ by describing geodesics $\alpha_j$, $j = 1,2$, which meet the $\beta_j$ once, perpendicularly.  Let $a_1$ be the arc in $A_1$ which is the  projection of the internal edges of $h_0 \cap \mathcal{P}_1$ (the vertical arcs on the left-hand square in Figure \ref{fig:M_Scusp}).  Recall that the internal edges of $h_0 \cap \mathcal{P}$ are its intersection with internal faces of $\mathcal{P}_1$.  In particular, $p_S(\partial a_1)$ is the intersection of $p_S(A_1)$ with the one-skeleton of the triangulation $\Delta_S$ defined below Corollary \ref{M_SandG}.

Let $b_1 \subset B_1$ and $b_3 \subset B_3$ similarly be projections of internal edges of $h_0 \cap \mathcal{P}_2$ and $h_{1-i\sqrt{2}/2} \cap \mathcal{P}_2$, respectively (see Figure \ref{fig:M_Tcusp}), and let $db_1$ and $db_3$ be the geodesic arcs of $DB_1$ and $DB_3$ containing them.  Let $db_j^{(i)} = \phi_i \circ p_T(db_j)$, and let $\partial_{\pm} db_j^{(i)} = db_j^{(i)} \cap \partial_{\pm} DB^{(i)}_j$, $j =1,3$ and $i \in \mathbb{N}$.  Let $\Delta^{-}_T$ be the image of the triangulation $\Delta_{T_0}^-$ defined below Corollary \ref{M_T0andH0} under the inclusion $M_{T_0} \to M_T$, and let $\Delta_T^+$ be its image under the doubling involution of $M_T$.  Then $\partial_{\pm} db_j^{(i)}$ is the intersection of $\partial DB_j^{(i)}$ with the one-skeleton of $\phi_i  (\Delta^{\pm}_T)$.

By Lemma \ref{M_Sboundary}, $\iota^{(0)}_+(\iota^{(0)}_-)^{-1}$ preserves triangulations, and the discussion above implies that the other gluing maps do as well.  From Figure \ref{Sface} it is apparent that the cusps of $F^{(0)}$ corresponding to $p_1$ and $p_3$ each contain only one end of an edge of the triangulation that $F^{(0)}$ inherits from the pictured fundamental domain $\mathcal{F}$.  Therefore $\iota_0(\partial a_1) = \partial_- db_1^{(1)} \cup \partial_- db_3^{(1)}$.  It then follows as before that 
$$ \alpha_1 = p_S(a_1) \cup db_1^{(1)} \cup \hdots \cup db_1^{(n)} \cup \phi_{\r}\circ p_S(a_1) \cup db_3^{(n)} \cup \hdots \cup db_3^{(1)}  $$
is a closed geodesic on $T_1$ which meets $\beta_1$ once, at right angles.  Therefore by Lemma \ref{mod_decomp} $\Re(m(T_1)) = 0$, so $m(T_1) = i(2+4n\sqrt{2})$.

A similar argument will give $m(T_2)$.  Let $\mathcal{A}_2$ be the collection of arcs in $A_2$ which are the projections of internal edges of the squares $h_v$, $v\in\mathcal{V}_1-\{0\}$.  From Figure \ref{fig:M_Scusp}, $\mathcal{A}_2$ consists of five arcs evenly spaced around $A_2$, each joining one component of $\partial A_2$ to the other and perpendicular to $\partial A_2$ at each endpoint.  For $j = 2,4$, we define a collection of arcs $D\mathcal{B}_j \subset DB_j$ analogously, and take $D\mathcal{B}_j^{(i)} = \phi_i \circ p_T (D\mathcal{B}_j)$.  Let $\partial_{\pm} D\mathcal{B}_j^{(i)} = D\mathcal{B}_j^{(i)} \cap \partial_{\pm} DB^{(i)}_j$, and note that the points of $\partial_{\pm} D\mathcal{B}_j^{(i)}$ are the points of intersection of $\partial DB_j^{(i)}$ with the one-skeleton of $\phi_i  (\Delta^{\pm}_T)$.

For the same reasons as above, $i^{(0)}_+ (i^{(0)}_-)^{-1}$ takes $p_S(\partial \mathcal{A}_2)$ to $\partial_- D\mathcal{B}_2^{(1)} \cup \partial_- D\mathcal{B}_4^{(1)}$, and the other gluing maps take the $\partial_+ D\mathcal{B}_j^{(i)}$ to $\partial_- D\mathcal{B}_j^{(i+1)}$ for the appropriate $i$ and $j$.  Then the collection
$$ p_S(\mathcal{A}_2) \cup D\mathcal{B}_2^{(1)} \cup \hdots \cup D\mathcal{B}_2^{(n)} \cup \phi_{\r}\circ p_S(\mathcal{A}_2) \cup D\mathcal{B}_4^{(n)} \cup \hdots \cup D\mathcal{B}_4^{(1)} $$
consists of a disjoint union of up to five closed geodesics, each meeting $\beta_2$ perpendicularly in at most $5$ points.  

Fix a component $\alpha_2$ of the collection above, let $k$ be the intersection number of $\alpha_2$ with $\beta_2$, and let $\widetilde{T}_2$ be the $k$-fold cover of $T_2$ dual to $\alpha_2$.  Then $\beta_2$ lifts to $\widetilde{T}_2$, and any lift intersects the preimage $\widetilde{\alpha}_2$ of $\alpha$ once, perpendicularly.  Computing the modulus of $\widetilde{T}_2$ using this pair, we obtain $\pm k \cdot i(2+4n\sqrt{2})/5$.  This is $\mathrm{PGL}_2(\mathbb{Q})$--equivalent to $m(T_1)$, so the result follows from Lemma \ref{commensurable moduli}.
\end{proof}

\begin{lemma} \label{PGL action}  Suppose $z = i(m+n\sqrt{2})$ is $\mathrm{PGL}_2(\mathbb{Q})$--equivalent to $z' = i(m+n'\sqrt{2})$, where $m,n,n' \in \mathbb{Q}$ and $m \neq 0$.  Then $n' = \pm n$.  \end{lemma}

\begin{remark}Since commensurable hyperbolic manifolds have commensurable cusps, the collection of $\mathrm{PGL}_2(\mathbb{Q})$-orbits of cusp parameters is a commensurability invariant (cf.~Lemma \ref{commensurable moduli}).  Thus Proposition \ref{M_nmoduli} and Lemma \ref{PGL action} imply Proposition \ref{M_n incomm}.\end{remark}

\begin{proof}  Suppose $\left(\begin{smallmatrix} a & b \\ c & d \end{smallmatrix}\right) \in \mathrm{PGL}_2(\mathbb{Q})$ takes $z$ to $z'$.  After clearing denominators (which does not change the action by M\"obius transformations), we may assume that $a,b,c,d \in \mathbb{Z}$.  We have
$$ \frac{ai(m+n\sqrt{2}) + b}{ci(m + n\sqrt{2}) + d} = i(m + n' \sqrt{2}). $$
Multiplying by the denominator on the left, and collecting the real and imaginary parts, we find  \begin{align*}  m(a-d) + (an - dn')\sqrt{2} = 0 & & b + c(m^2 + 2nn') + cm(n'+n) \sqrt{2} = 0 \end{align*}
Since $1$ and $\sqrt{2}$ are linearly independent over $\mathbb{Q}$, the left-hand equation above implies that $m(a-d)=0$ and $an - dn' = 0$.  Since $m \neq 0$, the first equation implies $a = d$.  Then the second equation implies $n = n'$ unless $a = d =0$.  But in this case, $c \neq 0$ since $\left(\begin{smallmatrix} a & b \\ c & d \end{smallmatrix}\right) \in \mathrm{PGL}_2(\mathbb{Q})$.  Hence, using the coefficient of $\sqrt{2}$ in the right-hand equation above, we find $n' = -n$.  \end{proof}


\section{Mutants}  \label{sec:mutants}

In the remaining sections, we will consider links obtained from $L_n$ by mutation along the separating spheres $S^{(i)}$, $0 \leq i \leq n$, from Definition \ref{L_n dfns}(3).  

\begin{dfn}\label{mutn defn}  For marked points $1,2,3,4\in S^2$, a \textit{mutation} of $(S^2,\{1,2,3,4\})$ is a mapping class of order $2$ which acts on $\{1,2,3,4\}$ by an even permutation.\end{dfn}

Above a \textit{mapping class} is the isotopy class, rel $\{1,2,3,4\}$, of an orientation-preserving self-homeomorphism of the pair $(S^2,\{1,2,3,4\})$.  The set $\mathrm{Mod}_{0,4}$ of such classes inherits the structure of a group from its bijection with the quotient of the group of orientation-preserving homeomorphisms by its identity component.  See eg.~\cite{FM} for an introduction to the study of mapping class groups; here we need only the following fact on recognizing mutations using the symmetric group $S_4$:

\begin{prop}\label{prop:mutation}  The homomorphism $\theta\co\mathrm{Mod}_{0,4}\to S_4$ that records the action on $\{1,2,3,4\}$ takes the set of mutations bijectively to $\{(12)(34),(13)(24),(14)(23)\}$.\end{prop}

\begin{proof}  Here we will embed $S^2$ as the unit sphere in $\mathbb{R}^3$ and take
$$ \{1,2,3,4\} = \left\{ \left(\frac{\pm1}{\sqrt{2}},\frac{\pm1}{\sqrt{2}},0\right),\left(\frac{\pm1}{\sqrt{2}},\frac{\mp1}{\sqrt{2}},0\right)\right\}. $$
The definitions imply that $\theta$ takes any mutation into the subset of $S_4$ listed above, and the $180$-degree rotations $m_x$, $m_y$, and $m_z$ in the three coordinate axes of $\mathbb{R}^3$  determine mutations of $(S^2,\{1,2,3,4\})$ taken by $\theta$ to each of its distinct elements.

The kernel of $\theta$ is the \textit{pure mapping class group} $\mathrm{PMod}_{0,4}$.  This group is free on two generators: Dehn twists in essential simple closed curves $\alpha,\beta\subset S^2-\{1,2,3,4\}$ that intersect exactly twice.  See the beginning of \cite[\S 4.2.4]{FM} for a proof of this fact, and for the definition of a Dehn twist see \cite[\S 3.1.1]{FM}.  With $S^2$ as above we can take $\alpha$ to be its intersection with the $xz$-plane and $\beta$ the intersection with the $yz$-plane; then it is clear that each of $m_x$, $m_y$, and $m_z$ takes each of $\alpha$ and $\beta$ to itself.  It follows that $m_x$, $m_y$ and $m_z$ centralize $\mathrm{PMod}_{0,4}$ (see \cite[Fact 3.8]{FM}).

For an arbitrary mutation $m\in\mathrm{Mod}_{0,4}$ we have $\theta(m) = \theta(m_x)$, $\theta(m) = \theta(m_y)$ or $\theta(m) = \theta(m_z)$.  Assuming (without loss of generality) that the first case holds, it follows that $m = m_x h$ for some $h\in\mathrm{PMod}_{0,4}$.  Since $m$ has order two we have:
$$ \mathit{id} = m^2 = (m_x h)^2 = m_x^2 h^2 = h^2 $$
Thus since $\mathrm{PMod}_{0,4}$ is a free group, $h = \mathit{id}$ and $m = m_x$.\end{proof}

It is easy to see that every mutation of $(S^2,\{1,2,3,4\})$ is isotopic to the identity as a self-homeomorphism of $S^2$, so cutting $S^3$ along a smoothly embedded copy and re-gluing by a mutation recovers $S^3$.  This motivates:

\begin{dfn}\label{mutn defn deux} For a link $L\subset S^3$ and a smoothly embedded two-sphere $S\subset S^3$ intersecting $L$ in four points, let $B^{\pm}$ be the closures of the components of $S^3 - S$ and $T^{\pm} = L\cap B^{\pm}$.  For a mutation $m$ of $(S,S\cap L)$, we define $(S^3,L') = (B^-,T^-)\cup_m(B^+,T^+)$ and say $L'$ is \textit{obtained from $L$ by mutation along $S$}.\end{dfn}

The lemma below describes the change in projection from $L$ to a link $L'$ obtained from it by mutation.  Below we refer to mutations by their images under $\theta$.  

\begin{figure}
\begin{center}
\input{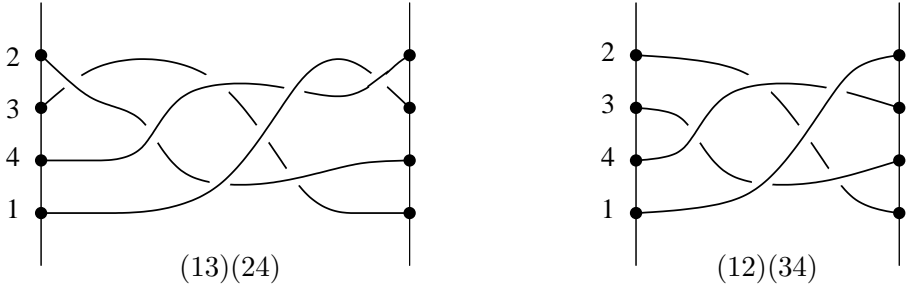}
\end{center}
\caption{The mutations as braids.}
\label{mutatebraid}
\end{figure}

\begin{lemma}\label{mutated links}  For a link $L$ projected to $\mathbb{R}^2$, if a two-sphere $S\subset S^3$ intersects $\mathbb{R}^2$ in a vertical line and $L$ in four points, label them $2$, $3$, $4$, and $1$, reading top-to-bottom.  The link obtained by the mutation $(13)(24)$ (respectively, $(12)(34)$) along $S$ has projection obtained by cutting $L$ along $S$ and inserting the braid on the left (resp. right) of Figure \ref{mutatebraid}.  \end{lemma} 

\begin{proof}  We may assume $L$ is arranged so that there is an axis in $\mathbb{R}^2$ intersecting $S$ perpendicularly, midway between points $3$ and $4$ and so that points $2$ and $1$ are also equidistant from it.  The $180$-degree rotation in this axis restricts on $S$ to an involution acting on the marked points by the permutation $(12)(34)$.  

There is a homeomorphism $R \co S \times I \to S \times I$, which preserves slices $S \times \{t\}$ and restricts on each to rotation by $-180\cdot t$ degrees in the horizontal axis.  This interpolates between the identity, on $S \times \{0\}$, and the inverse of $(12)(34)$ on $S \times \{1\}$, although it does not preserve marked points for $0 < t <1$.

Let $(B^{\pm},T^{\pm})$ be as in Definition \ref{mutn defn deux}, and let $C$ be a collar of $S$ in $B^-$ that is small enough that it intersects $T^-$ in the collection of horizontal arcs $\{\{j\} \times I\,|\, j \in \{1,2,3,4\}\}$.  There is a homeomorphism $h\co B^-\cup_{(12)(34)} B^+ \to S^3$ defined as the identity on $B^+$ and the complement of $C$ in $B^-$, and as $R$ on $C$.  By the definition of $R$, the image of $T^- \cap C$ under $R$ is as pictured on the right-hand side of Figure \ref{mutatebraid}, thus the image of $L'$ in $S^3$ under $h$ is as stated in the lemma. 

Note that the braid on the left-hand side of Figure \ref{mutatebraid} is the conjugate of the braid on the right by a left-handed half-twist exchanging the points $2$ and $3$.  This reflects the fact that the conjugate of $(12)(34)$, by any homeomorphism of $(S,\{1,2,3,4\})$ which exchanges $2$ and $3$ and fixes $1$ and $4$, is a mapping class of order $2$ acting on the marked points as $(13)(24)$; hence such a conjugate is $(13)(24)$.  The conjugating braid in Figure \ref{mutatebraid} tracks the marked points under an isotopy $S \times I \to S$ taking the simplest such conjugator to the identity.  The conclusion for $(13)(24)$ thus follows as it did above for $(12)(34)$.
\end{proof}

The numbering of marked points from Lemma \ref{mutated links} and Figure \ref{mutatebraid} agrees with the numbering of the $S^{(i)} \cap L_n$ from Definition \ref{L_n dfns}(3).  This in turn was chosen to agree with the numbering of parabolics of $\Lambda$ from Lemma \ref{M_Sboundary}.  To be more precise:

Let $S$ be the sphere obtained by compactifying each cusp of $F^{(0)} = \calh/\Lambda$ with a single point.  Label each new point by a number between $1$ and $4$, according to the parabolic $\p_i$ corresponding to the cusp it compactifies.  With the points of $S^{(0)}\cap L_n$ numbered as in Definition \ref{L_n dfns}(3), it follows from Proposition \ref{T0map} that the restriction of $f_T$ (as in Proposition \ref{Tmap}) to $S^{(0)}-T$ extends to a map $S^{(0)}\to S$ that preserves numbering.  Corollary \ref{geodesic spheres} and the definition of $F^{(i)}$ (see Definitions \ref{geometric pieces}) now imply that for each $i$ between $0$ and $n$, $\phi_{\sfc^{2i}}\circ\iota_{i}^{-1} \circ f_n$ extends to a homeomorphism $S^{(i)} \to S$ that takes marked points to marked points preserving numbering.

By \cite[Theorem 2.2]{Ru}, each mutation of $(S,\{1,2,3,4\})$ is {\em realized} by an isometry of $F^{(0)}$;   that is, there exists an isometry of $F^{(0)}$ whose extension to $(S, \{ 1,2,3,4\})$ represents the mutation mapping class.  The lemma below identifies lifts to $\mathrm{PSL}_2(\mathbb{R})$ of the isometries realizing $(13)(24)$ and $(12)(34)$.

\begin{lemma}\label{mutators}  Define \begin{align*}
 & \sfm_1 =  \begin{pmatrix}  -3 & 5 \\ -2 & 3 \end{pmatrix} &
 & \sfm_2 = \begin{pmatrix}  0 & \sqrt{5} \\ \frac{-1}{\sqrt{5}} & 0 \end{pmatrix}  \end{align*}
Each of $\sfm_1$ and $\sfm_2$ normalizes $\Lambda$ (from Lemma \ref{M_Sboundary}), and the induced isometries $\phi_{\sfm_1}$ and $\phi_{\sfm_2}$ of $F^{(0)}$ realize $(13)(24)$ and $(12)(34)$, respectively.  \end{lemma}

\begin{proof}  Since each of $\sfm_1$ and $\sfm_2$ has trace equal to zero, it has order $2$ in $\mathrm{PSL}_2(\mathbb{C})$.  Their actions by conjugation described below, on the generators $\p_1$, $\p_2$, and $\p_3$ for $\Lambda$ defined above Lemma \ref{M_Sboundary}, may be verified by direct computation.  \begin{align*}
  & \p_1^{\sfm_1} = \p_3^{-1} && \p_2^{\sfm_1} = \p_4^{-1} \\
  & \p_1^{\sfm_2} = \p_2 && \p_3^{\sfm_2} = \p_4^{\p_1^{-1}}  \end{align*}
Here $\p_4 = \p_1\p_2\p_3^{-1}$ is as described in Remark 1 below Lemma \ref{M_Sboundary}.  Therefore $\sfm_1$ and $\sfm_2$ normalize $\Lambda$ and induce isometries $\phi_{\sfm_1}$ and $\phi_{\sfm_2}$, respectively, of $F^{(0)} = C(\Lambda)$.  

Each of $\phi_{\sfm_1}$ and $\phi_{\sfm_2}$ has order $2$, since $\sfm_1$ and $\sfm_2$ have order $2$, and their extensions to $S$ act on the set of marked points  as described in the statement of the lemma.  Its conclusion therefore follows from Proposition \ref{prop:mutation}.  \end{proof}

\begin{cor}\label{indexed mutators}  For $j = 1,2$ and $i \in \mathbb{Z}$, let $\sfm_j^{(i)} = \c^{-2i}\sfm_j\c^{2i}$.  Each of $\sfm_1^{(i)}$ and $\sfm_2^{(i)}$ normalizes $\Lambda^{(i)}$ (from Definitions \ref{geometric pieces}), and the induced isometries of $F^{(i)}$ realize $(13)(24)$ and $(12)(34)$, respectively. 
\end{cor}

Lemma \ref{mutated links} gives a prescription for describing links obtained from $L_n$ by the mutations $(13)(24)$ and $(12)(34)$.  The result below describes hyperbolic manifolds to which their complements are homeomorphic, analogous to Proposition \ref{geometric model}.

\begin{prop}\label{mutated geometric model}  For $I =(a_0,a_1,\hdots,a_n) \in \{0,1,2\}^{n+1}$, let $L_I$ be the link obtained from $L_n$ by the following prescription: for $0 \leq i \leq n$, if $a_i = 0$, do not mutate along $S^{(i)}$; if $a_i = 1$, mutate by $(13)(24)$; and if $a_i = 2$, mutate by $(12)(34)$.  Let $M_I = C(\Gamma_S) \cup C(\Gamma_T^{(1)}) \cup \hdots \cup C(\Gamma_T^{(n)}) \cup C(\overline{\Gamma}_S)$, where for each $i$ such that $a_i = 0$ the gluing is as in Proposition \ref{geometric model}, and otherwise is given by   
\begin{align*} 
  & \iota^{(i)}_+ \phi_{\sfm^{(i)}_j} (\iota^{(i)}_-)^{-1}\ \mbox{for}\ 0 \leq i < n,\ \mbox{where}\ a_i = j \in \{1,2\};\ \mbox{and}\\
  & \phi_{\r}\iota^{(0)}_-\phi_{\c^{2n}}\phi_{\sfm_j^{(n)}}(\iota^{(n)}_-)^{-1}\ \mbox{if}\ a_n = j \in \{1,2\}.  \end{align*}
Then there is a homeomorphism $f_I \co S^3-L_I \to M_I$ whose restriction to each complementary component of the collection $\{S^{(i)}\}$ agrees with that of $f_n$.  \end{prop}

Proposition \ref{mutated geometric model} follows immediately from Proposition \ref{geometric model} and Corollary \ref{indexed mutators}.  Below we note a couple of ``obvious'' isometry relations on the $\{M_I\}$.

\begin{figure}[ht]
\setlength{\unitlength}{.1in}
\begin{picture}(40,15)
\put(11.5,0) {\includegraphics[height= 1.45in]{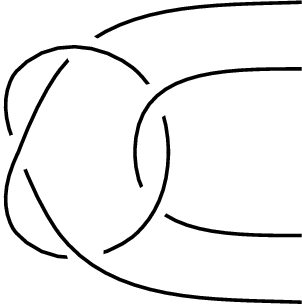}}

\put(27,14){$3$}

\put(27,11){$2$}

\put(27, 3){$4$}

\put(27, 0){$1$}

\end{picture}

\caption{The tangle $S$ admits an order two rotational symmetry which restricts to the mutation $(13)(24)$ on its boundary.}  
\label{fig:symmetric_S}

\end{figure}
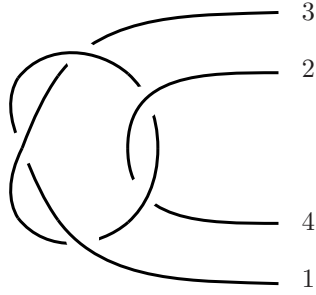

\begin{lemma}\label{mutation extension}  For fixed $(a_1,\hdots,a_n)\in\{0,1,2\}^n$ let $I_0 = (0,a_1,\hdots,a_n)$ and $I_1 = (1,a_1,\hdots,a_n)$.  $M_{I_0}$ is isometric to $M_{I_1}$.\end{lemma}

\begin{proof}  It is evident from Figure \ref{fig:symmetric_S} that the mutation
$(13)(24)$ extends to a homeomorphism on $B^3-S$.  Thus $(S^3,L_{I_0})$ is homeomorphic to $(S^3,L_{I_1})$, and the result follows from Mostow rigidity.  \end{proof}



\begin{lemma}\label{reflection isometry}For $I =(a_0,a_1,\hdots,a_n) \in \{0,1,2\}^{n+1}$, let $\bar{I} = (a_n,a_{n-1},\hdots,a_0)$.  There is an orientation-reversing isometry $M_I \to M_{\bar{I}}$ that, for each $i \in \{1,\hdots,n\}$, takes the image of $C(\Gamma_T^{(i)})$ in $M_I$ to the image of $C(\Gamma_T^{(n-i)})$ in $M_{\bar{I}}$.\end{lemma}

\begin{proof}  It is straightforward to check that the braids in Figure \ref{mutatebraid} are isotopic (in $S^2 \times I$) to their mirror images.  Therefore, there is an orientation reversing homeomorphism $L_I \to L_{\bar{I}}$.  By composing this homeomorphism with $f_I^{-1}$ and $f_{\bar{I}}$ we get a homeomorphism $M_I \to M_{\bar{I}}$.  The result follows by Mostow rigidity.
\end{proof}

Below we describe the change effected at the level of Kleinian groups by cutting a hyperbolic manifold along an embedded, separating totally geodesic surface and regluing by an isometry.

\begin{lemma}\label{sliced Maskit}  Suppose $\Gamma_0$ and $\Gamma_1$ \meetcute along a plane $\mathcal{K}$, and take $\Theta=\Gamma_0\cap\Gamma_1$, $E = \calk/\Theta$, and $\iota_0$ and $\iota_1$ as in Lemma \ref{Maskit}.  If $\sfn$ normalizes $\Theta$ and preserves components of $\mathbb{H}^3-\mathcal{K}$, then $\langle \Gamma_0, \Gamma_1^{\sfn} \rangle$ is a Kleinian  group, and there is an isometry 
$$  C(\Gamma_0) \cup_{\iota_1 \phi_{\sfn}^{-1} \iota_0^{-1}} C(\Gamma_1) \to C(\langle \Gamma_0,\Gamma_1^{\sfn} \rangle) $$
which restricts on $C(\Gamma_0)$ to the natural map, and on $C(\Gamma_1)$ to $\phi_{\sfn} \co C(\Gamma_1) \to C(\Gamma_1^{\sfn})$ followed by the natural map.  \end{lemma}

\begin{proof}  Since $\sfn$ normalizes $\Theta$, it preserves $\calk$; hence our hypotheses ensure that $\Gamma_0$ and $\Gamma_1^{\sfn}$ \meetcute along $\calk$, and Lemma \ref{Maskit} applies.  Thus $\langle \Gamma_0, \Gamma_1^{\sfn} \rangle$ is a Kleinian group, and in particular, the natural maps $C(\Gamma_0) \to C(\langle \Gamma_0, \Gamma_1^{\sfn} \rangle)$ and $C(\Gamma_1^{\sfn}) \to C(\langle \Gamma_0, \Gamma_1^{\sfn} \rangle)$ determine an isometry
$$ C(\Gamma_0) \cup_{n\iota_1 \iota_0^{-1}} C(\Gamma_1^{\sfn}) \to C(\langle \Gamma_0,\Gamma_1^{\sfn} \rangle). $$ 
Here we are using $n\iota_1 \co E \to C(\Gamma_1^{\sfn})$ to refer to the natural map.  It is now an exercise in definition-chasing to show that $n\iota_1 \circ \phi_{\sfn} = \phi_{\sfn} \circ \iota_1$, whence the map
$$ C(\Gamma_0) \cup_{\iota_1 \phi_{\sfn}^{-1} \iota_0^{-1}} C(\Gamma_1) \to C(\Gamma_0) \cup_{n\iota_1 \iota_0^{-1}} C(\Gamma_1^{\sfn}), $$
defined as the identity on $\C(\Gamma_0)$ and $\phi_{\sfn}$ on $C(\Gamma_1)$, is well-defined.  The lemma follows.  \end{proof}

Since $\sfm_1$ and $\sfm_2$ have order $2$, $\phi_{\sfm_i} = \phi_{\sfm_i}^{-1}$ for $i = 1,2$.  Lemma \ref{sliced Maskit} thus yields the result below, which describes how the algebraic model for $M_n$ from Proposition \ref{algebraic model} changes under mutation.  

\begin{prop}\label{mutated algebraic model}  For $I = (a_0,a_1,\hdots,a_n) \in \{0,1,2\}^{n+1}$ let $\sfq_{i+1}=\sfm_{a_0}^{(0)} \cdots \sfm_{a_i}^{(i)}$ for $0 \leq i \leq n$, with $\sfm_0^{(j)}:=id$ and $\sfm_1^{(j)}$, $\sfm_2^{(j)}$ as in Corollary \ref{indexed mutators} for every $j$.  Define\[ \Gamma_I \ =\  \left\langle \Gamma_S, \Big( \Gamma_T^{(1)}\Big)^{\sfq_1}, \ldots, \Big( \Gamma_T^{(n)}\Big)^{\sfq_n}, \Big(\overline{\Gamma}_S^{\c^{-2n}}\Big)^{\sfq_{n+1}}\right\rangle\] There is an isometry $M_I \to C(\Gamma_I)$ that restricts on $C(\Gamma_S)$ to the natural map, and on $C\Big(\Gamma_T^{(i)}\Big)$ to $\phi_{\sfq_i}\circ \phi_i$ followed by the natural map, for $1 \leq i \leq n$.\end{prop}

The proof of Proposition \ref{mutated algebraic model} follows the inductive approach of that of Proposition \ref{algebraic model}, but at each stage appeals to Lemma \ref{sliced Maskit} for instructions on how to change the construction.  We will not write the details, as it is very similar.

\section{Commensurable mutants}\label{sec:commensurablo!}

Here we show that $M_n$ is commensurable to each of its mutants by $(13)(24)$ and with this fact classify the $M_I$ up to isometry for $I\in\{0,1\}^{n+1}$, proving Theorem \ref{omnibus1.5}.  In the process, we show that our polyhedral decomposition of $M_n$ is ``canonical'' in the sense of \cite[\S 2]{GHH}; i.e.~produced by a construction of Epstein--Penner \cite{EP}.  This allows us to identify the commensurator for $\Gamma_n$ and the minimal orbifold quotient of $M_n$.  In practice, it is a challenge to find Epstein-Penner decompositions, commensurators, and commensurator quotients.  Infinite families where these are known are rare. 

Below, let $\calb_0$ be the open half-ball in $\mathbb{H}^3$ bounded by the Euclidean hemisphere of unit radius centered at $0 \in \mathbb{C}$, and let $\calb_j = \c^{-j}(\calb_0)$, where $\c$ is as defined in Lemma \ref{M_TandH}.  Recall that we have defined $\calh$ as the geodesic hyperplane of $\mathbb{H}^3$ with ideal boundary $\mathbb{R} \cup \{\infty\}$.  If $w$ and $z$ are complex numbers, we will take $w\calh+z$ to be the hyperplane with ideal boundary $\left(w\mathbb{R} + z\right) \cup \{\infty\}$.  

\begin{dfns}\label{some isoms}
\[\] \vspace{-.40in}
 \begin{enumerate}  \item  Let $\f_0$ be obtained by first reflecting in $i\calh$ and then in $i\calh+1/2$.
  \item  Let $\sfb_0$ be obtained by first reflecting in $\calh+i/2$ and then in $\partial \calb_0$.
  \item  For $j \geq 0$, let $\sfa_j$ be obtained by reflecting in $i\calh+1/2$ and then in $\partial \calb_j$.  
\end{enumerate}  \end{dfns}

Since $i\calh$ and $i\calh+1/2$ are parallel and share the ideal point $\infty$, $f_0$ is a parabolic isometry fixing $\infty$.  $\calh+i/2$ meets $\partial \calb_0$ at an angle of $\pi/3$, so $\sfb_0$ is an elliptic isometry of order $3$ rotating around the geodesic of intersection.  For the same reason, $\sfa_i$ is elliptic of order $3$, rotating around the geodesic $i\calh+1/2 \cap \partial \calb_i$, for each $i\geq 0$.  

\begin{lemma}\label{Gamma_n commensurator}  Let $G_n$ be the group generated by reflections in the face of $P_n$, where
$$ P_n = \left\{ (z,t) \in \mathbb{H}^3\, |\ 0 \leq \Re(z) \leq 1/2, -n\sqrt{2} \leq \Im(z) \leq 1/2 \right\} - \left( \bigcup_{k=0}^n \calb_k \right).  $$
Then $G_n$ contains $\sfa_i$ for $0 \leq i \leq 2n$, as well as $\f_0$ and $\sfb_0$, and $O_n\doteq\mathbb{H}^3/G_n$ is a one-cusped hyperbolic orbifold.  \end{lemma}

\begin{proof}  By its definition, $P_n$ is cut out by $\calh + i/2$, $i\calh$, $i\calh+1/2$, $\calh-n\cdot i\sqrt{2}$, and the $\partial \calb_k$, $0 \leq k \leq n$.  It is not hard to show directly that the dihedral angle between any two of these planes that intersect is an integer submultiple of $\pi$, whence by the Poincar\'e polyhedron theorem $G_n$ is discrete and $\mathbb{H}^3/G_n$ is an orbifold isometric to $P_n$ with mirrored sides (cf.~\cite[Theorem 13.5.1]{Ra}).  In particular, since $P_n$ has a single ideal point $\mathbb{H}^3/G_n$ has one cusp.

One finds that $G_n$ contains $\f_0$, $\sfb_0$, and the $\sfa_i$, for $0 \leq i \leq n$, by direct appeal to Definitions \ref{some isoms}.  It remains to establish that $G_n$ contains $\sfa_i$ for $n < i \leq 2n$.  Note that $\calh-n\cdot i\sqrt{2}$ is the image of $\calh$ under $\c^{-n}$, so reflection in $\calh-n\cdot i\sqrt{2}$ is given by $\c^{-n}\r\c^n$, where $\r$ is the reflection through $\calh$.  By the property of $\r$ observed above Lemma \ref{M_TandH}, conjugating an element $\sfx \in \mathrm{PSL}_2(\mathbb{C})$ by reflection in $\calh-n\cdot i\sqrt{2}$ gives: \begin{align}\label{back conjugation}
  \c^{-n}\r\c^n \sfx \c^{-n}\r \c^{n} = \c^{-2n} \bar{\sfx} \c^{2n}  \end{align}
We further observe that $\c$ conjugates $\sfa_i$ to $\sfa_{i-1}$ for $i \geq 1$, since $\c(i\calh+1/2) = i\calh+1/2$ and $\c(\calb_i) = \calb_{i-1}$, and we note that $\bar{\sfa}_0 = \sfa_0$.  Thus: \begin{align}\label{a vs abar}
  \overline{\c^{i}\sfa_i\c^{-i}} = \bar{\sfa}_0 = \sfa_0 = \c^{i}\sfa_i\c^{-i}\ \ \Rightarrow\ \ \c^{-2i}\bar{\sfa}_i\c^{2i} = \sfa_i  \end{align}
For $0 \leq i \leq n$, it follows that the conjugate of $\sfa_i$ by reflection in $\calh - n\cdot i\sqrt{2}$ is:
$$  \c^{-2n}\bar{\sfa}_i\c^{2n} = \c^{-2(n-i)} \sfa_i \c^{2(n-i)} = \sfa_{2n-i} \in G_n.  $$
Therefore $G_n$ contains $\sfa_i$ for $n \leq i \leq 2n$ as well, and the lemma is proved.  \end{proof}

Since $\calh$ meets both $\partial \calb_0$ and $i\calh+1/2$ at right angles, it does the same for the fixed geodesic of $\sfa_0$ and is therefore preserved by $\sfa_0$.  In fact, the following description of $\sfa_0 \in \mathrm{PSL}_2(\mathbb{C})$ is easily obtained from its definition:
$$ \sfa_0 = \left( \begin{smallmatrix}  0 & 1 \\ -1 & 1 \end{smallmatrix} \right)  $$
In particular, $\sfa_0$ acts on the ideal points of $\calp_1 \cap \calp_2$ by $0 \mapsto 1 \mapsto \infty \mapsto 0$.  Similarly, it is easy to see that $f_0(z,t)= (z+1,t)$

Then the face pairings $\f$ (defined in Lemma \ref{M_T0andH0}) and $\s$ (defined in Lemma \ref{M_SandG}), which are equal, are obtained from $\f_0$ by conjugating by $\sfa_0$:  \begin{align}\label{f vs f0}
  \s = \f = \sfa_0\, \f_0\, \sfa_0^{-1}.  \end{align}
One may use similar analyses to establish the following.  \begin{align}\label{g and h vs f0}
  &  \t = (\sfb_0\sfa_0)^{-1}\f_0(\sfb_0\sfa_0)\,\sfa_0 && \g = (\sfa_0^{-1} \sfa_1) \f_0^{-1} (\sfa_0^{-1} \sfa_1)^{-1} && \h = \sfa_1 \sfa_0 \f_0^{-1} \sfa_1  \end{align}

The main group-theoretic fact of this section extends these observations.

\begin{prop}\label{G_n is hungry}  For each $n \in \mathbb{N}$, $G_n$ contains $\Gamma_n$ and $\sfm_1^{(i)}$ for $0 \leq i \leq n$.  \end{prop}

\begin{proof}  We recall from Proposition \ref{algebraic model} that $\Gamma_n = \langle \Gamma_S, \Gamma_T^{(1)},\hdots,\Gamma_T^{(n)},\overline{\Gamma}_S^{\c^{-2n}} \rangle$, where by  Definition \ref{geometric pieces}(2), $\Gamma_T^{(i)} \doteq \Gamma_T^{\c^{-2(i-1)}}$ for each $i$ between $1$ and $n$.  Furthermore, by Lemma \ref{M_TandH}, $\Gamma_T = \langle \Gamma_{T_0},\overline{\Gamma}_{T_0}^{\c^{-2}}\rangle$.

It is a direct consequence of the descriptions (\ref{f vs f0}) and (\ref{g and h vs f0}) above that $\Gamma_S < G_n$ and $\Gamma_{T_0} < G_n$.  Furthermore, since $\f_0$ commutes with $\c$ and $\bar{\f}_0 = \f_0$, (\ref{a vs abar}) implies for instance that
$$ \c^{-2}\overline{\f}\c^2 = \c^{-2}(\bar{\sfa}_0 \bar{\f}_0 \bar{\sfa}_0^{-1})c^2 = \sfa_2 \f_0 \sfa_2^{-1} \in \Gamma_n,  $$
since $\bar{\sfa}_0 = \sfa_0$ and $\c^{-2}\sfa_0c^2 = \sfa_2$.  Using the same strategy, we find:
$$  \c^{-2}\bar{\g}c^{2} = (\sfa_2^{-1}\sfa_1)\f_0^{-1}(\sfa_2^{-1}\sfa_1)^{-1} \in G_n\ \mbox{and}\ \c^{-2}\bar{\h}\c^2 = \sfa_1 \sfa_2 \f_0^{-1} \sfa_1 \in G_n  $$
Thus $G_n$ contains $\Gamma_T = \Gamma_T^{(1)}$.  Since conjugation by $\c^{-1}$ takes $\sfa_i$ to $\sfa_{i+1}$, and $\sfa_i \in \Gamma_n$ for each $i$ between $0$ and $2n$, it follows from the descriptions above and in (\ref{f vs f0}) and (\ref{g and h vs f0}) that each $\Gamma_T^{(i)}$, $1 \leq i \leq n$.  Finally the relation (\ref{back conjugation}) immediately implies that $\overline{\Gamma}_S^{\c^{-2n}} < G_n$, and we have established that $\Gamma_n < G_n$.

To show that $G_n$ contains the elements $\sfm_1^{(j)}$ for each $j$ between $0$ and $n$, we observe that the element obtained by reflecting first in $\partial \calb_0$ and then in $i\calh$ is the rotation of order $2$ described by $\left( \begin{smallmatrix} 0 & -1 \\ 1 & 0 \end{smallmatrix} \right)$.  This is well known to generate $\mathrm{PSL}_2(\mathbb{Z})$, along with $\sfa_0$.  Since $\sfm_1 \in \mathrm{PSL}_2(\mathbb{Z})$, it follows that $\sfm_1 \in G$.  

We note that $\c^{-2j}$ preserves $i\calh$ and takes $\calb_0$ to $\calb_{2j}$, and that $\calb_{2j}$ intersects $P_n$, for $j\leq n/2$, and intersects its image under reflection in $\calh - n\cdot i\sqrt{2}$ for $n/2 \leq j \leq n$.  Thus for each $j$ between $0$ and $n$, the rotation obtained by reflecting first in $\partial \calb_{2j}$ and then in $i\calh$ is contained in $G_n$.  If $\sfm_1$ is expressed as a word in the two elements described in the paragraph above, then $\c^{-2j}\sfm_1\c^{2j}$ is expressed as the same word in $\sfa_{2j}$ and the rotation obtained from $\partial \calb_{2j}$ as above.  The lemma follows.
\end{proof}

It is now easy to prove the first part of Theorem \ref{omnibus1.5}, that the complement of each link obtained from $L_n$ using only the mutation $(13)(24)$ is commensurable to $M_n$.

\begin{prop}\label{commensurable mutants}  $M_n$ branched covers $O_n$, as does $M_I$ for any $I \in \{0,1\}^{n+1}$. Hence these are commensurable.  \end{prop}

\begin{proof}  Since $G_n$ is a discrete reflection group, it is enough to show that $\Gamma_I \subset G_n$.  This is immediate from Propositions \ref{mutated algebraic model} and \ref{G_n is hungry}.
\end{proof}

To finish the proof of Theorem \ref{omnibus1.5} we need an isometry classification of the link complements that fall under the purview of Proposition \ref{commensurable mutants}.  Our first step is to show that $G_n$ is the commensurator of $\Gamma_n$.  

The \textit{commensurator} of a Kleinian group $\Gamma$ is the group
\[ \text{Comm}(\Gamma) \ = \ \left\{ {\sf g} \in \text{Isom}(\mathbb{H}^3) \, | \, [\Gamma : {\sf g}\Gamma {\sf g}^{-1}] < \infty \right\}\]
It follows easily from the definition that since $\Gamma_n$ is a finite-index subgroup of $G_n$,  $G_n$ is contained in $\text{Comm}(\Gamma_n)$.  Since $\Gamma_n$ is non-arithmetic (by Proposition \ref{trace fields}),  by a famous theorem of Margulis $\text{Comm}(\Gamma_n)$ is discrete (see \cite[(1) Theorem]{Ma}).

Let  $O_n'$ be the hyperbolic orbifold $\mathbb{H}^3/\text{Comm}(\Gamma_n)$.  Since $G_n<\mathrm{Comm}(\Gamma_n)$, $O'_n$ is finitely covered by $O_n$.  Recall from Lemma \ref{Gamma_n commensurator} that $O_n$ and therefore also $O_n'$ has exactly one cusp.  It is our goal to show that $G_n=\mathrm{Comm}(\Gamma_n)$; hence $O_n = O_n'$.

We use the strategy of \cite{GHH}.  Recall the \textit{hyperboloid} model for $\mathbb{H}^3$.  The \textit{Lorentz inner product} on $\mathbb{R}^4$ is the indefinite bilinear form 
\[ \langle {\bf v, w} \rangle \ = \ v_1 w_1+v_2w_2+v_3w_3-v_4w_4.\]
We let $\mathbb{H}^3=\{ \mathbf{v} \, | \,  \langle {\bf v, v} \rangle=-1, \, v_4>0 \}$ and equip $T_{\bv}\mathbb{H}^3$ with the Riemannian metric determined by the Lorentz inner product.  The \textit{positive light cone} is the set $L^+=\{  {\bf v} \, | \,  \langle {\bf v,v} \rangle=0, \, v_4\geq0 \}$.  The \textit{ideal point} of $\mathbb{H}^3$ represented by $\bv\in L^+$ is the set $[\bv]$ of scalar multiples of $\bv$ in $L^+-\{\mathbf{0}\}$.  $\text{Isom}(\mathbb{H}^3)$ is the group of matrices in $\text{GL}_4(\mathbb{R})$ which act on $\mathbb{R}^4$ preserving the Lorentz inner product and the sign of the last coordinate, hence acting on $\mathbb{H}^3$ by isometries.  Those in $\mathrm{Isom}^+(\mathbb{H}^3)\subset \mathrm{Isom}(\mathbb{H}^3)$ preserve orientation on $\mathbb{H}^3$.

For ${\bf v} \in L^+-\{0\}$ the set $H_{\bf v} = \{ {\bf w} \in \mathbb{H}^3 \, | \, \langle {\bf v, w} \rangle =-1\}$ is a horosphere centered at the ideal point $[\bv]$.  If $\alpha \in \mathbb{R}^+$ then $H_{\alpha {\bf v}}$ is a horosphere centered at $[\alpha\bv]=[\bv]$, and if $\alpha \leq 1$ then $H_{\bf v}$ is contained in the horoball $\{\bw\,|\,\langle\alpha\bv,\bw\rangle\geq-1\}$ determined by $\alpha {\bf v}$.  This determines a bijective correspondence between vectors in $L^+$ and horospheres in $\mathbb{H}^3$, so we call the vectors in $L^+$ \emph{horospherical vectors}.

We use the hyperboloid model to construct certain canonical tilings of $\mathbb{H}^3$ associated to $M_n$ as in \cite{EP}.  First, choose a horospherical vector ${\bf v} \in L^+$ fixed by a peripheral element of $\Gamma_n$, so that under the covering map $\mathbb{H}^3 \to O_n'$ the horosphere $H_{\bf v}$ projects to a cross section of the cusp.  Then $V_n=\text{Comm}(\Gamma_n) \cdot{\bf v}$ is $\text{Comm}(\Gamma_n)$-invariant and determines a $\text{Comm}(\Gamma_n)$-invariant set of horospheres.  The convex hull of $V_n$ in $\mathbb{R}^4$ is called the \emph{Epstein--Penner convex hull}, we denote it as $C_n$.  Epstein and Penner show that $\bound C_n$ consists of a countable set of 3-dimensional faces $F_i$, where each $F_i$ is a finite sided Euclidean polyhedron in $\mathbb{R}^4$.  Furthermore, this decomposition of $\bound C_n$ projects along straight lines through the origin to a $\text{Comm}(\Gamma_n)$-invariant tiling $\mathcal{T}_n$ of $\mathbb{H}^3$ by ideal polyhedra \cite[Prop. 3.5 and Thm. 3.6]{EP}.   
We refer to the tiling so obtained as a {\it canonical tiling}.   (It is easy to see that a different choice for the vector ${\bf v}$ yields a convex hull which differs from $C_n$ by multiplication by a positive scalar.  Therefore it projects to the same canonical tiling as $C_n$.)

Consider the group of symmetries $\text{Sym}(\mathcal{T}_n) < \text{Isom}(\mathbb{H}^3)$.  Since $\mathcal{T}_n$ is $\text{Comm}(\Gamma_n)$-invariant we have that $\text{Comm}(\mathcal{T}_n) < \text{Sym}(\mathcal{T}_n)$.  On the other hand, $\text{Sym}(\mathcal{T}_n)$ is discrete \cite[Lemma 2.1]{GHH} and since $\Gamma_n$ is non-arithmetic $\text{Comm}(\Gamma_n)$ is the maximal discrete group containing $\Gamma_n$.  Therefore $\text{Sym}(\mathcal{T}_n)=\text{Comm}(\Gamma_n)$.  Below we will first identify the tiling $\mathcal{T}_n$ and then show that $G_n = \mathrm{Sym}(\calt_n)$.

\begin{thm} \label{tiling}  With $\mathcal{S}$ as in Lemma \ref{poly decomp}. $\mathcal{T}_n=\Gamma_n \cdot \bigcup\{\calp\in\mathcal{S}\}$ is the canonical tiling for $\mathrm{Comm}(\Gamma_n)$.
\end{thm}

\proof
For a $4 \times n$ matrix $X$ below, let $x_i$ be the $i^{\text{th}}$ column of $X$.  Each $x_i$ below lies in $L^+$ and so represents an ideal point of $\mathbb{H}^3$.  We will call $\calp_X$ the convex hull  in $\mathbb{H}^3$ of the $[x_i]$.
\[M \ = \ \left(
\begin{array}{llllllllllll}
 2 & 1 & 0 & 1 & 0 & -1 & -2 & -1 & 1 & -1 & -1 & 1 \\
 0 & 1 & 2 & 1 & -2 & -1 & 0 & -1 & -1 & 1 & 1 & -1 \\
 0 & \sqrt{2} & 0 & -\sqrt{2} & 0 & \sqrt{2} & 0 & -\sqrt{2} & -\sqrt{2} & -\sqrt{2} & \sqrt{2} & \sqrt{2} \\
 2 & 2 & 2 & 2 & 2 & 2 & 2 & 2 & 2 & 2 & 2 & 2
\end{array}
\right)\]
\[N \ = \ \left(
\begin{array}{llllll}
 \sqrt{2} & 0 & 0 & -\sqrt{2} & 0 & 0 \\
 0 & \sqrt{2} & 0 & 0 & -\sqrt{2} & 0 \\
 0 & 0 & \sqrt{2} & 0 & 0 & -\sqrt{2} \\
 \sqrt{2} & \sqrt{2} & \sqrt{2} & \sqrt{2} & \sqrt{2} & \sqrt{2}
\end{array}
\right)\]
$M$ and $N$ were chosen so that $\calp_M$ is a right-angled ideal cuboctahedron and $\calp_N$ a regular ideal octahedron, and furthermore:\begin{itemize}
  \item For $X=M,N$, each member of $\text{Isom}(\calp_X)$ fixes $(0,0,0,1)^T \in \mathbb{H}^3$ and the set of columns of $X$ is $\text{Isom}(\calp_X)$-invariant.  
  \item There exists ${\sf h}\in\mathrm{Isom}^+(\mathbb{H}^3)$ with ${\sf h}(n_1) = m_1$, ${\sf h}(n_2)=m_9$ and ${\sf h}(n_3)=m_4$, so that ${\sf h}(\calp_N) \cap \calp_M$ is the face $(m_1,m_9,m_4)$ with ideal vertices at $[m_1]$, $[m_9]$ and $[m_4]$.\end{itemize}  
Let $\calp_1={\sf h}(\calp_N)$ and $\calp_2=\calp_M$.  There is an isometry $p$ from the upper half-space to the hyperboloid model taking $\calp_i$ (as in Corollaries \ref{M_SandG} and \ref{M_T0andH0}) to $\calp_i$ as above for $i=1,2$, and $\infty$ to the center $[m_1]$ of the horosphere $H_{m_1}$.   We again refer by $\mathcal{S}$ to the image under $p$ of the set $\mathcal{S}$ from Lemma \ref{poly decomp}.  Also, $p$ conjugates each of the isometries we've used thus far in our constructions to elements of $\mathrm{GL}_4(\mathbb{R})$, to which we'll refer by the same names.

From the explicit description in Lemma \ref{Gamma_n commensurator} it is clear that $[m_1]$ is a parabolic fixed point of $G_n$.  Since $G_n$ is discrete, each element fixing $[m_1]$ actually fixes $m_1$, so the orbit $V_n = G_n.m_1$ is a $G_n$-invariant collection of horospherical vectors bijective to the set of parabolic fixed points of $G_n$.  Since $O_n=\mathbb{H}^3/G_n$ has one cusp and the same holds for $O_n'=\mathbb{H}^3/\mathrm{Comm}(\Gamma_n)$ it follows that $V_n$ is also $\mathrm{Comm}(\Gamma_n)$-invariant.

Lemma \ref{poly decomp} implies that $\Gamma_n.\bigcup\{\calp\in\mathcal{S}\}$ is a $\Gamma_n$-invariant tiling of $\mathbb{H}^3$.  We claim that it is identical to the canonical tiling $\calt_n$, the projection to $\mathbb{H}^3$ of the boundary of the convex hull of $V_n$ in $\mathbb{R}^4$.  Note that $\calt_n$ is also $\Gamma_n$-invariant, since it is $G_n$-invariant by construction and $\Gamma_n<G_n$.

We will use \cite[Proposition 6.1]{GHH} to prove the claim.  The proposition requires for each element of $\mathcal{S}$ that the horospherical vectors representing its vertices be coplanar in $\mathbb{R}^4$, and that the angle between this plane and the plane determined by each neighboring tile be convex.  Equivalently, if $v_1, \ldots v_k \in V_n$ represent the ideal vertices of an element of $\mathcal{S}$ and $w\in V_n - \{v_1, \ldots v_k\}$ represents a vertex of a neighboring tile, then there exists a vector ${\bf n} \in \mathbb{R}^4$ such that
\begin{enumerate}
\item (coplanarity) ${\bf n} \cdot v_i =1$ for every $i=1, \ldots k$, and
\item (convex angles) ${\bf n} \cdot w > 1$.
\end{enumerate}
(See the proof of \cite[Proposition 6.1]{GHH}.)  Observe that these conditions are invariant under $\text{Isom}(\mathbb{H}^3)$, for if ${\bf n} \cdot v=\alpha$ and $A \in \text{Isom}(\mathbb{H}^3)$ then $({\bf n} A^{-1}) \cdot Av= \alpha$.  

For each member $\calp$ of $\cals$, we note that the subset of $V_n$ representing the set of ideal points of $\calp$ contains $m_1$ and is $\mathrm{Isom}(\calp)$-invariant.  This is because the members of $\cals$ all share the ideal vertex $[m_1]$, and the stabilizer in $G_n$ of any $\calp\in\cals$ acts transitively on its set of ideal vertices.  (The latter assertion can be proved by directly examining $P_n\cap \calp$, for $P_n$ as in Lemma \ref{Gamma_n commensurator}.)  In particular, the ideal vertices of $\calp_1$ are represented in $V_n$ by $\{\sfh(n_i)\}_{i=1}^6$ and those of $\calp_2$ by $\{m_i\}_{i=1}^{12}$, by the properties bulleted above.

Take ${\bf n} = (0,0,0,1/2)^T$.  Then ${\bf n} \cdot m_i =1$ for $i=1, \ldots 12$, so the $m_i$ are coplanar.  The $n_i$ (and hence also the $\sfh(n_i)$) are also coplanar, since $\sqrt{2} {\bf n} \cdot n_i =1$ for $i=1,\ldots, 6$ and the same $\mathbf{n}$.  Coplanarity follows for the other elements of $\cals$, since for example $\{\c^{-1}(m_i)\}_{i=1}^{12}$ is an $\mathrm{Isom}(\c^{-1}(\calp_2))$-invariant collection of horospherical vectors containing $m_1 = \c^{-1}(m_1)$ and representing the ideal vertices of $\c^{-1}(\calp_2)$.

Consider all pairs $(\mathcal{Q}, \calp_X)$ where $X\in \{ M,N\}$ and $\mathcal{Q}$ is a regular ideal octahedron or cuboctahedron which meets $\calp_X$ in a face.  Choose horospherical vectors for $\mathcal{Q}$ to agree with those chosen for $\calp_X$ and to be $\mathrm{Isom}(\mathcal{Q})$-invariant.  Since the convexity condition (2) is invariant under isometries, to finish the proof it suffices to check this condition for each possible pair $(\mathcal{Q}, \calp_X$).

If $\mathcal{Q}$ is a cuboctahedron adjacent to $\calp_M$ sharing the triangular face $(m_1, m_9, m_4)$ then $w=(7, 1, -5 \sqrt{2}, 10)^T$ is a horospherical vector for $\mathcal{Q}$ which is not shared by $\calp_M$.  We have ${\bf n} \cdot w = 5 >1$.  
If $\mathcal{Q}$ is a cuboctahedron adjacent to $\calp_M$ sharing the square face $(m_1, m_2, m_3, m_4)$ then $w=(3, 5, - \sqrt{2}, 6)^T$ is a horospherical vector for $\mathcal{Q}$ which is not shared by $\calp_M$.  We have ${\bf n} \cdot w = 3 >1$.
If $\mathcal{Q}$ is an octahedron adjacent to $\calp_N$ sharing the face $(n_1, n_2, n_3)$ then $w=\sqrt{2} (1, 2, 2, 3)^T$ is a horospherical vector for $\mathcal{Q}$ which is not shared by $\calp_N$.  We have $\sqrt{2} {\bf n} \cdot w = 3>1$.
   By construction, $\calp_1={\sf h}(\calp_N)$ is an octahedron intersecting $\calp_M$ in $(m_1, m_9, m_4)$.  For $w=\sfh(n_1) = (2+2\sqrt{2}, 0, -2-2\sqrt{2}, 4+4\sqrt{2})^T$ we have ${\bf n} \cdot w = 2+\sqrt{2}>1$.

With coplanarity and convex angles thus established, Proposition 6.1 of \cite{GHH} implies that $\Gamma_n.\bigcup\{\calp\in\cals\}$ implies the claim, and hence the result.\endproof

By construction $G_n$ is a subgroup of the symmetry group for $\mathcal{T}_n$.  We complete the proof that $G_n=\mathrm{Comm}(\Gamma_n)$ below, showing that it is the full symmetry group.

\begin{cor}\label{mutants' commensurator}  $G_n$ is the commensurator of $\Gamma_n$ and $O_n$ is the minimal orbifold quotient of $M_n$.  If $I \in \{ 0,1\}^{n+1}$, then $\text{Comm}(\Gamma_I)=G_n$ and $O_n$ is the minimal quotient of $M_I$.\end{cor}

\proof
Proposition \ref{commensurable mutants} implies $\text{Comm}(\Gamma_I)=\text{Comm}(\Gamma_n)$.  Take ${\sf x} \in \text{Comm}(\Gamma_n)$.  We want to show that ${\sf x} \in G_n$.  Recall that $\c^{-n} {\sf r} \c^n \in G_n$ exchanges $\calp_1$ and $\c^{-2n}\r\calp_1$.  Therefore the octahedral tiles of $\mathcal{T}_n$ lie in a single $G_n$-orbit, and we may assume that ${\sf x}$ fixes $\calp_1$.

Recall, for instance from Corollary \ref{M_SandG}, that $\calp_1$ is checkered and its face $A$ spanned by the vertices $0$, $1$, and $\infty$ is external, with $A = \calp_1\cap\calp_2$.  We have that ${\sf a}_0, {\sf b}_0 \in \text{Isom}(\calp_1) \cap G_n$.  The internal faces of $\calp_1$ are paired by elements of $\Gamma_S$, so every internal face of $\calp_1$ meets an octahedron in $\mathcal{T}_n$.  Since $\calp_2$ is a cuboctahedron, ${\sf x}(A)$ must be an external face of $\calp_1$.  

It follows immediately from the definitions of ${\sf a}_0$ and ${\sf b}_0$ that $\langle {\sf a}_0, {\sf b}_0 \rangle$ acts transitively on the external faces of $\calp_1$.  Hence after multiplying by an element of $\langle {\sf a}_0, {\sf b}_0 \rangle < G_n$, we may assume that ${\sf x}(A)=A$.  By construction it is clear that $G_n$ contains the stabilizer of $A$ in $\text{Isom}(\calp_1)$, so we have ${\sf x} \in G_n$ as desired.
\endproof

The second half of Theorem \ref{omnibus1.5} follows 
from the isometry classification below.

\begin{prop}\label{isometric mutants}  Suppose $I=(0,a_1 \ldots, a_{n-1},0)$ and $J=(0,b_1, \ldots , b_{n-1},0)$ are elements of $\{ 0, 1\}^{n+1}$.   $M_I$ is isometric to $M_J$ if and only if $J = I$ or $J = \bar{I}$. 
\end{prop}

\begin{remark}
We have assumed that the first and last entries of $I$ and $J$ are all zero to make the proposition easier to state.  By Lemma \ref{mutation extension}, changing the first or last entry of either $I$ or $J$ to ``1'' yields another isometric manifold. 
\end{remark}

\begin{proof} Any two distinct tiles of $\calt_n$ which meet the interior of the fundamental domain $P_n$ from Lemma \ref{Gamma_n commensurator} have distinct $G_n$-orbits.  On the other hand, any tile that does not, is contained in the orbit of one that does.  It follows that $G_n$ has a unique orbit of octahedral tiles (that of $\calp_1$) and exactly $n$ of cuboctahedral tiles, those of $\calp_2,\c^{-1}(\calp_2),\hdots,\c^{-n+1}(\calp_2)$, since $P_n$ has an open subset in each of these and is contained in their union.

The planes $\c^{-i}(\calh)$ meet the interior of $P_n$ for $i\in\{0,1,\hdots,n-1\}$, so their $G_n$-orbits are also distinct.  We note that the $G_n$-orbit of $\calh$ is distinct from that of $i\calh$ since $\calh$ contains points of the interior of $P_n$ but $i\calh$ contains a face.  Since $i\calh\cap P_n$ is contained in an internal face of $\calp_1$ and $\calh\cap P_n$ in an external face, it follows that the  $G_n$-orbit of an internal face of $\calp_1$ is distinct from that of an external face.

For $I$ as in the hypothesis, it follows as in Lemma \ref{poly decomp} that the members of: 
$$\cals_I=\{\calp_1,\calp_2,\c^{-1}\calp_2,\sfq_2\c^{-2}\calp_2,\hdots,\sfq_n\c^{-2n+1}\calp_2,\sfq_{n}\c^{-2n}\r\calp_1\}$$
project to a polyhedral decomposition of $M_I$, where the $\sfq_i$ are as defined in Proposition \ref{mutated algebraic model}. (In particular, $\sfq_1=1$ and $\sfq_{n+1}=\sfq_n$ since $I$ has first and last entries equal to $0$.)  This is because $\sfq_i(\c^{-2(i-1)}\calp_2\cup\c^{-2i+1}\calp_2)$ projects to a decomposition of $C\left((\Gamma_T^{(i)})^{\sfq_i})\right)$ for each $i$ (see the proof of Lemma \ref{poly decomp}), and $\phi_{\sfm_1}$ preserves the triangulation $\Delta_{\calf}$ of Lemma \ref{M_Sboundary}.  Therefore $\cals_I$ is in bijective correspondence with the set of $\Gamma_I$-orbits of the top-dimensional tiles of $\calt_n$.

Clearly $\sfq_i(\c^{-2(i-1)}\calp_2)$ is $G_n$-equivalent to $\c^{-2(i-1)}\calp_2$ for each $i$ between $1$ and $n$, and $\sfq_i(\c^{-2i+1}\calp_2)$ to $\c^{-2i+1}\calp_2$.  The reflection $\sfu$ through $\calh-n\cdot i\sqrt{2}$, also in $G_n$, exchanges $\calp_1$ with $\c^{-2n}\r\calp_1$ and $\c^{-i}\calp_2$ with $\c^{-2n+i+1}\calp_2$ for each $i$ between $0$ and $2n-1$.  It follows that each $G_n$-orbit of top-dimensional tiles of $\calt_n$ is the union of exactly two $\Gamma_I$-orbits.  

Now suppose for some $J$ as in the hypothesis that there is an isometry $M_J\to M_I$.  This lifts to $\sfx\in\mathrm{Isom}(\mathbb{H}^3)$ with the property that $\Gamma_J^{\sfx}=\Gamma_I$.  Since $\Gamma_I$ and $\Gamma_J$ are each finite-index subgroups of $G_n$ they are commensurable, by definition $\sfx\in\mathrm{Comm}(\Gamma_J) = G_n$.  By the above, $\sfx\calp_1$ is $\Gamma_I$-equivalent to one of $\calp_1$ or $\sfq_{n}\c^{-2n}\r\calp_1$.  

The reflection isometry of Lemma \ref{reflection isometry} determines $\rho\in\mathrm{Isom}(\mathbb{H}^3)$ that conjugates $\Gamma_I$ to $\Gamma_{\bar{I}}$ and takes $\sfq_{n}\c^{-2n}\r\calp_1$ into the $\Gamma_{\bar{I}}$-orbit of $\calp_1$, so replacing $\sfx$ by $\rho\sfx$ (and $I$ by $\bar{I}$) if necessary, we may ensure that there exists $\gamma\in\Gamma_I$ with $\gamma\sfx\calp_1=\calp_1$.   By the above $\gamma\sfx$ takes internal faces of $\calp_1$ to internal faces.  Because it conjugates $\Gamma_J$ to $\Gamma_I$ and $\calp_1$ is contained in a fundamental domain for each, $\gamma\sfx$ preserves the internal face-pairings induced by projections to $M_I$ and $M_J$, respectively. 

It follows from Proposition \ref{mutated algebraic model} that each of these is the pairing described in Lemma \ref{M_SandG}.  The combinatorial description there implies that $\gamma\sfx$ preserves the pairs $\{X_1,X_2\}$ and $\{X_3,X_4\}$ (cf.~Figure \ref{fig:idealoct}), so it is either the identity or $180$-degree rotation in the axis joining the ideal vertex $0$ (the ``intersection'' $X_1\cap X_2$ on the sphere at infinity) to $1+i = X_3\cap X_4$.  However the latter map does not preserve equivalence classes of the ideal vertices of $X_3$ and $X_4$ under face pairing, so $\gamma\sfx=1$.  It follows that $\sfx\in\Gamma_I$, so $\Gamma_J = \Gamma_I$.

We claim however that if $J\neq I$ then $\Gamma_J\neq\Gamma_I$.  The key fact here is that $\Gamma_{T_0}^{\sfm_1} \neq\Gamma_{T_0}$: for instance the face $(\f\g)^{-1}(Y_2')$ of $(\f\g)^{-1}(\calp_2)$ is taken by $(\f\g)^{-1}\h\f\g\in\Gamma_{T_0}$ to $(\f\g)^{-1}(Y_3')$ (see the proof of Lemma \ref{M_T0andH0}), but $(\f\g)^{-1}(Y_2') = \sfm_1\g^{-1}(Y_1')$ is taken by $\sfm_1\g^{-1}\sfm_1^{-1}\in\Gamma_{T_0}^{\sfm_1}$ to $\sfm_1\g^{-1}(Y_3)$.  (This description follows from the fact that $\sfm_1$ preserves the polygon $\calf$ from Lemma \ref{M_Sboundary}, acting on it as a rotation exchanging $\g^{-1}(E)$ with $(\f\g)^{-1}(D)$.)  In fact this further implies that no group $\Gamma$ containing $\Gamma_{T_0}$ also contains $\Gamma_{T_0}^{\sfm_1}$, as long as the natural map $C(\Gamma_{T_0})\to C(\Gamma)$ is embedding.

If $J\neq I$ then for the minimal $i$ such that $b_i\neq a_i$ we have $\Gamma_{T_0}^{\sfw}<\Gamma_I$ and $\left(\Gamma_{T_0}^{\sfm_1}\right)^{\sfw}<\Gamma_J$, where $\sfw = \sfq_i\c^{-2(i-1)}$ (see Proposition \ref{mutated algebraic model}).  The claim, and hence also the result, thus follows from Proposition \ref{mutated algebraic model} and Lemma \ref{M_TandH}.\end{proof}

\section{Incommensurable mutants}\label{sec:in-commensurablo!}

Lemma \ref{mutators} might lead one to suspect that the mutations $(13)(24)$ and $(12)(34)$ of $F^{(0)}$ act very differently at the level of Kleinian groups.  Indeed, it follows from Proposition \ref{nonintegral trace} below, together with Proposition \ref{prop:integraltraces}, that $S^3 - L_n$ is incommensurable with the complement of any link obtained from it by the mutation $(12)(34)$ along a subcollection of the $S^{(i)}$.  In fact, we consider it likely that no two such mutants are commensurable unless they are isometric.

We lack the tools to fully prove this assertion --- mutants are notoriously difficult to distinguish --- but in this section we will describe large families of mutants whose members have different cusp parameters and are mutually incommensurable.  We begin with traces, however.  By \cite{NeR1} the $M_I$ all have trace field $\mathbb{Q}(i,\sqrt{2})$.

\begin{prop}\label{nonintegral trace} For fixed $n$ and any $I = (a_0,\hdots,a_n) \in \{0,1,2\}^{n+1}$ such that $a_i = 2$ for some $i$, $\Gamma_I$ has a nonintegral trace.  \end{prop}

\begin{proof}  Suppose $I = (a_0,\hdots,a_n) \in \{0,1,2\}^{n+1}$ satisfies the hypothesis, and fix $i_0$ with $a_{i_0} = 2$.  By Proposition \ref{mutated algebraic model}, if $i_0 = 1$ then $\Gamma_I$ contains the matrix below.
\[ \t\sfm_2\g\sfm_2^{-1}\ = \ \begin{pmatrix} \frac{1}{5}(-2+12\sqrt{2}+31i+4i\sqrt{2}) & -2+\sqrt{2}+21i+2i\sqrt{2} \\ \frac{1}{5}(-1+7\sqrt{2}+16i+2i\sqrt{2}) & -1+\sqrt{2}+11i+i\sqrt{2} \end{pmatrix} \]
(Recall from Corollary \ref{indexed mutators} that $\sfm_2^{(0)} = \sfm_2$.)  The trace of $\t\sfm_2\g\sfm_2^{-1}$ is not an algebraic integer, since the ring of integers of $\mathbb{Q}(i,\sqrt{2})$ is $\mathbb{Z}[i,\sqrt{2}]$.  If $i_0 = n$ then $\Gamma_I$ contains a conjugate of 
$\bar{\g}\sfm_2\bar{\t}\sfm_2^{-1}=(\sfm_2\bar{g})^{-1}\left(\overline{\t\sfm_2\g\sfm_2^{-1}}\right)\sfm_2\bar{g}$.

In all other cases, Proposition \ref{mutated algebraic model} implies that $\Gamma_I$ contains an element with the same trace as the matrix below.
$$ \bar{\h}(\sfm_2\h\sfm_2^{-1}) = \begin{pmatrix} -2i\sqrt{2} & -3+i\sqrt{2} \\ -3-i\sqrt{2} & 3i\sqrt{2} \end{pmatrix}\begin{pmatrix} -3i\sqrt{2} & 15-5i\sqrt{2} \\ \frac{3+i\sqrt{2}}{5} & 2i\sqrt{2}  \end{pmatrix}, $$
$$   = \begin{pmatrix} -71/5 & -20 -30i\sqrt{2} \\ \frac{18}{5}(-2 + 3i\sqrt{2}) & 55 \end{pmatrix}  $$
The trace of this matrix is evidently not an algebraic integer.  
\end{proof}

For fixed $n$ and any $I \in \{0,1,2\}^{n+1}$, since $M_n$ and $M_I$ decompose along totally geodesic surfaces into isometric pieces, they have the same volume.  (In fact, \cite[Theorem 1.3]{Ru} asserts that hyperbolic volume is always invariant under mutation.)  It would follow from the classical ``Dehn invariant sufficiency" conjecture that any two hyperbolic manifolds with the same volume are scissors congruent (again see \cite{Ne}, for instance).  In our situation we will verify this explicitly.

\begin{prop} \label{mutators Bloch}  For fixed $n$ and any $I \in \{0,1,2\}^{n+1}$, $M_n$ and $M_I$ have the same Bloch invariant.  \end{prop}

\begin{proof}  Recall from Lemma \ref{M_Sboundary} that $F^{(0)}$ inherits a triangulation $\Delta_F$ from the fundamental domain $\mathcal{F}$ for the action of $\Lambda$ on $\mathbf{H}$ pictured in Figure \ref{Sface}.  From the figure, one finds that $\Delta_F$ has six edges, each a geodesic arc joining cusps of $F$.  For example, the geodesic joining $0$ and $\infty$ projects to an edge which joins cusp $1$ to cusp $2$.  Of the other five edges, one joins $3$ to $4$, two join $2$ to $4$, and for each of $2$ and $4$ there is an edge joining it to itself.

Since $\sfm_1 \in \mathrm{PSL}_2(\mathbb{Z})$ it preserves the Farey tesselation of $\calh$, which restricts  on $\mathcal{F}$ to the triangulation pictured in Figure \ref{Sface}.  Therefore $\phi_{\sfm_1}$ preserves $\Delta_F$.  On the other hand, since $\phi_{\sfm_2}$ exchanges $1$ with $2$ and $3$ with $4$ it does not preserve $\Delta_F$.  For instance, if $e$ is the edge joining $2$ to itself then $\phi_{\sfm_2}(e)$ joins $1$ to itself.  

Fix $I = (a_0,\hdots,a_n) \in \{0,1,2\}^{n+1}$ and suppose $a_i=2$ for some $0<i<n$.  The gluing map $C(\Gamma_T^{(i)})\to C(\Gamma_T^{(i+1)})$ factors through $\phi_{\sfm_2^{(i)}}\co F^{(i)}\to F^{(i)}$ by Proposition \ref{mutated geometric model}.  This is conjugate to $\phi_{\sfm_2}$ by the inverse of $\phi_{i+1}$ from Definitions \ref{geometric pieces}, so the gluing does not preserve the triangulations of $F^{(i)}$ induced by its intersections with external faces of the cuboctahedra on either side (cf. Lemma \ref{M_Sboundary}(3)).  The cases $i=0$ and $i=n$ are analogous, and show that if $a_i = 2$ for any $i$ then the division of $M_I$ into octahedra and cuboctahedra is not a true ideal polyhedral decomposition.

It is possible to rectify this by gluing ``flat'' tetrahedra between copies of $C(\Gamma_S)$ and/or $C(\Gamma_T)$ joined by the mutation $\phi_{\sfm_2}$.  If $\calt$ is a flat tetrahedron glued to, say, $C(\Gamma_S)$ along two adjacent triangles in $\partial C(\Gamma_S)$, then $C(\Gamma_S) \cup \calt$ is homeomorphic to $C(\Gamma_S)$, but in the induced triangulation of the boundary, the edge separating the triangles along which $\calt$ is glued has been replaced by an edge joining their two opposite vertices.  For a more thorough exposition, see \cite[\S 4]{NeYang}.

\begin{figure}[ht]
\begin{center}
\input{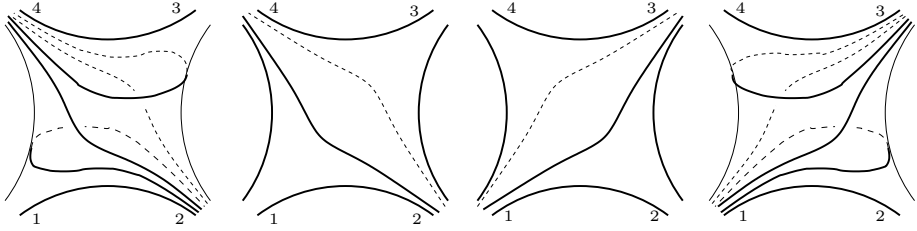}
\end{center}
\caption{Interpolating between $\Delta_F$, on the left, and $\phi_{\sfm_2}(\Delta_F)$.}
\label{fig:interpolate}
\end{figure}

Figure \ref{fig:interpolate} illustrates a process by which $\Delta_F$ may be changed to its image under $\phi_{\sfm_2}$ by a sequence of moves on edges.  The edges of $\Delta_F$ are pictured on the left in bold.  Moving left-to-right, at each stage two edges are replaced by edges transverse to them and disjoint from the remaining edges.  After three such moves, the original triangulation has been changed to its image under $\phi_{\sfm_2}$.

Now suppose $I = (a_1,\hdots,a_n) \in \{0,1,2\}^{n+1}$.  For each $i< n$ such that $a_i =2$, replace $C(\Gamma_T^{(i+1)})$ by its union with $6$ flat tetrahedra, glued successively along $\partial_- C(\Gamma_T^{(i+1)})$ to realize the change of triangulations illustrated in Figure \ref{fig:interpolate}.  The result is homeomorphic to $C(\Gamma_T^{(i+1)})$, since adding a flat tetrahedron does not change the homeomorphism type, but the gluing induced by $\phi_{\sfm_2^{i+1}}$ now preserves the triangulation.  The case $i=n$ is similar, but $C(\overline{\Gamma}_S)$ is changed instead.

It follows from the above that the Bloch invariant $\beta(M_I)$ may be calculated using the resulting polyhedral decomposition.  This differs from the original by the addition of the cross ratio parameters of the flat tetrahedra.  Each of these is equal to $2$, since the triangulation of $F$ is a projection of the Farey tessellation of $\mathbf{H}$.  But in the Bloch group, $2 \cdot [2] = 0$ is a consequence of the relation $[z] = \left[ \frac{z}{z-1} \right]$.  Since the number of flat tetrahedra is a multiple of $6$, the sum of their cross ratio parameters contributes nothing to the Bloch invariant.  \end{proof}

The proposition below tracks the change of cusp parameters under mutation.  To simplify our task, we restrict our attention to complements of links obtained by mutating only with $(12)(34)$ along a subcollection of the $S^{(i)}$ and note in passing that since those obtained by mutating only with $(13)(24)$ are commensurable with $M_n$, their cusp parameters are $\mathrm{PGL}_2(\mathbb{Q})$-equivalent to those of $M_n$.

\begin{prop} \label{M_Imoduli}  For $I = (t_0,t_1,\hdots,t_n) \in \{0,2\}^{n+1}$ and $j \in \{0,1,\hdots,n\}$, define 
$$c_j = \sum_{k = 0}^j \frac{t_k}{2}\ \ (\mathrm{mod}\ 2).$$  
Let $T_1$ be a cross section of the cusp of $M_I$ such that $T_1 \cap C(\Gamma_S) =p_S(A_1)$ (as defined in Lemma \ref{lem:M_S_moduli}), and let $T_2$ be a cross section of the cusp of $M_I$ with $T_2 \cap C(\Gamma_S) = p_S(A_2)$.  Up to the action of $\mathrm{PGL}_2(\mathbb{Q})$, their complex moduli are:  \begin{align*}
 & m(T_1) = i\left[ 1 + 2\sum_{j = 1}^{n} \frac{2\sqrt{2}}{5^{c_{j-1}}} + \frac{1}{5^{c_n}} \right]  \\
 & m(T_2) =  i\left[ \frac{1}{5} + 2\sum_{j = 1}^{n} \frac{2\sqrt{2}}{5^{(1-c_{j-1})}} + \frac{1}{5^{(1- c_n)}} \right]
\end{align*}  \end{prop}

\begin{proof}  To simplify notation, we will identify $A_k$ with $p_S(A_k)$ and view $A_k \subset C(\Gamma_S)$ for $k=1,2$.  Recall the decomposition of $M_I$, along the surfaces $F^{(j)}$, into a union of isometric copies of $C(\Gamma_S)$ and $C(\Gamma_T)$ as described in Proposition \ref{mutated geometric model}:
$$ C(\Gamma_S) \cup C(\Gamma_T^{(1)}) \cup \hdots \cup C(\Gamma_T^{(n)}) \cup C(\overline{\Gamma}_S) \to M_I  $$
We will denote by $l_j$ the gluing map supplied by Proposition \ref{mutated geometric model}, taking $\partial_+ C(\Gamma_T^{(j)})$ to $\partial_- C(\Gamma_T^{(j+1)})$ when $1 \leq j < n$.  The map $l_0$ takes $\partial C(\Gamma_S)$ to $\partial_- C(\Gamma_{T}^{(1)})$, and $l_n \co \partial_+ C(\Gamma_T^{(n)}) \to \partial C(\overline{\Gamma}_S)$.  

For $1\leq j\leq n$ and $k\in\{1,2,3,4\}$ we take $DB^{(j)}_k=\phi_j\circ p_T(DB_k)$ as in the proof of Proposition \ref{M_nmoduli}.  $DB_k$ is defined above Lemma \ref{real M_T moduli}, which implies that $DB^{(j)}_k$ is an annular cross section of the cusp of $C(\Gamma_T^{(j)})$ corresponding to $\p_k^{\c^{-2j}}$.  Each of $T_1$ and $T_2$ meets each of the $C(\Gamma_T^{(j)})$ in a collection of cusp cross sections parallel to a subcollection of the $DB_k^{(j)}$, $k \in \{ 1,2,3,4 \}$.  Similarly, each of $T_1 \cap C(\overline{\Gamma}_S)$ and $T_2 \cap C(\overline{\Gamma}_S)$ is parallel to one of the  cross sections $\overline{A}_1$ or $\overline{A}_2$.

By the proof of Proposition \ref{M_nmoduli}, for  $1\leq j<n$, if $t_j= 0$ then $l_j = \iota_+^{(j)} (\iota_-^{(j)})^{-1}$ takes $\partial_+ DB_k^{(j)}$ to $\partial_- DB_k^{(j+1)}$ for each $k\in\{1,2,3,4\}$.  However if $t_j = 2$ then $l_j$ acts on the indices $k$ by the permutation $(12)(34)$, since it uses $\phi_{\sfm_2}^{(j)}$.  Likewise if $t_0 = 0$ then $l_0(\partial A_k)=\partial_- DB_k^{(1)} \sqcup \partial_- DB_{k+2}^{(1)}$ for $k = 1,2$ by the proof of Proposition \ref{M_nmoduli}; hence if $t_0 = 2$, then $l_0(\partial A_k) = \partial_- DB_{3-k}^{(1)} \sqcup \partial_- DB_{5-k}^{(1)}$.  A similar dichotomy holds for $l_n$.

\begin{remark}  The definitions of the annular cusp cross-sections in Lemmas \ref{lem:M_S_moduli} and \ref{lem:M_T_moduli} depended on a particular collection of horospheres centered at the ideal vertices of $\calp_1$ and $\calp_2$.  These give rise to a particular collection of horospherical cross-sections of the cusps of $F^{(0)}$, which is not preserved by $\phi_{\sfm_2}$.  

It is more accurate to say, for example, that when $t_j=2$ and $1\leq j<n$, $l_j(\partial_+ DB^{(j)}_1)$ is a cusp cross-section of $\partial_- C(\Gamma_T^{(j+1)})$ {\it parallel} (and therefore similar) to $\partial_- DB^{(j+1)}_2$.  Since the modulus is unaffected by similarities, we have largely ignored this distinction above and will continue to do so below.  \end{remark}

\begin{claim}  For each $j \in \{ 1, \ldots , n\}$, 
$$T_1 \cap C(\Gamma_T^{(j)}) = \left\{ \begin{array}{ll} DB_1^{(j)} \cup DB_3^{(j)} & \mbox{if}\  c_{j-1} = 0  \\
  DB_2^{(j)} \cup DB_4^{(j)} & \mbox{if}\  c_{j-1} = 1. \end{array} \right.  $$  
Furthermore, $T_1 \cap C( \overline{\Gamma}_S) = \overline{A}_1$ if $c_n = 0$ and $\overline{A}_2$ if $c_n =1$.  \end{claim}

\begin{proof}[Proof of claim]  This is proved by induction on $j$.  In the base case $j=1$, since $c_0 = t_0/2$ and $T_1 \cap C(\Gamma_S) = A_1$, the conclusion in this case follows directly from the dichotomy in the behavior of $l_0$ recorded above the claim.

Suppose now that the claim holds for some $j<n$, and note that therefore $T_1 \cap M_T^{(j)}$ has components $DB^{(j)}_k$ and $DB^{(j)}_{k'}$, where $k,k' \in \{1,2,3,4\}$ have the same parity, which is opposite that of $c_{j-1}$.  By definition, $c_j$ has the opposite parity from $c_{j-1}$ if and only if $t_j = 2$.  Writing $l_j(\partial_+DB^{(j)}_k) = \partial_-DB^{(j+1)}_{k''}$, the above implies that $k''$ has parity opposite that of $k$ if and only if $t_j = 2$.  A similar assertion holds for $k'$, and the claim follows for $j+1$.

By induction, the claim holds for each $j\leq n$.  The final statement in the claim follows by an argument that mimics the one used in the inductive step.  \end{proof}

The moduli of $A_1$, $A_2$, $\overline{A}_1$, and $\overline{A}_2$ are described in Lemma \ref{lem:M_S_moduli}, and those of the $DB_j^{(i)}$ are described in Lemma \ref{real M_T moduli}.  Using these descriptions and Lemma \ref{lem:annulus_sum}, the claim above shows that the imaginary part of $m(T_1)$ is as described in the statement of the proposition.  The description of the imaginary part of $m(T_2)$ follows similarly.

Now recall the definitions of the arcs $a_1$ and $db_k^{(j)}$ for $1 \leq j \leq n$ and $k = 1,3$, and the collections of arcs $\mathcal{A}_2$ and $D\mathcal{B}_k^{(j)}$ for $1 \leq j \leq n$ and $k = 2,4$, from the proof of Proposition \ref{M_nmoduli}.  For our purposes here, we additionally define $\mathcal{A}_1$ to be a collection of five arcs evenly spaced around $A_1$, each perpendicular to $\partial A_1$ at each of its endpoints, such that $a_1 \in \mathcal{A}_1$.  We analogously define $D\mathcal{B}_k^{(j)}$ as a collection of evenly spaced arcs in $DB_k^{(j)}$ containing $db_k^{(j)}$ for $1 \leq j \leq n$ and $k = 1,3$.

\begin{claim}   If $t_0=0$ then $l_0(\partial \mathcal{A}_k) = \partial_- D\mathcal{B}_{k}^{(1)} \cup \partial_- D\mathcal{B}_{k+2}^{(1)}$ for $k=1,2$, and if $t_0 =2$ then $l_0(\partial \mathcal{A}_k) = \partial_- D\mathcal{B}_{3-k}^{(1)} \cup \partial_- D\mathcal{B}_{5-k}^{(1)}$.  Similarly, for $1 \leq j \leq n-1$,  \begin{align*}  
   l_j(\partial_+ D\mathcal{B}_k^{(j)}) & = \partial_- D\mathcal{B}_{k}^{(j+1)}\  
    \mbox{for}\ k=1,2,3,4,\ \mbox{if}\ t_j = 0, \\
   l_j(\partial_+ D\mathcal{B}_k^{(j)}) & = \left\{  \begin{array}{ll}
    \partial_- D\mathcal{B}_{3-k}^{(j+1)} & \mbox{for}\ k=1,2 \\
    \partial_- D\mathcal{B}_{7-k}^{(j+1)} & \mbox{for}\ k=3,4  \end{array} \right. ,   \mbox{if}\ t_j=2.
\end{align*}
Also, if $t_n=0$ then $l_n^{-1}(\partial \overline{\mathcal{A}}_k) = \partial_+ D\mathcal{B}_{k}^{(n)} \cup \partial_+ D\mathcal{B}_{k+2}^{(n)}$ for $k=1,2$, and if $t_n=2$ then $l_n^{-1}(\partial \overline{\mathcal{A}}_k) = \partial_+ D\mathcal{B}_{3-k}^{(n)} \cup \partial_- D\mathcal{B}_{5-k}^{(n)}$.  \end{claim}

In the discussion above the first claim, we recorded the analogous dichotomy to that of the claim above for the action of the gluing maps $l_j$ on boundaries of annular cusp cross-sections.  The substance of this claim is thus that the gluing maps preserve arc endpoints.

\begin{proof}[Proof of claim.]  Suppose first that $t_j= 0$, so by its definition $l_j = \iota^{(j)}_+(\iota^{(j)}_-)^{-1}$.  The proof of Proposition \ref{M_nmoduli} directly addresses the cases of $\mathcal{A}_2$, $\overline{\mathcal{A}}_2$, and $D\mathcal{B}_k^{(j)}$, where $k = 2$ or $4$.  In the remaining case of $\mathcal{A}_1$, the definition implies that $\partial \mathcal{A}_1$ consists of ten points, five evenly spaced around each component of $\partial A_1$, with each such collection containing a point of $\partial a_1$.  Also by definition, $\partial_- D\mathcal{B}_k^{(1)}$ is a collection of five points spaced evenly around $\partial_- DB_k^{(1)}$, one of which is $\partial_- db_k^{(1)}$ for $k=1,3$.  By the proof of Proposition \ref{M_nmoduli}, $\iota^{(0)}_+(\iota^{(0)}_-)^{-1}$ takes $\partial a_1$ to $\partial_- db_1^{(1)} \cup \partial_- db_3^{(1)}$; hence the entire collection $\partial \mathcal{A}_1$ is taken to $\partial_- D\mathcal{B}_1^{(1)} \cup \partial_- D\mathcal{B}_3^{(1)}$ since $\iota^{(0)}_+(\iota^{(0)}_-)^{-1}$ is an isometry.  The remaining cases when $t_j = 0$, $j \geq 1$, follow similarly.

To illustrate the case $t_j = 2$ we focus on the subcase $1\leq j<n$.  When $t_0=2$, $l_j$ takes $\partial_+ DB_1^{(j)}$ to $\partial_- DB_2^{(j+1)}$, for example.  The crucial observation here is that $l_0(\partial_+ db_1^{(j)})$ is in $\partial_- D\calb_2^{(j+1)}$.  This holds because by definition, $\partial_+ db_1^{(j)}$ is a point in the edge of the triangulation $\Delta_T$ which exits the ideal vertex $1$.  (This is the top edge in Figure \ref{fig:interpolate}.)  Although $\phi_{\sfm_2}$ does not preserve $\Delta_T$, it  preserves this edge, exchanging its endpoints at $1$ and $2$.  Since $\partial_- D\calb_2^{(j+1)}$ has a point in each edge which exits $2$, it contains $\phi_{\sfm_2}(\partial_+ db_1^{(j)})$.  Since the points of $\partial_+ D\calb_1^{(j)}$ are evenly spaced around $\partial_+ DB_1^{(j)}$ and the same is true for $\partial_- D\calb_2^{(j+1)}$, it follows that $l_0(\partial_+ D\calb_1^{(j)}) = \partial_- D\mathcal{B}_{2}^{(j+1)}$.

Since $\phi_{\sfm_2}$ takes the edge of $\Delta_T$ to itself and exchanges its endpoints, $l_0(\partial_+ db_3^{(j)}) \in \partial_- D\calb_4^{(j+1)}$ in this case.  Then it follows from ``even-spacedness'' that $l_0(\partial_+ D\calb_3^{(j)}) = \partial_- D\mathcal{B}_4^{(j+1)}$.  The same argument implies that $\partial_- db_1^{(j+1)} \in l_0(\partial_+ D\calb_2^{(j+1)})$ and therefore that $l_0(\partial_+ D\calb_2^{(j)}) = \partial_- D\mathcal{B}_1^{(j+1)}$, and similarly that $l_0(\partial_+ D\calb_4^{(j)}) = \partial_- D\mathcal{B}_3^{(j+1)}$.  The same sequence of observations, applied to $\partial \mathcal{A}_k$ and $\partial \overline{\mathcal{A}}_k$, $k=1,2$, completes the claim.
\end{proof}

The second claim implies that the set
$$ \mathcal{A}_1\ \cup\ \mathcal{A}_2\ \cup\ \bigcup_{j,k} D\mathcal{B}_k^{(j)}\ \cup\ \overline{\mathcal{A}}_1\ \cup\ \overline{\mathcal{A}}_2  $$
consists of a disjoint union of closed geodesics, some in $T_1$ and some in $T_2$, each meeting any of the geodesics $F^{(j)} \cap T_1$ or $F^{(j)} \cap T_2$ perpendicularly in up to five points.  That $m(T_1)$ and $m(T_2)$ have real part equal to $0$ (up to the action of $\mathrm{PGL}_2(\mathbb{Q})$) now follows as in the proof of Proposition \ref{M_nmoduli}.  \end{proof}

Proposition \ref{M_Imoduli} allows us to describe arbitrarily large subfamilies of the $M_I$ which have $\mathrm{PGL}_2(\mathbb{Q})$-inequivalent cusp parameters and hence are not commensurable.

\begin{cor} \label{different mods}  For $0 \leq k \leq n$, let $I_k = (t_0, t_1,\hdots,t_n)$ be defined by $t_i = 0$ for $i \neq k$, and $t_k = 2$.  The cusp parameters of $M_{I_k}$ are not $\mathrm{PGL}_2(\mathbb{Q})$--equivalent to those of $M_{I_{k'}}$ for $k \neq k'$, when both are less than $(n+1)/2$.  \end{cor}

\begin{proof}  By Proposition \ref{M_Imoduli}, the cusps of $M_{I_k}$ have moduli described below.  \begin{align*}
 &  m(T_1) = i \left[ \frac{6}{5} + \frac{4}{5}\left(n + 4k\right) \sqrt{2} \right] &
 &  m(T_2) = i \left[ \frac{6}{5} + \frac{4}{5}\left(5n - 4k\right) \sqrt{2} \right]  \end{align*}
Since $m(T_1)$ and $m(T_2)$ are both of the form described in Lemma \ref{PGL action} for any $k$, if the cusp parameters of $M_{I_k}$ are equivalent to those of $M_{I_{k'}}$, then one of the two cases below holds.  \begin{align*}  
   n + 4k = n + 4k' & & \mbox{and} & & 5n - 4k = 5n - 4k' \\
   n + 4k = 5n - 4k' && \mbox{and} & & 5n - 4k = n + 4k' \end{align*}
In the first case, $k = k'$, and in the second, $k' = n - k$.  Thus as long as $k$ and $k' < (n+1)/2$ are unequal, their cusp parameters are as well.  \end{proof}

There are also arbitrarily large subfamilies which share cusp parameters, even among complements of links obtained by mutating only with $(12)(34)$.  We do not know if these are commensurable, although we suspect they are not.

\begin{cor} \label{same moduli}  For $0 \leq k < n$, let $I_k = (t_0,\hdots, t_n)$ be defined by $t_i = 0$ for $i \neq k, k+1$, and $t_k = t_{k+1} = 2$.  For each $k$, the cusp parameters of $M_{I_k}$ are \begin{align*}
  m(T_1) = i \left[ 2 + 4\left(n - \frac{4}{5}\right)\sqrt{2} \right] & & m(T_2) = i \left[ \frac{2}{5} + \frac{4}{5}\left( n+4 \right)\sqrt{2} \right], \end{align*}  
up to the action of $\mathrm{PGL}_2(\mathbb{Q})$.  \end{cor}

Corollaries \ref{different mods} and \ref{same moduli} prove parts (2) and (3), respectively, of Theorem \ref{omnibus2}.

\appendix

\section{Proof of Lemma \ref{convexcore}}  \label{appendix:convexcore}

Following Morgan \cite{Mo}, we define a \textit{pared manifold} to be a pair $(M,P)$, where $M$ is a compact, orientable, irreducible 3-manifold with nonempty boundary which is not a 3-ball, and $P \subseteq \partial M$ is the union of a collection of disjoint incompressible annuli and tori satisfying the following properties: \begin{itemize}

\item  Every noncyclic abelian subgroup of $\pi_1 M$ is conjugate into the fundamental group of a component of $P$.

\item  Every map $\phi\co (S^1\times I, S^1\times \partial I) \rightarrow (M,P)$ which induces an injection on fundamental groups is homotopic as a map of pairs to a map $\psi$ such that $\psi(S^1\times I) \subset P$.

\end{itemize}

This definition describes the topology of the compact manifold obtained by truncating the cusps of the convex core of a geometrically finite hyperbolic 3-manifold by open horoball neighborhoods.  Indeed, Corollary 6.10 of \cite{Mo} asserts that if $(M,P)$ is obtained in this way, where $P$ consists of the collection of boundaries of the truncating horoball neighborhoods, then $(M,P)$ is a pared manifold.

Lemma \ref{convexcore} from the body of this paper asserts that if $(M,P)$ has the pared homotopy type of a geometrically finite hyperbolic manifold $\mathbb{H}^3/\Gamma$ where $\Gamma$ is not Fuchsian and $\partial C(\Gamma)$ is totally geodesic, then $M-P$ is homeomorphic to $C(\Gamma)$.  The key point of the proof is that the geometric conditions on $\Gamma$ ensure that $(M,P)$ is an acylindrical pared manifold.  Then Johannson's Theorem \cite{Jo}, that pared homotopy equivalences between acylindrical pared manifolds are homotopic to pared homeomorphisms, applies.  We expand on this below.

It is worth noting that Lemma \ref{convexcore} fails in more general circumstances.  Canary-McCullough give examples of this phenomenon in \cite{CM}, where for instance they describe homotopy equivalent non-Fuchsian geometrically finite manifolds with incompressible convex core boundary which are not homeomorphic (Example 1.4.5).  Their memoir \cite{CM} is devoted to understanding the ways in which homotopy equivalences of hyperbolic 3-manifolds can fail to be homotopic to homeomorphisms, and Lemma \ref{convexcore} follows quickly from results therein.

The treatment of Canary-McCullough itself uses the theory of {\it characteristic submanifolds} of manifolds with \textit{boundary pattern} developed in \cite{Jo}.  The characteristic submanifold of a manifold with boundary pattern is a maximal collection of disjoint codimension--zero submanifolds, each an interval bundle or Seifert--fibered space embedded reasonably with respect to the boundary pattern.  Rather than attempting to establish all of the notation necessary to define this formally, we refer the interested reader to \cite{Jo} and \cite{CM}.  Here we simply transcribe the relevant theorem of \cite{CM}, which strongly restricts the topology of the characteristic submanifold of a pared manifold with boundary pattern determined by the pared locus.

For the purposes of Lemma \ref{convexcore} we exclude from consideration certain pared manifolds which never arise from convex cores of geometrically finite hyperbolic 3-manifolds.  We say $(M,P)$ is \textit{elementary} if it is homeomorphic to one of $(T^2\times I, T^2\times\{0\})$, $(A^2 \times I,A^2 \times \{0\})$, or $(A^2 \times I ,\emptyset)$, where $T^2$ and $A^2$ denote the torus and annulus, respectively;  otherwise $(M,P)$ is nonelementary.  Define $\partial_0 M := \overline{M-P}$.  We say an annulus properly embedded in $M-P$ is \textit{essential} in $(M,P)$ if it is incompressible and boundary--incompressible in $M-P$.  For a codimension--0 submanifold $V$ embedded in $M$, we denote by $\mathrm{Fr}(V)$ the \textit{frontier} of $V$ (that is, its topological boundary in $M$), and note that $\mathrm{Fr}(V) = \overline{\partial V - (V \cap \partial M)}$.  With notation thus established, the following theorem combines the definition of the characteristic submanifold with Theorem 5.3.4 of \cite{CM}.

\begin{thm1} Let $(M,P)$ be a nonelementary pared manifold with $\partial_0 M$ incompressible.   Select the fibering of the characteristic submanifold so that no component is an $I$--bundle over an annulus or M\"obius band.  \begin{enumerate}

\item  Suppose $V$ is a component of the characteristic submanifold which is an $I$--bundle over a surface $B$.  Then each component of the associated $\partial I$--bundle is contained in $\partial_0 M$, each component of the associated $I$--bundle over $\partial B$ is either a component of $P$ or a properly embedded essential annulus, and $B$ has negative Euler characteristic.

\item  A Seifert fibered component $V$ of the characteristic submanifold is homeomorphic either to $T^2 \times I$ or to a solid torus.  If $V$ is $T^2 \times I$ then one component of $T^2\times \partial I$ lies in $P$ and the other components of $V \cap \partial M$ are annuli in $\partial_0 M$.  If $V$ is a solid torus, then $V \cap \partial M$ has at least one component, each an annulus either containing a component of $P$ or contained in $\partial_0 M$.  In either case, each component of the frontier $\mathrm{Fr}(V)$ of $V$ in $M$ is a properly embedded essential annulus.  \end{enumerate} 

The characteristic submanifold contains regular neighborhoods of all components of $P$. \label{CM} \end{thm1}

The key claim in the proof of Lemma \ref{convexcore} is a further restriction on the characteristic submanifold of $(M,P)$, in the case that $M$ is obtained from the convex core of a non-Fuchsian geometrically finite manifold with totally geodesic convex core boundary by removing horoball neighborhoods of the cusps.  $P$ is the union of the boundaries of these neighborhoods.

\begin{claim}  $(M,P)$ as above is nonelementary, and $\partial_0 M$ is incompressible.  The characteristic submanifold of $(M, P)$ consists only of (Seifert fibered) regular neighborhoods of the components of $P$, each of whose boundary has a unique component of intersection with $\partial M$.  
\end{claim}

We prove the claim below, but assuming it for now, the proof of Lemma 3 proceeds as follows.  A representation as given in the statement of the lemma induces a pared homotopy equivalence between $(M,P)$ and the pared manifold $(N,Q)$ obtained by truncating $C(\Gamma)$ with open horoball neighborhoods.  Since $C(\Gamma)$ has totally geodesic convex core boundary, $(N,Q)$ is as described by the claim; hence $(M,P)$ is as well (see Theorem 2.11.1 of \cite{CM}, for example).  Johansson's Classification Theorem (cf. \cite{CM}, Theorem 2.9.10) implies that the original pared homotopy equivalence is homotopic to one which maps the complement of the characteristic submanifold of $(M,P)$ homeomorphically to the complement of the characteristic submanifold of $(N,Q)$.  It follows from the claim that these are homeomorphic to $M-P$ and $N-Q$, respectively, and the lemma follows.

\begin{proof}[Proof of claim]  As was mentioned above, the elementary pared manifolds do not arise from geometrically finite hyperbolic manifolds.  Since $(M, P)$ is obtained from the convex core of a geometrically finite manifold with totally geodesic convex core boundary, the following are known not to occur: \begin{enumerate}

\item  A compressing disk for $\partial_0 M$.  

(By definition $\partial_0 M$ lifts to a geodesic hyperplane in $\mathbb{H}^3$, hence the induced map $\pi_1 \partial M_0 \rightarrow \pi_1 M$ is injective.) 

\item  An \textit{accidental parabolic}: an essential annulus properly embedded in $M$ with one boundary component in $P$ and one in $\partial_0 M$, which is not parallel to $P$.  

(Every essential curve on $\partial_0 M$ that is not boundary-parallel is homotopic to a geodesic, but an element of $\pi_1(M)$ corresponding to an accidental parabolic has translation length $0$.)

\item  A \textit{cylinder}; that is, a properly embedded essential annulus in $M-P$.  

(The double $DM$ of $M$ across $\partial_0 M$ is a hyperbolic manifold, but the double of a cylinder in $M$ would be an essential torus in $DM$.)  \end{enumerate}

We show that if the characteristic manifold has any components other than those listed in the claim then at least one of the above facts cannot hold.

For a component $V$ of the characteristic submanifold which is an $I$--bundle over a surface $B$, at least one component of the associated $I$--bundle over $\partial B$ must be properly embedded, since otherwise we would have $M= V$ and it is well known that an $I$--bundle over a surface does not admit a hyperbolic structure with totally geodesic convex core boundary unless the convex core is a Fuchsian surface.  But this annulus violates fact 2 or 3.  Thus there are no $I$--bundle components of the characteristic submanifold.

If $V$ is a Seifert fibered component of the characteristic submanifold homeomorphic to $T^2 \times I$, then one component of $\partial V$ is a torus $P_1 \subset P$, and all other components of $\partial V \cap \partial M$ are annulli in $\bound_0 M$.  If this second class is nonempty then each component of $\mathrm{Fr}(V)$ is an essential annulus properly embedded in 
$M-P$, contradicting fact 3.  Thus $\partial V \cap \partial M=P_1$ and $V$ is a regular neighborhood of $P_1$.

If $V$ is a solid torus and $V \cap \partial M$ contains a component of $P$, then a similar argument shows that this is the unique component of $\partial V \cap \partial M$, so in this case $V$ is a regular neighborhood of an annular component of $P$.  If on the other hand $V \cap \partial M$ does \textit{not} contain any components of $P$, then it has at least two components, for otherwise a meridional disk of $V$ determines a boundary compression of the annulus $\mathrm{Fr}(V)$ in $M-P$.  But then any component of $\mathrm{Fr}(V)$ violates fact 3.  \end{proof}

\bibliographystyle{plain}

\bibliography{commensurability}


\end{document}